\documentclass[a4paper]{amsart}
\usepackage[active]{srcltx}
\usepackage[all]{xy}
\usepackage{subfig}
\usepackage{hyperref}
\usepackage{mathrsfs}

\usepackage{amsmath}
\usepackage{amssymb,latexsym}
\usepackage{mathrsfs}
\usepackage{graphics}
\usepackage{latexsym}
\usepackage{psfrag}
\usepackage{import}
\usepackage{verbatim}
\usepackage{graphicx}
\usepackage[usenames]{color}
\usepackage{pifont,marvosym}

\newcounter{assu}

\theoremstyle{plain}
\newtheorem{lemma}{Lemma}[section]
\newtheorem{theorem}[lemma]{Theorem}
\newtheorem{proposition}[lemma]{Proposition}
\newtheorem{corollary}[lemma]{Corollary}

\theoremstyle{definition}
\newtheorem{assumption}[assu]{Assumption}
\newtheorem{definition}[lemma]{Definition}
\newtheorem{remark}[lemma]{Remark}

\numberwithin{equation}{section}

\newcommand{\supp}{\text{\rm supp}}

\newcommand{\ve}{\varepsilon}

\newcommand{\erre}{\mathbb{R}}
\newcommand{\enne}{\mathbb{N}}

\newcommand{\f}{\varphi}
\newcommand{\gammA}{\boldsymbol{\gamma}}

\newcommand{\G}{\mathcal{G}}
\renewcommand{\S}{\mathcal{S}}

\renewcommand{\r}{\varrho}

\begin{document}
\title[Decomposition and globalization]{Decomposition of geodesics in the Wasserstein space \\ and the globalization problem}

\author{Fabio Cavalletti}

\address{RWTH, Department of Mathematics, Templergraben  64, D-52062 Aachen (Germany)}
\email{cavalletti@instmath.rwth-aachen.de}

\begin{abstract}
We will prove a decomposition for Wasserstein geodesics in the following sense:
let $(X,d,m)$ be a non-branching metric measure space verifying 
$\mathsf{CD}_{loc}(K,N)$ or equivalently $\mathsf{CD}^{*}(K,N)$. 
We prove that every geodesic $\mu_{t}$ in the $L^{2}$-Wasserstein space, with $\mu_{t} \ll m$, is  decomposable  
as the product of two densities, one corresponding to a geodesic with support of codimension one verifying 
$\mathsf{CD}^{*}(K,N-1)$, and the other associated with a precise one dimensional measure, 
provided the length map enjoys local Lipschitz regularity.
The motivation for our decomposition is in the use of the component evolving like $\mathsf{CD}^{*}$ in the globalization problem.

For a particular class of optimal transportation we prove  the linearity in time of the other component, obtaining therefore the 
global $\mathsf{CD}(K,N)$ for $\mu_{t}$. The result can be therefore interpret as a 
 globalization theorem for $\mathsf{CD}(K,N)$ for this class of optimal transportation,  
 or as a ``self-improving property'' for $\mathsf{CD}^{*}(K,N)$.

Assuming more regularity, namely in the setting of infinitesimally strictly convex metric measure space, the one dimensional density 
is the product of two differentials giving more insight on the density decomposition.
\end{abstract}

\maketitle

\tableofcontents

\bibliographystyle{plain}

\section{Introduction}

The class of metric measure spaces with generalized lower bounds on the Ricci curvature formulated in terms of optimal transportation,
has been introduced by Sturm in \cite{sturm:MGH1, sturm:MGH2} and 
independently by Lott and Villani in \cite{villott:curv}.  
The spaces belonging to this class are called $\mathsf{CD}(K,N)$-spaces and the condition characterizing them is 
denoted with $\mathsf{CD}(K,N)$.

In the \emph{curvature-dimension condition} $\mathsf{CD}(K,N)$ the two parameters $K$ and $N$ play the role
of a curvature lower bound and a dimension upper bound, respectively. 
Among the many relevant properties enjoyed by $\mathsf{CD}(K,N)$, the following one also serves as a motivation: 
 a complete Riemannian manifold satisfies $\mathsf{CD}(K,N)$ if and only if its Ricci curvature is bounded from below by $K$ and 
its dimension from above by $N$.

Roughly speaking curvature-dimension condition $\mathsf{CD}(K,N)$ prescribes  how the volume of a given set is affected by curvature 
when it is moved via optimal transportation. It imposes that the distortion is ruled by a coefficient denoted by $\tau_{K,N}^{(t)}(\theta)$
depending on the curvature $K$, on the dimension $N$, on the time of the evolution $t$ and on the point length $\theta$. 
The main feature of $\tau_{K,N}^{(t)}(\theta)$ is that it is obtained mixing two different volume distortions:
an $(N-1)$-dimensional distortion depending on the curvature $K$ and a one dimensional evolution that doesn't contain any curvature information. Namely 
\[
\tau_{K,N}^{(t)}(\theta) = t^{1/N} \sigma_{K,N-1}^{(t)}(\theta)^{(N-1)/N},
\] 
where $\sigma_{K,N-1}^{(t)}(\theta)^{(N-1)/N}$ contains the information on the $(N-1)$-dimensional volume distortion and 
the evolution in the remaining direction is ruled just by $t^{1/N}$. The coefficient $\sigma_{K,N}^{(t)}(\theta)$
is the solution (in time) of the second order differential equation 
\[
y'' + \theta^{2} \frac{K}{N} y = 0, \qquad y(0) = 0, \quad y'(0)=1.
\]
The previous equation appears naturally in the study of the Jacobian of the differential of the exponential map in the context of 
differential geometry, and indeed it rules the part of the Jacobian associated to the restriction to an hyperplane of the differential of the exponential map, see \cite{sturm:MGH2} for more details.

A broad variety of geometric and functional analytic properties can be deduced from the curvature-dimension condition $\mathsf{CD}(K,N)$: 
the Brunn-Minkowski inequality, the Bishop-Gromov volume comparison theorem, the Bonnet-Myers theorem, the doubling property and local Poincar\'e inequalities on balls. All these listed results are in a quantitative form (volume of intermediate points, volume growth, upper bound 
on the diameter and so on) depending on $K,N$.

One of the most important questions on $\mathsf{CD}(K,N)$ that are still open, and we will try to understand in this note, is whether this notion enjoys a globalization property:
can we say that a metric measure space $(X,d,m)$ satisfies $\mathsf{CD}(K,N)$ provided $\mathsf{CD}(K,N)$ 
holds true locally on a family of sets $X_{i}$ covering $X$?

A first tentative of answer this problem was given by Bacher and Sturm in \cite{sturm:loc}: they proved that 
a non-branching metric measure space $(X,d,m)$ verifies
the local curvature-dimension condition $\mathsf{CD}_{loc}(K,N)$ if and only if it verifies the global reduced curvature-dimension condition 
$\mathsf{CD}^{*}(K,N)$. The latter is obtained from $\mathsf{CD}(K,N)$ imposing that the volume distortion, 
during the evolution through an optimal transportation, is ruled by $\sigma_{K,N}^{(t)}(\theta)$ instead of $\tau_{K,N}^{(t)}(\theta)$.
The $\mathsf{CD}^{*}(K,N)$ is a priori weaker than $\mathsf{CD}(K,N)$ and the converse comparison can be obtained 
only changing the value of the lower bound on the curvature: condition $\mathsf{CD}^{*}(K,N)$ implies $\mathsf{CD}(K^{*},N)$ where $K^{*}= K(N-1)/N$ (for $K\geq 0$ and for $K< 0$ a suitable formula holds). Therefore the curvature condition contained in $\mathsf{CD}^{*}(K,N)$ 
is a priori weaker  than the one contained $\mathsf{CD}(K,N)$.

Roughly speaking, the main reason why the globalization property holds for the reduced curvature-dimension condition 
stays in the good behavior (in time) of $\sigma_{K,N}^{(t)}(\theta)$, which in turn can be led back to the previous second order differential equation. The same approach applied to $\mathsf{CD}(K,N)$, that is try to prove the globalization property for $\mathsf{CD}(K,N)$ directly from the properties of $\tau_{K,N}^{(t)}(\theta)$, seems to not work.

A different approach to the problem has been presented by the author together with Sturm in \cite{cavasturm:MCP}. 
The approach in \cite{cavasturm:MCP} was, in the case of an optimal transportation between a diffuse measure and a Dirac delta,  
to isolate a local $(N-1)$-dimensional evolution ruled by $\sigma_{K,N-1}^{(t)}(\theta)$ 
and then using the nice properties of $\sigma_{K,N-1}^{(t)}(\theta)$, obtain a global $(N-1)$-dimensional evolution 
ruled by the coefficient $\sigma_{K,N-1}^{(t)}(\theta)$. Then using H\"older inequality and the linear behavior of the other direction, 
pass from the $(N-1)$-dimensional version to the full-dimensional version with coefficient $\tau_{K,N}^{(t)}(\theta)$.

So the strategy was to reproduce in the setting of metric measure spaces the calculations done in the Riemannian framework where, taking advantage of parallel transport,  from $Ric \geq K$ it is possible to split the Jacobian determinant of the differential of the exponential map into two components: one of codimension $1$ evolving accordingly to $\sigma_{K,N-1}$ and one representing the distortion in the direction of motion that is concave.

To be more precise in \cite{cavasturm:MCP} it is proved that if $(X,d,m)$ is a non-branching metric measure space that verifies $\mathsf{CD}_{loc}(K,N)$ then it verifies the weaker $\mathsf{MCP}(K,N)$.
While $\mathsf{CD}(K,N)$ is a condition on the optimal transport between  any pair of absolutely continuous (w.r.t. $m$) probability measure on $X$, 
$\mathsf{MCP}(K,N)$ is a condition on the optimal transport between a Dirac delta and the uniform distribution $m$ on $X$.
Indeed to detect the $(N-1)$-dimensional evolution it is necessary to decompose the whole evolution. 
Considering the optimal transport between a Dirac mass in $o$ and the uniform distribution $m$ 
permits to immediately understand that  the family of spheres around $o$ provides the 
correct $(N-1)$-dimensional support of the evolving measures. 
So the choice of a Dirac delta as second marginal was really crucial and strongly influenced the geometry of the optimal transportation.

The aim of this paper is to identify, in the general case of optimal transportation between any measures, 
the $(N-1)$-dimensional evolution verifying $\mathsf{CD}^{*}(K,N-1)$ and, starting from that, 
provide a decomposition for densities of geodesics that can be interpret as a parallel transport.    
The $N$-dimensional density will be written as the product between the $(N-1)$-dimensional density verifying $\mathsf{CD}^{*}(K,N-1)$ 
and of a 1-dimensional density not necessarily associated to a 1-dimensional geodesic.
In the framework of infinitesimally strictly convex spaces, the 1-dimensional density  will be obtained as the product of two differential, producing then a more direct decomposition.  

We will construct a full decomposition for any optimal transportation verifying a local Lipschitz regularity, see
Assumption \ref{A:lengthreg} and \ref{A:Phi} for the precise hypothesis.
We apply this decomposition to the globalization problem for $\mathsf{CD}(K,N)$.
With this approach we are able to reduce the problem to prove the concavity in time of the 1-dimensional density, 
provided Assumption \ref{A:lengthreg} and Assumption \ref{A:Phi} are verified. It is important to underline here that in the framework of 
Riemannian manifolds endowed with volume measures both Assumption \ref{A:lengthreg} and \ref{A:Phi} are proved to hold.

Moreover in the particular case of optimal transport plans giving the same speed to geodesics leaving from the same level set of the associated Kantorovich potential, we prove indeed both regularity and linearity of the 1-dimensional factor and we get the full $\mathsf{CD}(K,N)$ inequality. 
So we prove the global estimate of $\mathsf{CD}(K,N)$ for a certain class of optimal transportation, 
clearly including all the cases treated in \cite{cavasturm:MCP}.

We now present the paper in more details. \\
Let $(X,d,m)$ be a non-branching metric measure space verifying the local curvature dimension condition 
and $\mu_{t} = \r_{t} m$ be a geodesic (in the $L^{2}$-Wasserstein space) that we want to decompose as stated before.
The first difficulty we have to handle with is to find a suitable partition of the space. 
Unlikely optimal transportations connecting measures to deltas, 
there is not just a universal family of sets but one for each $t \in [0,1]$: if  
$\f$ is a  Kantorovich potential associated to $(\mu_{0},\mu_{1})$, then
\[
\{\gamma_{t} : \f (\gamma_{0}) = a, \gamma \in \supp(\gammA) \}_{a \in \erre}
\]
is the family of partitions, one for each $t \in [0,1]$, that will be considered. 
Here $\gammA \in \mathcal{P}(\G(X))$ is a dynamical optimal transference plan of $\mu_{t}$ and $\mathcal{P}(\G(X))$ denotes the space of probability measures over $\G(X)$, 
the space of geodesic in $X$ endowed with the uniform topology inherited as a subset of $C([0,1],X)$.

The intuitive reason suggesting that the previous family is the right one, stays in the Brenier-McCann Theorem for optimal transportation on manifold that gives a precise formula for the optimal maps: 
\[
T_{t} (x) = \exp_{x}(-t \nabla \f (x)), \qquad (T_{t})_{\sharp}\mu_{0} = \mu_{t}.
\]
By definition, geodesics on manifold verify $\nabla_{\dot \gamma} \dot\gamma = 0$, where $\nabla$ only here 
denotes the Levi-Civita connection, meaning that there is no curvature in the direction of $\gamma$. 
Hence the direction orthogonal to the motion should be the one carrying all the curvature information. 
Since $\dot \gamma_{0} = - \nabla \f$, (here $\nabla \f$ is the gradient of $\f$) the family of sets
orthogonal to the motion are the level sets of $\f$.

On the rigorous mathematical side, the reason why that family is the right one stays in the following property: the set 
\[
\{ (\gamma_{0},\gamma_{1}) \in X \times X : \f(\gamma_{0})=a \}
\]
is $d$-cyclically monotone (Proposition \ref{P:monotone}). Hence for $\gamma \neq \hat \gamma \in \supp(\gammA)$ 
with $\f(\gamma_{0}) = \f(\gamma_{1})$ it holds 
\[
\gamma_{s} \neq \gamma_{t}, \quad \forall s,t \in (0,1).
\]
Therefore for $s \neq t$, $\{\gamma_{s} : \f (\gamma_{0}) = a \}$ and $\{\gamma_{t} : \f (\gamma_{0}) = a \}$ are disjoint. 
This key property permits to consider the evolution of each ``slice'' of the geodesic $\mu_{t}$, 
where with ``slice'' we mean its conditional measure with respect to the level sets of the chosen Kantorovich potential.  

Here the structure is very rich. Using this new property of $d$-cyclical monotonicity, 
it is possible to construct  $L^{2}$-Wasserstein geodesics with also $d$-monotone support.
The whole construction does not rely on any curvature bound of the space and its interest goes beyond the scope of this paper.
For this reason we commit Section \ref{S:dmonotone} to the presentation of these results in their fully generality.

As it is well known, any $d$-monotone set is formed by family of geodesics that do not intersect at any time. 
For this reason a translation along this geodesics is well defined.
Denote by $\phi_{a}$ a Kantorovich potential associated to the $d$-monotone set
$\{(\gamma_{s},\gamma_{t}) : \gamma \in G_{a}, s\leq t \in [0,1] \}$.
The crucial idea to construct $L^{2}$-geodesics is to move via ``translation'' 
level sets of $\phi_{a}$ to level sets of $\phi_{a}$.
As proved in Lemma \ref{L:12monotone} and Proposition \ref{P:w2geo} 
this will produce a geodesic in the $L^{2}$-Wasserstein space, showing a new connection between $L^{1}$ 
and $L^{2}$ optimal transportation problems.

The relevance of this construction for the globalization problem stays in the following property:  the family of 
geodesics obtained in Section \ref{S:dmonotone}
have a linear structure on each geodesic forming the $d$-monotone set. 
Therefore there is one degree of freedom to play with. This property, that was already present 
in \cite{cavasturm:MCP} but somehow hidden, will 
be fundamental here to improve the curvature estimates for the element of codimension one passing from 
$N$ to $N-1$.
\\

Coming back the the decomposition, if we want to 
perform a dimensional reduction argument on measures the right tool is Disintegration Theorem (Theorem \ref{T:disintr}):
(Proposition \ref{P:quotient})
\[
\gammA = \int_{\f(\mu_{0})} \gammA_{a} \mathcal{L}^{1}(da),  \qquad \gammA_{a} \in \mathcal{P}(G), \quad \gammA_{a}(\{\gamma \in G : \f(\gamma_{0}) = a \}) = \|\gammA_{a} \|,
\]
where $\f(\mu_{0}) = \f(\supp(\mu_{0}))$ and $G$ is the support of $\gammA$.
Since 
\[
\mu_{t} = (e_{t})_{\sharp} \gammA = \int_{\f(\mu_{0})} (e_{t})_{\sharp} \gammA_{a} \mathcal{L}^{1}(da),  
\]
the geodesics of codimension one that should verify curvature estimates like $\mathsf{CD}^{*}(K,N-1)$ is $t\mapsto (e_{t})_{\sharp}(\gammA_{a})$, for all $a \in \f(\mu_{0})$. 
Since curvature properties in metric measure spaces are formulated in terms of a reference measure
and $(e_{t})_{\sharp}(\gammA_{a})$ is singular with respect to $m$, 
it is not obvious which reference measures of codimension one we have to choose. 
One option could be to consider for each $t \in [0,1]$, the family
\[
\{ \gamma_{t} : \f(\gamma_{0}) =a, \gamma \in G  \}_{a \in \f(\mu_{0})}.
\]
Then for each $t \in [0,1]$, by $d^{2}$-cyclical monotonicity, the family is a partition of $e_{t}(G)$ and 
hence we have (Proposition \ref{P:quotient} and Lemma \ref{L:cod1})
\[
m\llcorner_{e_{t}(G)} = \int_{\f(\mu_{0})} \hat m_{a,t} \mathcal{L}^{1}(da),\qquad \hat m_{a,t} (\{ \gamma_{t} : \f(\gamma_{0})=a \}) = \| \hat m_{a,t} \|.
\]
But the $(N-1)$-dimensional measures $\hat m_{a,t}$ are not the right reference measures to prove $\mathsf{CD}^{*}(K,N-1)$ 
estimate for the densities of $(e_{t})_{\sharp}\gammA_{a}$. Indeed if $(e_{t})_{\sharp} \gammA = \mu_{t} = \r_{t} m$, then, 
\[
\int \r_{t} \hat m_{a,t} \mathcal{L}^{1}(da) = \r_{t} m\llcorner_{e_{t}(G)} = \mu_{t} = \int (e_{t})_{\sharp}\gammA_{a} \mathcal{L}^{1}(da)
\]
and by uniqueness of disintegration $(e_{t})_{\sharp}\gammA_{a} = \r_{t} \hat m_{a,t}$
and therefore the density is $\r_{t}$ and no gain in dimension is possible. 

The correct reference measures are built as follows. For each $a \in \f(\mu_{0})$, consider the following family of sets
\[
\{ \gamma_{t} : \f(\gamma_{0})=a, \gamma \in G\}_{t \in [0,1]}, 
\]
that is for a fixed $a$ we take all the evolutions for $t\in [0,1]$ of the level set $a$ of $\f$.

By $d$-cyclical monotonicity, they are disjoint (Lemma \ref{L:partition}). 
If $\bar \Gamma_{a}(1): =\cup_{t \in [0,1]} \{ \gamma_{t} : \f(\gamma_{0})=a, \gamma \in G\}$, then  (Proposition \ref{P:absolutecont})
\[
m\llcorner_{\bar \Gamma_{a}(1)} = \int_{[0,1]} m_{a,t} \mathcal{L}^{1}(dt), \qquad m_{a,t}(\{ \gamma_{t}: \f(\gamma_{0} )=a \}) = \| m_{a,t}\|.
\]
Since in the disintegration above the quotient measure is supported on $[0,1]$, that is the range of the time variable, 
$m_{a,t}$ should be interpret as the conditional measure moving (with $t$) in the same direction of the optimal transportation.

In order to apply the results of Section \ref{S:dmonotone}
to get an improvement of curvature estimates, we have to show that $(e_{t})_{\sharp} \gammA_{a} \ll m_{a,t}$.
After having that, to get the improvement one could use the ``linear'' structure of geodesics of Section \ref{S:dmonotone} 
together with the curvature bound estimate they have to satisfy because of $(e_{t})_{\sharp} \gammA_{a} \ll m_{a,t}$.

So suppose that we have already proved $(e_{t})_{\sharp}(\gammA_{a}) = h_{a,t} m_{a,t}$
and $t \mapsto h_{a,t}(\gamma_{t})$ satisfies the 
local (and hence the global) reduced curvature-dimension condition $\mathsf{CD}_{loc}^{*}(K,N-1)$. Then the situation would be
\[
h_{a,t} m_{a,t} = e_{t\,\sharp}\gammA_{a} = \r_{t} \hat m_{a,t}.
\]
Our final scope is to prove properties on $\r_{t}$, and to translate information on $h_{a.t}$ into information on $\r_{t}$ is necessary to put in relation the two different reference measures of codimension one $m_{a,t}$ and $\hat m_{a,t}$.

Actually the path we will adopt in the note will be the other way round. First we will show that $\lambda_{t} m_{a,t} = \hat  m_{a,t}$ 
for some function $\lambda_{t}$ defined on $e_{t}(G)$ 
and then from that we deduce that $(e_{t})_{\sharp}(\gammA_{a})$ can be written as $h_{a,t} m_{a,t}$.
After that we will prove $\mathsf{CD}^{*}(K,N-1)$ for $h_{a,t}$.
We will obtain a decomposition of the following type
\[
\r_{t} = \frac{1}{\lambda_{t}} h_{a,t}
\]
and therefore to prove curvature estimate for $\r_{t}$ also information on $\lambda_{t}$ are needed.

We have additional properties of $\lambda_{t}$, that will permit to prove the full $\mathsf{CD}(K,N)$ estimate for $\r_{t}$, 
in the particular case of optimal transportation giving constant speed to geodesics leaving from the same level sets 
and not inverting the level sets of $\f$ during the evolution, that is 
\[
L(\gamma) = f(\f(\gamma_{0})), \quad \gammA-a.e. \ \gamma \in \G(X),
\]
with $f : \f(\supp[\mu_{0}]) \to \erre$ such that $a \mapsto a - f^{2}/2$ is a non increasing function of $a$. 
This condition permits to say, see Lemma \ref{L:levelset}, that a level set of $\f$ after time $t$ is moved to a level set of $\f_{t}$
and this produce a simplification on the geometry of the optimal transportation.
Indeed under this assumption, the map $t \mapsto \lambda_{t}(\gamma_{t})$ is linear.

Due to the relevance of this family of optimal transportations and to better explain why $\lambda_{t}$ is linear, 
we will first present  part of the decomposition procedure in Section \ref{S:particular} under this additional assumption on the length of geodesics.  
In particular in Section \ref{S:particular} we will show that (Proposition \ref{P:quotient}, Lemma \ref{L:cod1} and Proposition \ref{P:absolutecont})
\begin{equation}\label{E:referenceintro}
m\llcorner_{e_{t}(G)} = \int_{\f(\mu_{0})} \hat m_{a,t} \mathcal{L}^{1}(da),\qquad 
m\llcorner_{\bar \Gamma_{a}(1)} = \int_{[0,1]} m_{a,t} \mathcal{L}^{1}(dt),
\end{equation}
and  (Proposition \ref{P:quotient} and Lemma \ref{L:cod1})
\begin{equation}\label{E:coareaintro}
\hat m_{a,t} \ll \mathcal{S}^{h}\llcorner_{e_{t}(G_{a})}, \qquad m_{a,t} \ll \mathcal{S}^{h}\llcorner_{e_{t}(G_{a})}.
\end{equation}
The latter will be fundamental in order to compare $m_{a,t}$ to $\hat m_{a,t}$. Here $\S^{h}$ denotes the spherical Hausdorff measure of codimension one, see Section \ref{Ss:coarea}.
The proofs of these results will be easier and shorter compared to the one in the general case.

In Section \ref{S:general} we prove \eqref{E:referenceintro} and \eqref{E:coareaintro} without the extra assumption on the shape of the Wasserstein geodesic. 
Anyway while \eqref{E:referenceintro} can be proven with no difficulties, 
the proof of \eqref{E:coareaintro} necessary relies on some regularity property of two important function and it is here that 
we have to introduce Assumption \ref{A:lengthreg} and Assumption \ref{A:Phi}. 
The functions are the length map at time $t$ for $t \in (0,1)$, that is $L_{t} : e_{t}(G) \to (0,\infty)$ defined by
\[
L_{t}(\gamma_{t}) = L(\gamma).
\]
And the map $\Phi_{t} : e_{t}(G) \to \erre$ defined by $\Phi_{t}(\gamma_{t}) = \f(\gamma_{0})$. Thanks to the non branching assumption on the space, both functions are well defined.  
Note that here we also observe that in the hypothesis of Section \ref{S:particular}, both 
Assumption \ref{A:lengthreg} and Assumption \ref{A:Phi} are verified by $L_{t}$ and $\Phi_{t}$. Moreover we prove that 
Assumption \ref{A:lengthreg} and Assumption \ref{A:Phi} hold if $(X,d,m)$ is a Riemannian manifold with Riemannian volume.

In Section \ref{S:unique} 
through a careful blow-up analysis (Proposition \ref{P:stesso}, Proposition \ref{P:step2} and Lemma \ref{L:welldefined}),
we prove that
\[
\hat m_{a,t} \ll m_{a,t}.
\]
If $\hat m_{a,t} = \lambda_{t} m_{a,t} $, we also prove (Theorem \ref{T:expression1}) that 
\[
\frac{1}{\lambda_{t}(\gamma_{t})} = \lim_{s \to 0} \frac{\Phi_{t}(\gamma_{t}) - \Phi_{t}(\gamma_{t+s})    }{s}.
\]
This result is a key step in the proof of the aforementioned decomposition of $\r_{t}$. 
It clarifies the expression of one of the two function decomposing $\r_{t}$.
Moreover as a consequence (Corollary \ref{C:surfacevo}) for every $t \in [0,1]$, we have $(e_{t})_{\sharp}(\gammA_{a}) \ll m_{a,t}$ . 

In Section \ref{S:reduction} we show that if $h_{a,t}$ is the 
density introduced before, then $t \mapsto h_{a,t}(\gamma_{t})$ satisfies the local reduced curvature-dimension condition $\mathsf{CD}_{loc}^{*}(K,N-1)$ (Theorem \ref{T:surface}) and therefore $\mathsf{CD}^{*}(K,N-1)$.
Here the main point, as already said before, is to use the results of Section \ref{S:dmonotone} 
and consider a geodesic in the Wasserstein space, absolute continuous with respect to $m$, 
moving in the same direction of $t \mapsto (e_{t})_{\sharp}\gammA_{a}$ 
Taking inspiration from the Riemannian framework, the volume distortion affects only $(N-1)$ dimensions.

So up to normalization constant
\[
h_{a,t} m_{a,t} = (e_{t})_{\sharp}(\gammA_{a}) = \r_{t} \hat m_{a,t} =  \r_{t}\lambda_{t} m_{a,t},
\]
with $h_{a,t}$ verifying $\mathsf{CD}^{*}(K,N-1)$. 
We have therefore proved the following result (Theorem \ref{T:main1})
\begin{theorem}\label{T:main1intro}
Let $(X,d,m)$ be a non-branching metric measure space verifying $\mathsf{CD}_{loc}(K,N)$ or $\mathsf{CD}^{*}(K,N)$
and let $\{ \mu_{t}\}_{t\in [0,1]} \subset \mathcal{P}_{2}(X,d,m)$ be a geodesic with $\mu_{t} = \r_{t} m$.
Assume moreover Assumption \ref{A:lengthreg} and Assumption \ref{A:Phi}. Then
\[
\r_{t}(\gamma_{t}) =C(a) \frac{1}{\lambda_{t}(\gamma_{t})}h_{a,t}(\gamma), \qquad \gammA-a.e. \ \gamma \in G,
\]
where $a = \f(\gamma_{0})$ and $C(a) = \| \gammA_{a} \|$ is a constant depending only on $a$. The map 
$[0,1] \ni t \mapsto h_{a,t}(\gamma)$ verifies $\mathsf{CD}^{*}(K,N-1)$ for $\gammA$-a.e. $\gamma \in G$ and 
\[
\frac{1}{\lambda_{t}(\gamma_{t})} = \lim_{s \to 0 } \frac{\Phi_{t}(\gamma_{t}) - \Phi_{t}(\gamma_{t+s})   }{s}.
\]
\end{theorem}

The constant $C(a)$ of Theorem \ref{T:main1intro} has the following explicit formula
\[
C(a) = \left( \int \r_{t}(z) \hat m_{a,t}(dz)\right)
\]
where $a = \f(\gamma_{0})$. 
Note again that the value of the integral does not depend on time, but just on $a$ and therefore in order to prove 
$\mathsf{CD}(K,N)$-like estimates, the integral can be dropped out.

In the second part of Section \ref{S:reduction} we prove that under the same assumptions of Section \ref{S:particular}
the function $\lambda_{t}(\gamma_{t})$ is linear in $t$ (Proposition \ref{P:1dim}). Hence we have obtained the other main result of this note (Theorem \ref{T:cd}).
\begin{theorem}\label{T:cdintro}
Let $(X,d,m)$ be a non-branching metric measure space verifying $\mathsf{CD}_{loc}(K,N)$ or $\mathsf{CD}^{*}(K,N)$
and let $\{ \mu_{t}\}_{t\in [0,1]} \subset \mathcal{P}_{2}(X,d,m)$ be a geodesic with $\mu_{t} = \r_{t} m$.
Assume moreover that 
\[
L(\gamma) = f(\f(\gamma_{0})), 
\]
for some $f : \f(\mu_{0}) \to (0,\infty)$ such that $\f(\mu_{0}) \ni a \mapsto a - f^{2}/a$ is non increasing.
Then 
\[
\r_{t}(\gamma_{t})^{-1/N} \geq \r_{0}(\gamma_{0})^{-1/N} \tau_{K,N}^{(1-t)}(d(\gamma_{0},\gamma_{1})) 
+ \r_{1}(\gamma_{1})^{-1/N} \tau_{K,N}^{(s)}(d(\gamma_{0},\gamma_{1})),
\]
for every $t \in [0,1]$ and for $\gammA$-a.e. $\gamma \in G$.
\end{theorem}
The family of geodesics verifying the hypothesis of Theorem \ref{T:cdintro}
includes for instance all of those optimal transportation having as Kantorovich potential 
\[
\f(x)  = \frac{1}{2} d^{2}(x,A)
\]
for any $A \subset X$. Indeed such $\f$ is $d^{2}$-concave and its weak upper gradient is always one.
No assumption on $A$ is needed and therefore no assumption on the shape of $\f^{c}$.

We conclude the note with Section \ref{S:formal} where assuming the space to be infinitesimally strictly convex (see \eqref{E:striconvex}), 
we prove that (Proposition \ref{P:lambda})
\[
\frac{1}{\lambda_{t}(\gamma_{t})} = D\Phi_{t} (\nabla \f_{t})(\gamma_{t}),\qquad \gammA-a.e. \gamma,
\]
and hence the general decomposition: up to a constant (in time) factor become
\[
\r_{t} =    D\Phi_{t} (\nabla \f_{t})   h_{t}. 
\]
We conclude the note with a formal calculation in the Euclidean space
putting in relation $D\Phi_{t} (\nabla \f_{t})$ with the Hessian of $\f_{t}$.
\\
\\

Our starting hypothesis can be chosen to be equivalently $\mathsf{CD}_{loc}(K,N)$ or $\mathsf{CD}^{*}(K,N)$.
Hence the results proved can be read from two different perspective, accordingly to $\mathsf{CD}_{loc}(K,N)$ or $\mathsf{CD}^{*}(K,N)$.
From the point of view of $\mathsf{CD}^{*}(K,N)$, where the globalization property is already known, the main 
result is that for nice optimal transportations the entropy inequality can be improved to the curvature-dimension condition, 
giving a ``self-improving'' type of result. From the point of view of $\mathsf{CD}_{loc}(K,N)$ clearly the main issue is the globalization problem. 
Here the main statement is that the local-to-global property is true for nice optimal transportations and in
the general case under the aforementioned regularity properties, is almost equivalent to  the concavity of the 1-dimensional density $\lambda_{t}$. 
The latter it is in turn strongly linked to the composition property of the differential operator $D$.

The last comment is for the assumption of non branching property for $(X,d,m)$. As shown by Rajala and Sturm in \cite{rajasturm:branch},
strong $\mathsf{CD}(K,\infty)$-spaces and Riemannian $\mathsf{CD}(K,N)$ for $N \in \erre \cup \{\infty\}$ have the property 
that for any couple of probability measures $\mu_{0},\mu_{1}$ with $\mu_{0},\mu_{1} \ll m$ all the 
$L^{2}$-optimal transportations are concentrated on a set of non branching geodesics.  
That is all $\gammA \in \mathcal{P}(\G(X))$, dynamical optimal plans 
with starting point $\mu_{0}$ and ending point $\mu_{1}$ are such that the evaluation map
for each $t \in [0,1)$
\[
e_{t} : G \to X
\]
is injective, even if the space is not assumed to be non branching,  where $G$ is the support of $\gammA$.

Since our construction relies not only on the $L^{2}$-optimal dynamical plan but
on the strong interplay between $d^{2}$-cyclically monotone sets and $d$-cyclically monotone sets, 
the substitution of the non branching property of the space with $\mathsf{RCD}$-condition or 
with the strong $\mathsf{CD}(K,\infty)$  is a delicate task that would go beyond the scope of this note. For instance 
$\mathsf{RCD}$-condition will not prevent the following ``bad'' situation:  $\gamma, \hat \gamma \in G_{a}$
so that they have a common point $z = \gamma_{s} =\hat \gamma_{t}$ for $t \neq s$. 
In particular the proof of Lemma \ref{L:partition}, that is one the building block of our analysis, does not work only assuming 
non branching support of $\gammA$.
\\
\\
\textbf{Acknowledgement.}
I would like to warmly thank Martin Huesmann for comments and discussions on an earlier draft. 
I also warmly thank an anonymous reviewer for his extremely detailed and constructive report.

\section{Preliminaries}\label{S:preli}
Let $(X,d)$ be a metric space. 
The length $L(\gamma)$ of a continuous curve $\gamma : [0,1] \to X$ is defined as
\[
L(\gamma) : = \sup \sum_{k=1}^{n} d(\gamma(t_{k-1}),\gamma(t_{k}))
\]
where the supremum runs over $n \in \enne$ and over all partitions $0 = t_{0} < t_{1}< \dots < t_{n}=1$. 
Note that $L(\gamma) \geq d(\gamma(0),\gamma(1))$. A curve is called \emph{geodesic} if and only if 
$L(\gamma) = d(\gamma(0),\gamma(1))$.
If this is the case, we can assume $\gamma$ to have constant speed, i.e. 
$L(\gamma\llcorner_{[s,t]}) =|s-t| L(\gamma)= |s-t|d(\gamma(0),\gamma(1))$ for 
every $0\leq s \leq t \leq 1$.

Denote by $\G(X)$ the space of geodesic $\gamma: [0,1] \to X$ in $X$, regarded as subset of $C([0,1],M)$
of continuous functions equipped with the topology of uniform convergence.

$(X,d)$ is said to be a \emph{length space} if and only if for every $x,y \in X$,
\[
d(x,y) = \inf L(\gamma)
\]
where the infimum runs over all continuous curves joining $x$ and $y$. It is said to be a \emph{geodesic space} if 
all $x$ and $y$ are connected by a geodesic.
A point $z$ will be called $t$-intermediate point of points $x$ and $y$ if $d(x, z)=td(x, y)$ and $d(z, y)=(1-t)d(x, y)$.
\begin{definition}
A geodesic space $(X,d)$ is \emph{non-branching} if and only if for every $r \geq 0$ and $x,y \in X$ such that $d(x,y)=r/2$, the set 
\[
\{z \in X : d (x,z) = r \} \cap \{ z \in X : d(y,z) = r/2 \}
\]
consists of a single point.
\end{definition}
Throughout the following we will denote by $B_{r}(z)$ the open ball of radius $r$ centered in $z$. A standard map in optimal transportation is 
the evaluation map: for a fixed $t \in [0,1]$, $e_{t} : \G(X) \to X$ is defined by $e_{t}(\gamma) : = \gamma_{t}$.
The push-forward of a given measure, say $\eta$, via a map $f$ will be denoted by $f_{\sharp} \eta$ and is defined by 
$f_{\sharp}\eta (A) : = \eta(f^{-1}(A))$, for any measurable $A$.

\subsection{Geometry of metric measure spaces}\label{Ss:geom}
What follows is contained \cite{sturm:MGH2}.

A \emph{metric measure space} is a triple $(X,d,m)$ where 
$(X,d)$ is a complete separable metric space and $m$ is a locally finite measure (i.e. $m(B_{r}(x))< \infty$ for all $x\in X$ 
and all sufficiently small $r>$0) on $X$ equipped with its Borel $\sigma$-algebra. We exclude the case $m(X)=0$.
A \emph{non-branching} metric measure space will be a metric measure space $(X,d,m)$ such that $(X,d)$ is a non-branching geodesic space.

$\mathcal{P}_{2}(X,d)$ denotes the $L^{2}$-Wasserstein space of Borel probability measures on $X$ and $W_{2}$ the corresponding $L^{2}$-Wasserstein distance. 
The subspace of $m$-absolutely continuous measures is denoted by $\mathcal{P}_{2}(X, d, m)$.

The following are well-known results in optimal transportation theory and are valid for general metric measure spaces.
\begin{lemma}\label{L:geod}
Let $(X,d,m)$ be a metric measure space.
For each geodesic $\mu: [0,1] \to \mathcal{P}_{2}(X,d)$ there exists a probability measure $\gammA$ on $\G(X)$ such that 
\begin{itemize}
\item $e_{t\,\sharp} \gammA = \mu_{t}$ for all $t \in [0,1]$;
\item for each pair $(s,t)$ the transference plan $(e_{s},e_{t})_{\sharp} \gammA$ is an optimal coupling for $W_{2}$.
\end{itemize}
\end{lemma}

Consider the R\'enyi entropy functional 
\[
\S_{N}(\, \cdot\, | m ) : \mathcal{P}_{2}(X,d) \to \erre
\]
with respect to $m$, defined by
\begin{equation}\label{E:entropy}
\mathcal{S}_{N}(\mu | m) : = - \int_{X} \r^{-1/N}(x) \mu(dx)
\end{equation}
for $\mu \in \mathcal{P}_{2}(X)$, where $\r$ is the density of the absolutely continuous part $\mu^{c}$ in the Lebesgue decomposition 
$\mu = \mu^{c} + \mu^{s} = \r m + \mu^{s}$.

Given two numbers $K,N\in \erre$ with $N\geq1$, we put for $(t,\theta) \in[0,1] \times \erre_{+}$,
\begin{equation}\label{E:tau}
\tau_{K,N}^{(t)}(\theta):= 
\begin{cases}
\infty, & \textrm{if}\ K\theta^{2} \geq (N-1)\pi^{2}, \crcr
\displaystyle  t^{1/N}\Bigg(\frac{\sin(t\theta\sqrt{K/(N-1)})}{\sin(\theta\sqrt{K/(N-1)})}\Bigg)^{1-1/N} & \textrm{if}\ 0< K\theta^{2} \leq (N-1)\pi^{2}, \crcr
t & \textrm{if}\ K \theta^{2}<0\ \textrm{or}\\& \textrm{if}\ K \theta^{2}=0\ \textrm{and}\ N=1,  \crcr
\displaystyle  t^{1/N}\Bigg(\frac{\sinh(t\theta\sqrt{-K/(N-1)})}{\sinh(\theta\sqrt{-K/(N-1)})}\Bigg)^{1-1/N} & \textrm{if}\ K\theta^{2} \leq 0 \ \textrm{and}\ N>1.
\end{cases}
\end{equation}
That is, $\tau_{K,N}^{(t)}(\theta): = t^{1/N} \sigma_{K,N-1}^{(t)}(\theta)^{(N-1)/N}$ where
\[
\sigma_{K,N}^{(t)}(\theta) = \frac{\sin(t\theta\sqrt{K/N})}{\sin(\theta\sqrt{K/N})}, 
\]
if $0 < K\theta^{2}<N\pi^{2}$ and with appropriate interpretation otherwise. Moreover we put 
\[
\varsigma_{K,N}^{(t)}(\theta): = \tau_{K,N}^{(t)}(\theta)^{N}.
\]
The coefficients $\tau_{K,N}^{(t)}(\theta),\sigma_{K,N}^{(t)}(\theta)$ and $\varsigma_{K,N}^{(t)}(\theta)$ are the volume distortion coefficients 
with $K$ playing the role of curvature and $N$ the one of dimension.

The curvature-dimension condition $\mathsf{CD}(K,N)$ is defined in terms of convexity properties of the entropy functional.
In the following definitions $K$ and $N$ will be real numbers with $N\geq1$.

\begin{definition}[Curvature-Dimension condition]\label{D:CD}
We say that $(X,d,m)$ satisfies $\mathsf{CD}(K,N)$  if and only if for each pair 
$\mu_{0}, \mu_{1} \in \mathcal{P}_{2}(X,d,m)$ there exists an optimal coupling $\pi$ of $\mu_{0}=\r_{0}m$ and $\mu_{1}=\r_{1}m$,
and a geodesic $\mu:[0,1] \to \mathcal{P}_{2}(X,d,m)$ connecting $\mu_{0}$ and $\mu_{1}$ such that 
\begin{equation}\label{E:CD}
\begin{aligned}
\mathcal{S}_{N'}(\mu_{t}|m) \leq   - \int_{X\times X}& \Big[ \tau_{K,N'}^{(1-t)}(d(x_{0},x_{1}))\r_{0}^{-1/N'}(x_{0})  \crcr
&~ + \tau_{K,N'}^{(t)}(d(x_{0},x_{1}))\r_{1}^{-1/N'}(x_{1})  \Big]  \pi(dx_{0}dx_{1}),
\end{aligned}
\end{equation}
for all $t \in [0,1]$ and all $N'\geq N$.
\end{definition}

The following is a variant of $\mathsf{CD}(K,N)$ and it has been introduced in \cite{sturm:loc}.

\begin{definition}[Reduced Curvature-Dimension condition]\label{D:CD*}
We say that $(X,d,m)$ satisfies $\mathsf{CD}^{*}(K,N)$ if and only if for each pair 
$\mu_{0}, \mu_{1} \in \mathcal{P}_{2}(X,d,m)$ there exists an optimal coupling $\pi$ of $\mu_{0}=\r_{0}m$ and $\mu_{1}=\r_{1}m$,
and a geodesic $\mu:[0,1] \to \mathcal{P}_{2}(X,d,m)$ connecting $\mu_{0}$ and $\mu_{1}$ 
such that \eqref{E:CD} holds true for all $t \in [0,1]$ and all $N'\geq N$ 
with the coefficients $\tau_{K,N}^{(t)}(d(x_{0},x_{1}))$ and $\tau_{K,N}^{(1-t)}(d(x_{0},x_{1}))$
replaced by $\sigma_{K,N}^{(t)}(d(x_{0},x_{1}))$ and $\sigma_{K,N}^{(1-t)}(d(x_{0},x_{1}))$, respectively.
\end{definition}

For both definitions there is a local version. Here we state only the local counterpart of $\mathsf{CD}(K,N)$, 
being clear what would be the one for $\mathsf{CD}^{*}(K,N)$.

\begin{definition}[Local Curvature-Dimension condition]\label{D:loc}
We say that $(X,d,m)$ satisfies $\mathsf{CD}_{loc}(K,N)$ if and only if each point $x \in X$ has a neighborhood $X(x)$ such that for each pair 
$\mu_{0}, \mu_{1} \in \mathcal{P}_{2}(X,d,m)$ supported in $X(x)$ there exists an optimal coupling $\pi$ of $\mu_{0}=\r_{0}m$ and $\mu_{1}=\r_{1}m$,
and a geodesic $\mu:[0,1] \to \mathcal{P}_{2}(X,d,m)$ connecting $\mu_{0}$ and $\mu_{1}$
such that \eqref{E:CD} holds true for all $t \in [0,1]$ and all $N'\geq N$.
\end{definition}
It is worth noticing that in the previous definition the geodesic $\mu$ can exit from the neighborhood $X(x)$.

One of the main property of the reduced curvature dimension condition is the globalization one:  
under the non-branching assumption conditions $\mathsf{CD}^{*}_{loc}(K,N)$ and $\mathsf{CD}^{*}(K,N)$ are equivalent.
Moreover it holds:
\begin{itemize}
\item $\mathsf{CD}^{*}_{loc}(K,N)$ is equivalent to $\mathsf{CD}_{loc}(K,N)$; 
\item $\mathsf{CD}(K,N)$ implies $\mathsf{CD}^{*}(K,N)$;
\item $\mathsf{CD}^{*}(K,N)$ implies $\mathsf{CD}(K^{*},N)$ where $K^{*} = K(N-1)/N$.
\end{itemize}
Hence it is possible to pass from $\mathsf{CD}_{loc}$ to $\mathsf{CD}$ at the price of 
passing through $\mathsf{CD}^{*}$ and therefore worsening the lower bound on the curvature.
For all of these properties, see \cite{sturm:loc}.

If a non-branching $(X,d,m)$ satisfies $\mathsf{CD}(K,N)$ then geodesics are unique $m \otimes m$-a.e.. 
\begin{lemma}\label{L:map}
Assume that $(X,d,m)$ is non-branching and satisfies $\mathsf{CD}(K,N)$ for some pair $(K,N)$. Then for every $x\in\supp[m]$ and $m$-a.e. $y\in X$ 
(with the exceptional set depending on x) there exists a unique geodesic between $x$ and $y$. 

Moreover there exists a measurable map $\gamma: X^{2} \to \G(X)$ such that for $m\otimes m$-a.e. $(x,y) \in X^{2}$ the curve $t \mapsto \gamma_{t}(x,y)$
is the unique geodesic connecting $x$ and $y$.
\end{lemma}

Under non-branching assumption is possible to formulate $\mathsf{CD}(K,N)$ in an equivalent point-wise version:  $(X,d,m)$ satisfies 
$\mathsf{CD}(K,N)$ if and only if for each pair $\mu_{0},\mu_{1}\in \mathcal{P}_{2}(X,d,m)$ and each  dynamical optimal plan
$\gammA$,
\begin{equation}\label{E:cdpunto}
\r_{t}(\gamma_{t}(x_{0},x_{1}))\leq \Big[ \tau_{K,N'}^{(1-t)}(d(x_{0},x_{1}))\r_{0}^{-1/N'}(x_{0}) + 
\tau_{K,N'}^{(t)}(d(x_{0},x_{1}))\r_{1}^{-1/N'}(x_{1})  \Big]^{-N}, 
\end{equation}
for all $t \in [0,1]$, and $(e_{0},e_{1})_{\sharp}\gammA$-a.e. $(x_{0},x_{1}) \in X \times X$. Here $\r_{t}$ is the density of the geodesic 
$(e_{t})_{\sharp}\gammA$.
Recall that $\boldsymbol{\gamma} \in \mathcal{P}(\G(X))$ is a dynamical optimal plan if 
$\pi = (e_{0},e_{1})_{\sharp}\boldsymbol{\gamma} \in \Pi(\mu_{0},\mu_{1})$ is optimal and the 
map $t \mapsto \mu_{t}: = e_{t\,\sharp}\boldsymbol{\gamma}$ is a geodesic in the 2-Wasserstein space.

We conclude with a partial list of properties enjoyed by metric measure spaces satisfying $\mathsf{CD}^{*}(K,N)$ (or $\mathsf{CD}_{loc}(K,N))$. 
If  $(X,d,m)$ verifies $\mathsf{CD}^{*}(K,N)$ then: 
\begin{itemize}
\item $m$ is a doubling measure;
\item $m$ verifies Bishop-Gromov volume growth inequality;
\item $m$ verifies Brunn-Minkowski inequality;
\end{itemize}
with all of these properties stated in a quantitative form.

\subsection{Spherical Hausdorff measure of codimension 1 and Coarea formula}\label{Ss:coarea}
What follows is contained in \cite{ambro:perimeter} and is valid under milder assumption than 
$\mathsf{CD}^{*}(K,N)$ (or $\mathsf{CD}_{loc}(K,N)$) but for an easier exposition we assume $(X,d,m)$ to satisfy $\mathsf{CD}^{*}(K,N)$.

Recall that for $K \geq 0$ the measure $m$ is doubling that is $m( B_{2r}(x)) \leq (C_{D}/2) m( B_{r}(x))$ where $C_{D}$ is the doubling constant of $m$. 
If $K < 0$ the measure $m$ is locally uniformly doubling, i.e. $m( B_{2r}(x)) \leq (C_{R}/2) m( B_{r}(x))$ for any $r \leq R$ and some constant $C_{R}$
depending on $R$ but not on $x$.

If $B(X)$ is the set of balls, define the function $h : B(X) \to [0,\infty]$
as
\[
h(\bar B_{r}(x)) := \frac{m(\bar B_{r} (x))}{r}.
\]

Due to the (locally uniformly) doubling properties of $m$, the function $h$ turns out to be a (locally uniformly) doubling function.
Then, using the Carath\'eodory construction, we may define the generalized Hausdorff spherical measure $\S^{h}$ as 
\begin{equation}\label{E:spherical}
\S^{h}(A) : = \lim_{r\downarrow 0}  \inf\left\{ \sum_{i\in \enne} h(B_{i}) : B_{i} \in B(X), A \subset \bigcup_{i\in \enne} B_{i}, \textrm{diam} (B_{i}) \leq r \right\}.
\end{equation}

The space of functions of bounded variation $BV(X)$ and the perimeter measure have been studied in \cite{ambro:perimeterbv}, \cite{ambro:perimeter}, 
\cite{miranda:bvcoarea}, \cite{miranda:bvmetric}. If $u \in BV(X)$, its total variation measure will be denoted with $|Du|$.
We will use the following Coarea formula. 

\begin{theorem}[\cite{miranda:bvcoarea}, Theorem 4.3, Theorem 4.4]
For every $u \in BV(X)$ and every Borel set $A \subset X$ it holds
\[
|Du|(A) = \int_{-\infty}^{\infty}P(\{u>t\}, A) dt.
\]
Moreover for any set $E\subset X$ of finite perimeter, the measure $P(E,\cdot)$ is concentrated on a subset of the essential boundary $\partial^{*}E$ and 
for any Borel set $B\subset X$
\[
\frac{1}{c} \S^{h} (B\cap \partial^{*}E )  \leq P(E,B) \leq c \, \S^{h} (B\cap \partial^{*}E )
\]
with $c>0$ depending only on $K$ and $N$.
\end{theorem}

If $u$ is a Lipschitz function, its total variation is equivalent as measure to $\| \nabla u\| m$, where 
\begin{equation}\label{E:gradient}
\|\nabla u \|(x) : = \liminf_{r \to 0} \frac{1}{r} \sup_{y \in \bar B_{r}(x)} |u(y) - u(x)|.
\end{equation}
The following comparison is taken from \cite{miranda:bvmetric}: for any Borel set $A\subset X$
\begin{equation}\label{E:comparison}
c_{0} \int_{A} \| \nabla u\| (x) m(dx) \leq |Du|(A) \leq \int_{A} \| \nabla u\| (x) m(dx),
\end{equation}
for some constant $c_{0} > 0$ depending again only on $K,N$.
The last result we would like to recall is a particular form of Coarea formula for Lipschitz functions.
\begin{proposition}[\cite{ambro:perimeter}, Proposition 5.1]\label{P:noboundary}
For any $u$ Lipschitz function defined on $X$ and any $B$ Borel set we have 
\[ 
\int_{\erre} \S^{h}(B \cap u^{-1}(t)) dt \leq Lip(u) m(B). 
\]
\end{proposition}

\subsection{Gradients and differentials}\label{Ss:gradiff}

This part is taken from \cite{gigli:laplacian}.
A curve $\gamma \in C([0,1],X)$ is said to be absolutely continuous provided there exists $f \in L^{1}([0,1])$ such that 
\[
d(\gamma_{s},\gamma_{t}) \leq \int_{s}^{t}f(\tau)d\tau, \qquad \forall s,t\in [0,1], s\leq t.
\]
Let $AC([0,1],X)$ denote the set of absolutely continuous curves. If $\gamma \in AC([0,1],X)$ then the limit
\[
\lim_{\tau \to 0} \frac{d(\gamma_{t+\tau},\gamma_{t})}{\tau}
\]
exists for a.e. $t \in [0,1]$, is called metric derivative and denoted by $|\dot \gamma_{t}|$.

Given Borel functions $f:X \to \erre, G:X\to [0,\infty]$ we say that $G$ is an upper gradient of $f$ provided
\[
|f(\gamma_{0})-f(\gamma_{1})| \leq \int_{0}^{1} G(\gamma_{t})|\dot \gamma_{t}| dt, \quad \forall \gamma \in AC([0,1], M),
\]
where $|\dot \gamma_{t}|$ is the metric derivative of $\gamma$ in $t$.
For $f : X \to \erre$ the local Lipschitz constant $|Df|: X \to [0,\infty]$ is defined by
\[
|Df|(x) : = \limsup_{y \to x} \frac{|f(y)-f(x)|}{d(y,x)}
\]
if $x$ is not isolated, and $0$ otherwise. Define 
\[
|D^{+}f|(x) : = \limsup_{y \to x} \frac{(f(y)-f(x))^{+}}{d(y,x)},\quad
|D^{-}f|(x) : = \limsup_{y \to x} \frac{(f(y)-f(x))^{-}}{d(y,x)},
\]
the ascending and descending slope respectively. If $f$ is locally Lipschitz, then $|D^{\pm}f|, |Df|$ are all upper gradients of $f$.
In order to give a weaker notion of slope, consider the following family: $\gammA \in \mathcal{P}(C([0,1],X))$ is a \emph{test plan} if 
\[
e_{t\,\sharp}\gammA \leq C m, \quad \forall t \in [0,1], \qquad \textrm{and}\quad \int \int_{0}^{1}|\dot \gamma_{t}|dt \gammA(d\gamma)< \infty,
\]
where $C$ is a positive constant. Therefore we have the following.
\begin{definition}\label{D:sobolev}
A Borel map $f:X \to \erre$ belongs to the Sobolev class $S^{2}(X,d,m)$ (resp. $S^{2}_{loc}(X,d,m)$) if there exists a non-negative function 
$G \in L^{2}(X,m)$ (resp. $L^{2}_{loc}(X,m)$) such that 
\begin{equation}\label{E:sobolev}
\int |f(\gamma_{0}) - f(\gamma_{1})| \gammA(d\gamma) \leq \int \int_{0}^{1}G(\gamma_{s})|\dot \gamma_{s}|ds \gammA (d\gamma),\qquad \forall \gammA
\ \textrm{test plan.}
\end{equation} 
If this is the case, $G$ is called \emph{weak upper gradient}. 
\end{definition}

For $f \in S^{2}(X,d,m)$ there exists a minimal function $G$, in the $m$-a.e. sense, in $L^{2}(X,m)$ such that \eqref{E:sobolev} holds.
Denote such minimal function with $|Df|_{w}$. Accordingly define the semi-norm $\| f \|_{S^{2}(X,d,m)} : = \| |Df|_{w}\|_{L^{2}(X,m)}$.

We now state a result on the weak upper gradient of Kantorovich potentials also known as metric Brenier's Theorem.
\begin{proposition}[\cite{ambrgisav:heat}, Theorem 10.3]\label{P:speed} 
Let $(X,d,m)$ verify  $\mathsf{CD}(K,N)$ for $K \in \erre$ and $N\geq1$ and be non-branching. Let $\mu_{0}, \mu_{1} \in \mathcal{P}_{2}(X,d,m)$, $\f$ 
be a Kantorovich potential. Then for every $\gammA$ optimal dynamical transference plan it holds
\[
d(\gamma_{0},\gamma_{1}) = |D\f|_{w}(\gamma_{0}) = |D^{+}\f|(\gamma_{0}),  \qquad  \textrm{for } \gammA-a.e. \gamma. 
\]
If moreover the densities of $\mu_{0}$ and of $\mu_{1}$ are both in $L^{\infty}(X,m)$, then
\[
\lim_{t\downarrow 0} \frac{\f(\gamma_{0}) -\f(\gamma_{t})}{d(\gamma_{0},\gamma_{t})} = d(\gamma_{0},\gamma_{1}), \qquad
\textrm{in } L^{2}(\G(X),\gammA).
\]
\end{proposition}

In order to compute higher order derivatives, we introduce the following.
\begin{definition}\label{D:differential}
Let $f,g \in S^{2}(X,d,m)$. The functions 
\begin{align*}
D^{+}f(\nabla g):=&~ \liminf_{\ve \downarrow 0} \frac{|D(g+\ve f)|^{2}_{w}  -|Dg|_{w}^{2}}{2\ve}, \crcr
D^{-}f(\nabla g):=&~ \limsup_{\ve \uparrow 0} \frac{|D(g+\ve f)|^{2}_{w}  -|Dg|_{w}^{2}}{2\ve}.
\end{align*}
are well defined.
\end{definition}

Spaces where the two differentials coincide are called \emph{infinitesimally strictly convex}, i.e.
$(X,d,m)$ is said to be infinitesimally strictly convex provided
\begin{equation}\label{E:striconvex}
\int D^{+}f(\nabla g)(x) m(dx) = \int D^{-}f (\nabla g)(x) m(dx), \qquad \forall f,g \in S^{2}(X,d,m).
\end{equation}
It is proven in \cite{gigli:laplacian} that \eqref{E:striconvex} is equivalent to the point-wise one: 
\[
D^{+}f(\nabla g) = D^{-}f(\nabla g), \quad m-a.e., \quad \forall f,g \in S^{2}_{loc}(X,d,m).
\]
If the space is infinitesimally strictly convex, we can denote by $Df(\nabla g)$ the common value and 
$Df (\nabla g)$ is linear in $f$ and $1$-homogeneous and continuous in $g$.
 
There is a strong link between differentials and derivation along families of curves. 
For $\gammA \in \mathcal{P}(C([0,1],X))$, define the norm $\| \gammA\|_{2} \in [0,\infty]$ of $\gammA$ by
\[
\| \gammA\|_{2}^{2} : = \limsup_{t \downarrow 0} \frac{1}{t} \int \int_{0}^{t}|\dot \gamma_{s}|^{2}ds \gammA(d\gamma),
\] 
if $\gamma \in \mathcal{P}(AC([0,1],X))$ and $+\infty$ otherwise. 

\begin{definition}\label{D:representing}
Let $g \in S^{2}(X,d,m)$. We say that $\gamma \in \mathcal{P}(C([0,1],X))$ represents $\nabla g$ if 
$\gammA$ is of bounded compression, $\| \gammA\|_{2}< \infty$, and it holds   
\begin{equation}\label{E:representing}
\liminf_{t\downarrow 0} \int \frac{g(\gamma_{t}) -g(\gamma_{0})}{t} \gammA(d\gamma)
\geq \frac{1}{2} \big( \| |Dg|_{w} \|^{2}_{L^{2}(X,e_{0\,\sharp}\gammA)}  + \| \gammA\|_{2}^{2}\big).
\end{equation}
\end{definition}

A straightforward consequence of \eqref{E:representing} is that if $\gammA$ represents $\nabla g$, then 
the whole limit in the lefthand-side of \eqref{E:representing} exists and verifies 
\[
\lim_{t\downarrow 0} \int \frac{g(\gamma_{t}) -g(\gamma_{0})}{t} \gammA(d\gamma)
=\frac{1}{2} \big( \| |Dg|_{w} \|^{2}_{L^{2}(X,e_{0\,\sharp}\gammA)}  + \| \gammA\|_{2}^{2}\big).
\]

\begin{theorem}[\cite{gigli:laplacian}, Theorem 3.10]\label{T:changingorder}
Let $f,g \in S^{2}(X,d,m)$. For every $\gammA \in \mathcal{P}(C([0,1],M))$ representing $\nabla g$ 
it holds
\begin{align*}
\int D^{+}f(\nabla g) e_{0\,\sharp}\gammA \geq &~ 
 \limsup_{t\downarrow0} \int \frac{f(\gamma_{t}) -f( \gamma_{0})}{t} \gammA(d\gamma) \crcr
\geq &~\liminf_{t\downarrow0} \int \frac{f(\gamma_{t} )- f(\gamma_{0})}{t} \gammA(d\gamma) \geq
\int D^{-}f(\nabla g) e_{0\,\sharp}\gammA. 
\end{align*}
\end{theorem}

\subsection{Hopf-Lax formula for Kantorovich potentials}\label{Ss:slopKant}
What follows is contained in \cite{ambrgisav:heat}.

The definitions below make sense for a general Borel and real valued cost but we will only consider the $d^{2}/2$ case, for this reason $c$ has to be interpret as 
$d^{2}/2$.

\begin{definition}
Let $\f : X \to \erre \cup \{ \pm  \infty \}$. Its  $d^{2}$-\emph{transform} $\f^{c} : X \to \erre \cup \{ -  \infty \}$ 
is defined by
\[
\f^{c}(y) := \inf_{x \in X} \frac{1}{2}d^{2}(x,y) - \f(x).
\]
Accordingly $\f: X \mapsto \erre \cup \{ \pm  \infty \}$ is $d^{2}$-\emph{concave} if there exists $v : X \to \erre \cup \{-\infty \}$ such that $\f = v^{c}$.

A $d^{2}$-\emph{concave} function $\f$ such that $(\f,\f^{c})$ is a maximizing pair for the dual Kantorovich problem between 
$\mu_{0},\mu_{1}$ is called a $d^{2}$-\emph{concave Kantorovich potential} for the couple $(\mu_{0},\mu_{1})$. A function $\f$
is called a $d^{2}$-\emph{convex Kantorovich potential} if $- \f$ is a $d^{2}$-concave Kantorovich potential.
\end{definition}

We are interested in the evolution of potentials. 
They evolve accordingly to the Hopf-Lax evolution semigroup $H_{t}^{s}$ via the following formula: 
\begin{equation}\label{E:potentials}
H_{t}^{s}(\psi) (x) : =  
\begin{cases}
\displaystyle \inf_{y \in X} \frac{1}{2}\frac{d^{2}(x,y)}{s-t} + \psi(y), &~ \textrm{if}\ t < s, \crcr
\psi(x), &~ \textrm{if}\ t = s, \crcr
\displaystyle \sup_{y \in X} \psi(y)- \frac{1}{2}\frac{d^{2}(x,y)}{t-s}, &~ \textrm{if}\ t > s.
\end{cases}
\end{equation}
We also introduce the rescaled cost $c^{t.s}$ defined by
\[
c^{t,s}(x,y) : =\frac{1}{2} \frac{d^{2}(x,y)}{s-t}, \qquad \forall t < s, \, x,y \in X.
\]
Observe that for $t<r<s$
\[
c^{t,r}(x,y) + c^{r,s}(y,z) \geq c^{t,s}(x,z), \qquad \forall x,y,z \in X,
\]
and equality holds if and only if there is a constant speed geodesic $\gamma : [t,s] \to X$ such that $x = \gamma_{t}$, $y =\gamma_{r}$ and 
$z= \gamma_{s}$. The following result is taken from \cite{villa:Oldnew} (Theorem 7.30 and Theorem 7.36) but here we report a different version.

\begin{theorem}[\cite{ambro:userguide}, Theorem 2.18] \label{T:potential}
Let $(\mu_{t})\subset \mathcal{P}_{2}(X)$ be a constant speed geodesic in $(\mathcal{P}_{2}(X,d), d_{W})$ and $\psi$ a $c^{0,1}$-convex Kantorovich 
potential for the couple $(\mu_{0}, \mu_{1})$. Then $\psi_{s}: = H_{0}^{s}(\psi)$ is a $c^{t,s}$-concave Kantorovich potential for $(\mu_{s},\mu_{t})$, for 
any $ t < s$. 

Similarly, if $\phi$ is a $c$-concave Kantorovich potential for $(\mu_{1},\mu_{0})$, then $H_{1}^{t}$ is a $c^{t,s}$-convex Kantorovich potential for 
$(\mu_{t},\mu_{s})$, for any $t<s$.
\end{theorem}

The following is an easy consequence.

\begin{corollary}\label{C:1d-evo}
Let $\f$ be a $d^{2}$-concave Kantorovich potential for $(\mu_{0},\mu_{1})$. Let $\f_{t} : = - H^{t}_{1}(\f^{c})$ be a $c^{t,1}$-concave Kantorovich potential for $(\mu_{t},\mu_{1})$ and analogously let $\f^{c}_{t}:=H^t_0(-\f)$ a $c^{0,t}$-concave Kantorovich potential for $(\mu_{t},\mu_{0})$. 
Then:
\[
\f_{t}(\gamma_{t}) = \f(\gamma_{0}) - \frac{t}{2} d^{2}(\gamma_{0},\gamma_{1}), \qquad
\f^{c}_{t}(\gamma_{t}) = \f^{c}(\gamma_{1}) - \frac{1-t}{2} d^{2}(\gamma_{0},\gamma_{1}),
\qquad \boldsymbol{\gamma}-a.e. \ \gamma.
\]
\end{corollary}
\begin{proof}
Since the proofs of the statements for $\f_{t}$ and for $\f^{c}_{t}$ are the same, we prefer
to present only the one for $\f_{t}$.

Since  
\[
\f_{t} (x) = -H_{1}^{t}(\f^{c})(x) = \inf_{y \in X}\frac{1}{2} \frac{d^{2}(x,y)}{1-t} - \f^{c}(y). 
\]
for $\boldsymbol{\gamma}$-a.e.  $\gamma$
\[
\f_{t}(\gamma_{t}) \leq \frac{1}{2} \frac{d^{2}(\gamma_{t},\gamma_{1})}{1-t} + \f(\gamma_{0}) - \frac{1}{2} d^{2}(\gamma_{0},\gamma_{1}) = 
\f(\gamma_{0}) - \frac{t}{2} d^{2}(\gamma_{0},\gamma_{1}).
\]
To prove the opposite inequality: observe that 
\[
\frac{d^{2}(\gamma_{0},\gamma_{t})}{t} + \frac{d^{2}(\gamma_{t},y)}{1-t} \geq d^{2}(\gamma_{0},y), 
\]
therefore for $\boldsymbol{\gamma}$-a.e.  $\gamma$
\[
\frac{1}{2} \frac{d^{2}(\gamma_{t},y)}{1-t} - \f^{c}(y) \geq \frac{1}{2} \frac{d^{2}(\gamma_{t},y)}{1-t} - \frac{1}{2} d^{2}(\gamma_{0},y) +\f(\gamma_{0})
\geq\f(\gamma_{0}) -  \frac{1}{2} \frac{d^{2}(\gamma_{0},\gamma_{t})}{t}  =  \f(\gamma_{0}) -  \frac{t}{2} d^{2}(\gamma_{0},\gamma_{1}).
\]
Taking the infimum the claim follows.
\end{proof}

\subsection{Disintegration of measures}
\label{Ss:disintegrazione}
We conclude this introductory part with a short review on disintegration theory. What follows is taken from \cite{biacar:cmono}.

Given a measurable space $(R, \mathscr{R})$ and a function $r: R \to S$, with $S$ generic set, we can endow $S$ with the \emph{push forward $\sigma$-algebra} $\mathscr{S}$ of $\mathscr{R}$:
\[
Q \in \mathscr{S} \quad \Longleftrightarrow \quad r^{-1}(Q) \in \mathscr{R},
\]
which could also be defined as the biggest $\sigma$-algebra on $S$ such that $r$ is measurable. Moreover given a measure space 
$(R,\mathscr{R},\rho)$, the \emph{push forward measure} $\eta$ is then defined as $\eta := (r_{\sharp}\rho)$.

Consider a probability space $(R, \mathscr{R},\rho)$ and its push forward measure space $(S,\mathscr{S},\eta)$ induced by a map $r$. From the above definition the map $r$ is measurable.

\begin{definition}
\label{defi:dis}
A \emph{disintegration} of $\rho$ \emph{consistent with} $r$ is a map $\rho: \mathscr{R} \times S \to [0,1]$ such that
\begin{enumerate}
\item  $\rho_{s}(\cdot)$ is a probability measure on $(R,\mathscr{R})$ for all $s\in S$,
\item  $\rho_{\cdot}(B)$ is $\eta$-measurable for all $B \in \mathscr{R}$,
\end{enumerate}
and satisfies for all $B \in \mathscr{R}, C \in \mathscr{S}$ the consistency condition
\[
\rho\left(B \cap r^{-1}(C) \right) = \int_{C} \rho_{s}(B) \eta(ds).
\]
A disintegration is \emph{strongly consistent with respect to $r$} if for all $s$ we have $\rho_{s}(r^{-1}(s))=1$.
\end{definition}

The measures $\rho_s$ are called \emph{conditional probabilities}.

We say that a $\sigma$-algebra $\mathcal{H}$ is \emph{essentially countably generated} with respect to a measure $m$ if there exists a countably generated $\sigma$-algebra $\hat{\mathcal{H}}$ such that for all $A \in \mathcal{H}$ there exists $\hat{A} \in \hat{\mathcal{H}}$ such that $m (A \vartriangle \hat{A})=0$.

We recall the following version of the disintegration theorem that can be found on \cite{Fre:measuretheory4}, Section 452 (see \cite{biacar:cmono} for a direct proof).

\begin{theorem}[Disintegration of measures]
\label{T:disintr}
Assume that $(R,\mathscr{R},\rho)$ is a countably generated probability space, $R = \cup_{s \in S}R_{s}$ a partition of R, $r: R \to S$ the quotient map and $\left( S, \mathscr{S},\eta \right)$ the quotient measure space. Then  $\mathscr{S}$ is essentially countably generated w.r.t. $\eta$ and there exists a unique disintegration $s \mapsto \rho_{s}$ in the following sense: if $\rho_{1}, \rho_{2}$ are two consistent disintegration then $\rho_{1,s}(\cdot)=\rho_{2,s}(\cdot)$ for $\eta$-a.e. $s$.

If $\left\{ S_{n}\right\}_{n\in \enne}$ is a family essentially generating  $\mathscr{S}$ define the equivalence relation:
\[
s \sim s' \iff \   \{  s \in S_{n} \iff s'\in S_{n}, \ \forall\, n \in \enne\}.
\]
Denoting with p the quotient map associated to the above equivalence relation and with $(L,\mathscr{L}, \lambda)$ the quotient measure space, the following properties hold:
\begin{itemize}
\item $R_{l}:= \cup_{s\in p^{-1}(l)}R_{s} = (p \circ r)^{-1}(l)$ is $\rho$-measurable and $R = \cup_{l\in L}R_{l}$;
\item the disintegration $\rho = \int_{L}\rho_{l} \lambda(dl)$ satisfies $\rho_{l}(R_{l})=1$, for $\lambda$-a.e. $l$. In particular there exists a 
strongly consistent disintegration w.r.t. $p \circ r$;
\item the disintegration $\rho = \int_{S}\rho_{s} \eta(ds)$ satisfies $\rho_{s}= \rho_{p(s)}$ for $\eta$-a.e. $s$.
\end{itemize}
\end{theorem}

In particular we will use the following corollary.

\begin{corollary}
\label{C:disintegration}
If $(S,\mathscr{S})=(X,\mathcal{B}(X))$ with $X$ Polish space, then the disintegration is strongly consistent.
\end{corollary}

\section{Setting}\label{S:setting}

We fix here the objects, notations and hypothesis that will be used throughout this note. 

$(X,d,m)$ will be a non-branching metric measure space verifying 
$\mathsf{CD}_{loc}(K,N)$ or equivalently $\mathsf{CD}^{*}(K,N)$.
The marginal measure $\mu_{0},\mu_{1} \in \mathcal{P}_{2}(X,d,m)$ are fixed 
together with $\pi \in \Pi(\mu_{0},\mu_{1})$ the optimal coupling and 
$\gammA \in \mathcal{P}(\G(X))$ the associated optimal dynamical transference plan such that 
\[
[0,1]\ni t \mapsto (e_{t})_{\sharp } \gammA = \mu_{t},   \qquad (e_{0},e_{1})_{\sharp} \gammA = \pi, 
\]
with $\mu_{t}$ geodesic in the $L^{2}$-Wasserstein space and $e_{t}$ is the evaluation map at time $t$: 
for any geodesic $\gamma \in \mathcal{G}(X)$,  $e_{t}(\gamma) = \gamma_{t}$. 
$\mathcal{P}(\G(X))$ denotes the space of probability measures over $\G(X)$, 
the space of geodesic in $X$ endowed with the uniform topology inherited as a subset of $C([0,1],X)$.
The support of $\gammA$ will be denoted with $G$. 
The evaluation map $e$ without subscript is defined on $[0,1] \times G$ by $e(s,\gamma) = \gamma_{s}$.

Moreover 
\[
\mu_{t}= \r_{t} m, \qquad \forall \in  t \in [0,1].
\]
Thanks to recent results on existence and uniqueness of optimal maps, see \cite{gigli:maps}, only one geodesic in $G$ has a given couple of points as initial and final points, that is for $\gamma \in G$
\[
(e_{0},e_{1})^{-1} \{(\gamma_{0},\gamma_{1}) \} = \{\gamma \}.
\]

Moreover by inner regularity of compact sets we can assume without loss of generality that $G$ is compact, 
\[
\r_{t} \leq M, \qquad \forall \in  t \in [0,1],
\]
and metric Brenier's Theorem holds for all $\gamma \in G$, that is 
\begin{equation}\label{E:gradientbrenier}
d(\gamma_{0},\gamma_{1}) = |D\f|_{w}(\gamma_{0}).
\end{equation}

A $d^{2}$-concave Kantorovich potential for $(\mu_{0},\mu_{1})$ is $\f$ and
$\f_{t}$ will be the $d^{2}$-concave Kantorovich potential for $(\mu_{t},\mu_{1})$ obtained through Theorem \ref{T:potential}. 
When it will be needed, we will prefer the notation $\f_{0}$ to $\f$.
Thanks to compactness of $G$ we can also assume $\f$ to be Lipschitz. Its $d^{2}/2$-transform will be denoted by $\f^{c}$.
From Corollary \ref{C:1d-evo} it follows that $\f_{1} = -\f^{c}$ $\mu_{1}$-a.e. and 
\begin{equation}\label{E:dephi}
\f_{t}(\gamma_{t}) = (1-t)\f_{0}(\gamma_{0}) + t \f_{1}(\gamma_{1}).
\end{equation}
We will also use the following notation
\[
\f_{t}(\mu_{t}) = \f_{t}(\supp[\mu_{t}]), \qquad  \forall t \in [0,1].
\]

Since we will make an extensive use of the following sets, we fix their names once for all:
\begin{equation}\label{E:transportset}
\Gamma : =\left\{(x,y) \in X\times X : \f(x) + \f^{c}(y) = \frac{d^{2}(x,y)}{2} \right\}, \qquad
\end{equation}
contains the support of $\pi$ and the transportation set for $(\mu_{t},\mu_{1})$ is
\begin{equation}\label{E:transportsett}
\Gamma_{t} : \left\{ (x,y) \in X \times X : \f_{t}(x) + \f^{c}(y) = \frac{d^{2}(x,y)}{2(1-t)}\right\}.
\end{equation}
and again $(e_{t},e_{1})_{\sharp}\gammA (\Gamma_{t}) =1$.
Fix also the set of curves with starting point in $\f^{-1}(a)$:
\begin{equation}\label{E:curvesa}
G_{a} : = \big\{ \gamma \in G : \f(\gamma_{0}) = a \big\}.
\end{equation}
and the corresponding subset of $\Gamma$ 
\begin{equation}\label{E:transportseta}
\Gamma_{a} = \left\{ (x,y) \in X \times X : \f(x) + \f^{c}(y) = \frac{d^{2}(x,y)}{2}, \f(x) = a \right\} =  
\Gamma \cap \left(\f^{-1}(a) \times X \right).
\end{equation}
In Section \ref{Ss:directionmotionpartial} and Section \ref{Ss:motiongeneral} to disintegrate the reference measure $m$
in the direction of evolution, for $r\in [0,1]$ we will use the ``closed'' and ``open'' evolution sets:
\begin{equation}\label{E:evolutionset}
\bar \Gamma_{a}(r) : = e \left( [0,r] \times G_{a} \right), \qquad  
\Gamma_{a}(r) : = e \left( [0,r) \times G_{a} \right).
\end{equation}

As it will be proved in Proposition \ref{P:monotone}, the set $\Gamma_{a}$ is $d$-cyclically monotone. We will denote with 
$\phi_{a}$ a Kantorovich potential associated to it, that is $\phi_{a}$ is 1-Lipschitz function such that
\[
\Gamma_{a} \subset \{ (x,y)\in X \times X : \phi_{a}(x) - \phi_{a}(y) = d(x,y) \}.
\]
The $d$-monotone set associated to $\phi_{a}$ will be used again so we will denote it with $K_{a}$: 
\begin{equation}\label{E:dmonotoneset}
K_{a} : =\{ (x,y)\in X \times X : \phi_{a}(x) - \phi_{a}(y) = d(x,y) \}.
\end{equation}

A relevant function for the analysis is the length map at time $t$: for $t \in [0,1]$ the map $L_{t} : e_{t}(G) \to (0,\infty)$
is defined by
\[
L_{t}(x) : = L(e_{t}^{-1}(x)).
\]
Again by inner regularity of compact sets, we can assume that there exists a positive constant $C$ such that 
\[
\frac{1}{C} < L(\gamma) < C, \qquad \forall \gamma \in G.
\]

In order to study the behavior of the evolution after time $t$ of the level sets of $\f$, i.e. 
$\{ \gamma_{t} : \gamma \in G, \f(\gamma_{0})=a\}$ for $a \in \erre$, is convenient to see them as level set of a particular function.  
As it will be proven during this note this particular function is defined by
\begin{equation}\label{E:levelt}
e_{t}(G) \ni \gamma_{t} \mapsto \Phi_{t}(\gamma_{t}) : = \f_{t}(\gamma_{t}) + \frac{t}{2}L^{2}_{t}(\gamma_{t}),
\end{equation}
where in the definition of $\Phi_{t}$ we used that, for $t\in (0,1]$, for every $x \in e_{t}(G)$ there exists only  one geodesic $\gamma \in G$ 
with $\gamma_{t} =x$.  This property for $t=1$ holds only if $\mu_{1}\ll m$. Another possible definition is $\Phi_{t}(\gamma_{t}) : = \f(\gamma_{0})$, 
see \eqref{E:dephi}.

The map $\Phi_{t}$ enjoys the next monotonicity property.

\begin{lemma}\label{L:monotony}
Let $\gamma \in G$ be fixed. Then for every $s> 0$ it holds that 
\[
\Phi_{t}(\gamma_{t-s}) > \Phi_{t}(\gamma_{t}) > \Phi_{t}(\gamma_{t+s}),
\]
provided $\gamma_{t-s} \in e_{t}(G)$ for the first inequality and $\gamma_{t+s} \in e_{t}(G)$ for the second one.
\end{lemma}

\begin{proof}
We first prove the first inequality.
Suppose by contradiction the existence of $s>0$ such that $\gamma_{t-s} \in e_{t}(G)$ and
$\Phi_{t}(\gamma_{t-s}) \leq \Phi_{t}(\gamma_{t})$. From Proposition \ref{P:monotone}, necessarily
$\Phi_{t}(\gamma_{t-s}) < \Phi_{t}(\gamma_{t})$. So let $\hat \gamma : = e_{t}^{-1}(\gamma_{t-s})$, then the previous inequality reads as 
\[
\f(\hat \gamma_{0}) < \f(\gamma_{0}).
\] 
So we can deduce
\[
\frac{1}{2t}d^{2}(\hat \gamma_{0},\gamma_{t-s}) = \f(\hat \gamma_{0}) + \f^{c}_{t}(\gamma_{t-s})  
<  \f(\gamma_{0}) + \f^{c}_{t}(\gamma_{t-s})  \leq \frac{1}{2t}d^{2}(\gamma_{0},\gamma_{t-s}),
\]
and therefore $d (\hat \gamma_{0},\gamma_{t-s})< d(\gamma_{0},\gamma_{t-s})$.
Hence
\begin{align*}
d^{2}(\gamma_{0},\gamma_{t-s}) + d^{2}(\hat \gamma_{0},\gamma_{t}) \leq &~
d^{2}(\gamma_{0},\gamma_{t-s}) + \big( d(\hat \gamma_{0},\gamma_{t-s}) + d(\gamma_{t-s},\gamma_{t}) \big)^{2} \crcr
= &~ d^{2}(\gamma_{0},\gamma_{t-s}) +  d^{2}(\hat \gamma_{0},\gamma_{t-s}) + d^{2}(\gamma_{t-s},\gamma_{t}) + 2d(\hat \gamma_{0},\gamma_{t-s}) 
d(\gamma_{t-s},\gamma_{t}) \crcr
<&~ d^{2}(\gamma_{0},\gamma_{t-s}) +  d^{2}(\hat \gamma_{0},\gamma_{t-s}) + d^{2}(\gamma_{t-s},\gamma_{t}) + 2d(\gamma_{0},\gamma_{t-s}) 
d(\gamma_{t-s},\gamma_{t}) \crcr
= &~ d^{2}(\gamma_{0},\gamma_{t}) + d^{2}(\hat \gamma_{0},\gamma_{t-s}),
\end{align*}
and since $\hat \gamma_{t} = \gamma_{t-s}$ and $\hat \gamma \in G$, this is in contradiction with $d^{2}$-cyclical monotonicity.
The proof of the other inequality follows in the same way.
\end{proof}

Another important set for our analysis is the following one: for $\gamma \in G$ and $t\in (0,1)$
\begin{equation}\label{E:appartiene}
I_{t}(\gamma) := \{\tau  \in (0,1) : \gamma_{\tau} \in e_{t}(G) \},
\end{equation}
that is the set of $\tau$ for which $\gamma_{\tau}$ belongs to $e_{t}(G)$. 
A priori one can only say that $t$ belongs to $I_{t}(\gamma)$ but actually 
the set $I_{t}(\gamma)$ has sufficiently many points in a neighborhood of $t$.
The following Lemma proves a density result and 
it has been obtained in collaboration with Martin Huesmann in \cite{cavhue:regular}.
\begin{lemma}\label{L:density}
For $\mathcal{L}^{1}$-a.e. $t \in [0,1]$, 
\[
\lim_{\ve \to 0} \frac{1}{2\ve}\mathcal{L}^{1}(I_{t}(\gamma) \cap (t-\ve, t+\ve )) =1, \qquad \textrm{in } L^{1}(G,\gamma). 
\]
That is, 
the point $\tau=t$ is a point of Lebesgue density (in $L^{1}$ sense) 1 for the set $I_{t}(\gamma):=\{ \tau \in [0,1] : \gamma_{\tau} \in e_{t}(G) \}$. 
\end{lemma}

\section{On the metric structure of optimal transportation}\label{S:dmonotone}
Only for this Section the setting will be more general than the one specified in Section \ref{S:setting}. 
Here we drop all the assumption on the curvature of the space. So $(X,d,m)$ is a geodesic, non branching and separable metric measure space, 
$\mu_{t}$ is geodesic in the $L^{2}$-Wasserstein space together with a family of Kantorovich potential $\f_{t}$ for $t\in [0,1]$ associated to it. 
We will use the notation of Section \ref{S:setting} for everything and related to this objects.

Fix $a \in \f(\mu_{0})$. We will prove that $\Gamma_{a}$ is $d$-cyclically monotone. Recall that 
\[
\Gamma_{a} : = \Gamma \cap \f^{-1}(a) \times X,  
\]
with $\Gamma$ transport set for $(\mu_{0},\mu_{1})$ as from \eqref{E:transportset}.

\begin{proposition}\label{P:monotone}
The set $\Gamma_{a}$ is $d$-cyclically monotone.
\end{proposition}

\begin{proof}
Let $(x_{i},y_{i}) \in \Gamma_{a}$ for $i=1,\dots,n$ and observe that
\[
\frac{1}{2}d^{2}(x_{i},y_{i}) = \f(x_{i}) +\f^{c}(y_{i})   = \f(x_{i-1}) +\f^{c}(y_{i}) \leq \frac{1}{2}d^{2}(x_{i-1},y_{i}).
\]
Hence $d (x_{i},y_{i}) \leq d(x_{i-1},y_{i}) $ and therefore 
\[
\sum_{i=1}^{n}d(x_{i},y_{i})\leq \sum_{i=1}^{n}d(x_{i},y_{i+1})
\]
and the claim follows.
\end{proof}

The main consequence of Proposition \ref{P:monotone} is that two distinct geodesic of $G$, starting from the same level set of $\f$, can meet only for $t=0$ or $t=1$, provided the metric Brenier's Theorem holds (in the sense of \eqref{E:gradientbrenier}). Recall the definition of 
\[
\bar \Gamma_{a}(1) = e ([0,1]\times G_{a}).
\]
already introduced in \eqref{E:evolutionset} and 
$G_{a}$ the set of geodesics starting from the level set $a$ of $f$, see \eqref{E:curvesa}.

\begin{lemma}\label{L:partition} 
If the metric Brenier's Theorem holds, then the family $\{e_{t}(G_{a}) \}_{t\in [0,1]}$ is a partition of $\bar \Gamma_{a}(1)$.
\end{lemma}

\begin{proof}
By construction the family covers $\bar \Gamma_{a}(1)$, so we have only to show that overlapping doesn't occur.
Assume by contradiction the existence of $\hat \gamma, \tilde \gamma \in G_{a}$, $\hat \gamma \neq \tilde \gamma$ such that 
$\hat \gamma_{s} =\tilde \gamma_{t} = z$ with, say, $s <  t$. 

Then $d$-cyclical monotonicity implies that $\hat \gamma$ and $\tilde \gamma$ form a cycle of zero cost 
and then non-branching property of $(X,d,m)$ implies that 
they are contained in a longer geodesic:  if $\hat \gamma_{0} = x_{0}, \hat \gamma_{1} = y_{0}$ and $\tilde \gamma_{0} = x_{1}, \tilde \gamma_{1} = y_{1}$
then
\[
d(x_{0},y_{1}) + d(x_{1},y_{0}) = d (x_{0},y_{0}) + d(x_{1},y_{1}).
\] 
There are two possible cases: or $d(x_{1},y_{0}) \leq d(x_{0},y_{0})$ or  $d(x_{0},y_{1})\leq d(x_{1},y_{1})$, indeed if both were false we would have 
a contradiction with the previous identity.
In the first case 
\[
\frac{1}{2}d^{2}(x_{0},y_{0}) = \f(x_{0}) + \f^{c}(y_{0}) =  \f(x_{1}) + \f^{c}(y_{0}) \leq \frac{1}{2} d^{2}(x_{1},y_{0}) \leq \frac{1}{2} d^{2}(x_{0},y_{0}).
\]
Therefore $d(x_{0},y_{0}) = d(x_{1},y_{0})$ and since they lie on the same geodesic $x_{0} = x_{1}$. 
In the second case
\[
\frac{1}{2}d^{2}(x_{1},y_{1}) = \f(x_{1}) + \f^{c}(y_{1}) =  \f(x_{0}) + \f^{c}(y_{1}) \leq \frac{1}{2} d^{2}(x_{0},y_{1}) \leq \frac{1}{2} d^{2}(x_{1},y_{1}),
\]
and the same conclusion holds true: $x_{0} = x_{1}$.

Hence we have $(x_{0},y_{0}), (x_{0},y_{1}) \in \Gamma_{a}$. 
It follows from metric Brenier's Theorem (Proposition \ref{P:speed}) that for all $\gamma \in G$ 
\[
|D\f|_{w}(x) = d(\gamma_{0},\gamma_{1}). 
\]
Therefore necessarily $y_{0}=y_{1}$. Since $\hat \gamma, \tilde \gamma$ have also an inner common point, they must coincide implying a contradiction.
\end{proof}

The next is a simple consequence of Lemma \ref{L:partition}.

\begin{corollary}\label{C:isomorphism}
For each $a\in \erre$, the map $e: [0,1] \times G_{a} \to X$ defined by 
\[
e(s,\gamma) : = \gamma_{s}
\]
is a measurable isomorphism.
\end{corollary}

The following is, to our knowledge, a new result and it proves that for $t \in (0,1)$ the Kantorovich potentials $\f_{t}$,  
obtained with the Hopf-Lax formula from any Kantorovich potential $\f_{0}$, verifies a property similar to the point wise metric Brenier's Theorem.

\begin{proposition}\label{P:MBT}
For every $t \in (0,1)$ and for every $\gamma \in G$
\begin{equation}\label{E:gradgeo}
\lim_{s \to 0} \frac{\f_{t}(\gamma_{t}) - \f_{t}(\gamma_{t+s})}{d(\gamma_{t}, \gamma_{t+s})} = d(\gamma_{0},\gamma_{1}) = |D\f_{t}|(\gamma_{t}),
\end{equation}
where $|D\f_{t}|$ denotes the local Lipschitz constant of $\f_{t}$.
\end{proposition}

\begin{proof}
{\it Step 1.} Fix $\gamma \in G$. Observe that the set 
\[
\textrm{argmin} \Big\{y\mapsto \frac{d^{2}(\gamma_{t},y)}{2(1-t)} - \f^{c}(y) \Big\},
\]
is single valued and contains only $\gamma_{1}$. Indeed suppose by contradiction the contrary. Then there exists $z\in X$ different from $\gamma_{1}$
so that 
\[
\f_{t}(\gamma_{t}) + \f^{c}(z)  = \frac{d^{2}(\gamma_{t},z)}{2(1-t)},
\]
then since $\f_{t} = -\f_{t}^{c}$ we have 
\begin{align*}
\frac{1}{2}d^{2}(\gamma_{0},z) \geq &~ \f(\gamma_{0}) + \f^{c}(z) \crcr
= &~ \f(\gamma_{0}) -\f_{t}(\gamma_{t}) +\f_{t}(\gamma_{t}) + \f^{c}(z)\crcr
= &~ \frac{1}{2} \left(  \frac{d^{2}(\gamma_{0},\gamma_{t})}{t} +  \frac{d^{2}(\gamma_{t},z)}{1-t} \right) \crcr
\geq &~ \frac{1}{2} d^{2}(\gamma_{0},z),
\end{align*}
then necessarily $d(\gamma_{0},z) = d(\gamma_{0},\gamma_{t}) + d(\gamma_{t},z)$. But then non-branching property of $(X,d,m)$ implies 
a contradiction and then $z = \gamma_{1}$.

{\it Step 2.} Then by Hopf-Lax formula for Hamilton-Jacobi equations on length spaces 
\[
\limsup_{y \to \gamma_t} \frac{|\f_{t}(y) - \f_{t}(\gamma_{t})|}{d(y,\gamma_{t})} = \frac{D^{+}(\gamma_{t},1-t)}{1-t}
\]
see Proposition 3.6 in \cite{ambrgisav:heat}. Hence from {\it Step 1.} it follows that 
\[
\limsup_{y \to \gamma_t} \frac{|\f_{t}(y) - \f_{t}(\gamma_{t})|}{d(y,\gamma_{t})} = d(\gamma_{0},\gamma_{1}).
\]
To conclude the proof observe that 
\begin{align}\label{E:useful}
\f_{t}(\gamma_{t}) - \f_{t}(\gamma_{t+s}) = &~ \f_{t}(\gamma_{t}) +\f^{c}(\gamma_{1}) - \f^{c}(\gamma_{1}) - \f_{t}(\gamma_{t+s})  \crcr
\geq &~ \frac{1}{2(1-t)} \left( d^{2}(\gamma_{t},\gamma_{1})  - d^{2}(\gamma_{t+s},\gamma_{1}) \right) \crcr
= &~ \frac{1}{2(1-t)} \left( d(\gamma_{t},\gamma_{1})  - d(\gamma_{t+s},\gamma_{1}) \right)   \left( d(\gamma_{t},\gamma_{1})  + d(\gamma_{t+s},\gamma_{1}) \right) \crcr
= &~ \frac{1}{2(1-t)} d(\gamma_{t},\gamma_{t+s})     \left( d(\gamma_{t},\gamma_{1})  + d(\gamma_{t+s},\gamma_{1}) \right).
\end{align}
Hence
\[
\liminf_{s\to 0} \frac{\f_{t}(\gamma_{t}) -\f_{t}(\gamma_{t+s})}{d(\gamma_{t},\gamma_{t+s})} \geq  \frac{d(\gamma_{t},\gamma_{1})}{1-t} = 
d(\gamma_{0},\gamma_{1})
\]
and the claim follows.
\end{proof}

As a consequence of Proposition \ref{P:MBT}, the construction presented so far is purely metric. Indeed instead of analyzing the geometric properties 
of the Wasserstein geodesic $[0,1] \ni t \mapsto \mu_{t}$ one could restrict the domain of $\mu_{t}$ to $[\ve,1-\ve]$ for any $\ve >0$ and 
Lemma \ref{L:partition} is true without assuming any curvature bound on the space $(X,d,m)$.

\subsection{From $L^{2}$-geodesics to $L^{1}$-geodesics}

Thanks to the properties proved so far we can construct a link between $L^{2}$ Wasserstein geodesics and the linear structure of $d$-cyclically monotone sets.
Since the distance is finite, from $d$-monotonicity of $\Gamma_{a}$ we deduce the existence a $1$-Lipschitz function $\phi_{a}: X \to \erre$ so that
\[
\Gamma_{a} \subset K_{a}: = \left\{ (x,y) \in X \times X : \phi_{a}(x) - \phi_{a}(y) = d(x,y) \right\}.
\]
Note that also the following inclusion holds
\[
\{ (\gamma_{s},\gamma_{t}) :  \gamma\in G_{a}, s \leq t \} \subset K_{a}.
\]

\begin{remark}
Even if an explicit expression of $\phi_{a}$ is not strictly needed to our analysis, 
for the sake of completeness, a possible choice of $\phi_{a}$ is the following one: 
\[
e([0,1], G_{a})= \bar \Gamma_{a}(1) \ni \gamma_{s} \mapsto \phi_{a} (\gamma_{s}) = a - d(\gamma_{0},\gamma_{s}) = a - s L(\gamma).
\]
Indeed all the geodesics in $G_{a}$ follows at time $0$ the direction of $\nabla \f$ and therefore they are somehow orthogonal to $\f^{-1}(a)$. The same geodesics of $G_{a}$ follows also the steepest descent direction of $\phi_{a}$ and therefore they 
have the same direction of $\nabla \phi_{a}$. Hence one would expect that at time $0$ the set $\f^{-1}(a)$ 
is a level set also for $\phi_{a}$: therefore one could expect $\phi_{a} (\gamma_{s}) = a - d(\gamma_{0},\gamma_{s})$.

We now prove that this heuristic motivation make sense. If $0 \leq s \leq t \leq 1$ and $\gamma \in G_{a}$
\[
\phi_{a}(\gamma_{s}) - \phi_{a}(\gamma_{t}) = (t -s) L(\gamma)  = d(\gamma_{s}, \gamma_{t}),
\]
and therefore 
\[
(\gamma_{s},\gamma_{t}) : \gamma \in G_{a}, 0\leq s \leq t \leq1 \} \subset K_{a}.
\]
If $ s,  t \in [0,1]$ and $\gamma ,\hat \gamma \in G_{a}$
\[
\phi_{a}(\gamma_{s}) - \phi_{a}(\hat \gamma_{t}) = d(\hat \gamma_{0}, \hat \gamma_{t}) - d( \gamma_{0}, \gamma_{s}),
\]
and since 
\[
\frac{1}{2}d(\hat \gamma_{0},\hat \gamma_{t})^{2} = \f(\hat \gamma_{0}) + \f_{t}(\hat \gamma_{t}) = 
\f(\gamma_{0}) + \f_{t}(\hat  \gamma_{t})\leq \frac{1}{2} d(\gamma_{0},\hat \gamma_{t})^{2},
\]
it follows that 
\[
\phi_{a}(\gamma_{s}) - \phi_{a}(\hat \gamma_{t}) \leq d(\gamma_{0}, \hat \gamma_{t}) - d( \gamma_{0}, \gamma_{s}) \leq 
d( \gamma_{s},\hat \gamma_{t}).
\]
Hence $\phi_{a}$ is $1$-Lipschitz and therefore it is a good $L^{1}$-Kantorovich potential for the $d$-monotone set 
$\{ (\gamma_{s},\gamma_{t}) :  \gamma\in G_{a}, s \leq t \}$. Note moreover that the calculations above proves that 
\[
\bar \Gamma_{a}(1) \ni \gamma_{s} \mapsto d(\gamma_{0},\gamma_{s})
\]
is $1$-Lipschitz and coincides with $\gamma_{s} \mapsto \min\{d(\hat \gamma_{0},\gamma_{s}) :\hat \gamma \in G_{a} \}$.
\end{remark}

The following holds.
\begin{lemma}\label{L:12monotone}
Let $\Delta \subset K_{a}$ be any set so that: 
\[
(x_{0},y_{0}), (x_{1},y_{1}) \in \Delta \quad \Rightarrow \quad (\phi_{a}(y_{1}) - \phi_{a}(y_{0}) )\cdot (\phi_{a}(x_{1}) - \phi_{a}(x_{0}) ) \geq 0.
\]
Then $\Delta$ is $d^{2}$-cyclically monotone.
\end{lemma}

\begin{proof} It follows directly from the hypothesis of the Lemma that the set
\[
\{ (\phi_{a}(x), \phi_{a}(y) ) :   (x,y) \in \Delta \} \subset \erre \times \erre
\]
is $|\cdot|^{2}$-cyclically monotone, where $|\cdot|$ denotes the modulus. 
Then for $\{(x_{i},y_{i})\}_{ i \leq N} \subset \Delta$, since $\Delta \subset K_{a}$,
it holds
\begin{align*} 
\sum_{i=1}^{N} d^{2}(x_{i},y_{i}) = &~ \sum_{i =1}^{N}|\phi_{a}(x_{i}) - \phi_{a}(y_{i})|^{2} \crcr
\leq&~ \sum_{i =1}^{N}|\phi_{a}(x_{i}) - \phi_{a}(y_{i+1})|^{2} \crcr
\leq &~ \sum_{i=1}^{N} d^{2}(x_{i},y_{i+1}),
\end{align*}
where the last inequality is given by the 1-Lipschitz regularity of $\phi_{a}$. The claim follows.
\end{proof}

Fix an interval $(a_{0},b_{0}) \subset \erre$ and for any $\gamma$ so that $(a_{0},b_{0}) \subset \phi_{a}(\gamma_{[0,1]})$
we can define $R^{\gamma}_{0}, L^{\gamma}_{0} \subset [0,1]$ so that 
\[
\phi_{a}\circ \gamma \left((R^{\gamma}_{0}, R^{\gamma}_{0} + L^{\gamma}_{0} ) \right)= (a_{0},b_{0}) 
\]
that is equivalent to say
\[
\phi_{a}\circ \gamma ( R^{\gamma}_{0} ) = b_{0}, \qquad \phi_{a}\circ \gamma ( R^{\gamma}_{0} +L^{\gamma}_{0}) = a_{0}.
\]
In the same manner, for another interval $(a_{1},b_{1}) \subset \erre$ we can associate to any $\gamma$ so 
that $(a_{1},b_{1}) \subset \phi_{a}(\gamma_{[0,1]})$ the corresponding time interval $(R^{\gamma}_{1},R^{\gamma}_{1}+ L^{\gamma}_{1})$
Accordingly for all $t \in [0,1]$ we define 
$R^{\gamma}_{t}:= (1-t) R^{\gamma}_{0} + tR^{\gamma}_{1}$ and $L^{\gamma}_{t}:= (1-t)L^{\gamma}_{0} + tL^{\gamma}_{1}$.

We use these coefficients to construct an $L^{2}$-Wasserstein geodesic.

\begin{proposition}\label{P:w2geo}
Let $H \subset G_{a}$ be so that for all $\gamma \in H$ both $(a_{0},b_{0}) , (a_{1},b_{1}) \subset \phi_{a}(\gamma_{(0,1)})$
with $b_{0} > a_{0} > b_{1} > a_{1}$. Define the curve
\begin{equation}\label{E:curve}
[0,1] \ni t \mapsto \nu_{t}  : = \int_{H}   \frac{1}{L^{\gamma}_{t}} 
\mathcal{L}^{1}\llcorner_{[R^{\gamma}_{t},R^{\gamma}_{t} + L^{\gamma}_{t}]}  \eta(d\gamma) \in \mathcal{P}([0,1] \times G_{a}),
\end{equation}
 with $\eta$ probability measure on $\G(X)$ so that $\eta(H)=1$. Then $[0,1] \ni t \mapsto (e)_{\sharp} \nu_{t}$ is a  $W_{2}$-geodesic.
\end{proposition}

\begin{proof}
First note that for any fixed $s \in [0,1]$ the value $\phi_{a} (\gamma_{R^{\gamma}_{0} + s  L^{\gamma}_{0} })$ do not depend on $\gamma \in H$. Indeed
\begin{align*}
\phi_{a} (\gamma_{R^{\gamma}_{0} + s  L^{\gamma}_{0} }) = &~ \phi_{a} (\gamma_{R^{\gamma}_{0}}) - 
d(\gamma_{R^{\gamma}_{0}}, \gamma_{R^{\gamma}_{0} + s  L^{\gamma}_{0} }) \crcr
= &~\phi_{a} (\gamma_{R^{\gamma}_{0}}) - 
s d(\gamma_{R^{\gamma}_{0}}, \gamma_{R^{\gamma}_{0} +  L^{\gamma}_{0} }) \crcr
= &~ \phi_{a} (\gamma_{R^{\gamma}_{0}}) - s \left( \phi_{a}(\gamma_{R^{\gamma}_{0}}) - \phi_{a}(\gamma_{R^{\gamma}_{0} + L^{\gamma}_{0}  }) \right) \crcr
= &~ b_{0} - s (b_{0} - a_{0}),
\end{align*}
and the same applies for $\phi_{a} (\gamma_{R^{\gamma}_{1} + s  L^{\gamma}_{1} })$.
It follows that the set 
\[
\left\{  (\gamma_{R^{\gamma}_{0} + s L^{\gamma}_{0}}, \gamma_{R^{\gamma}_{1} + s L^{\gamma}_{1}}) : \gamma \in H, s\in [0,1]  \right\} 
\]
is $d^{2}$-cyclically monotone. Indeed, using Lemma \ref{L:12monotone}, we have only to show that for any $\hat \gamma,\gamma \in H$ and 
$\hat s, s \in [0,1]$:
\[
(\phi_{a}(\hat \gamma_{R^{\hat \gamma}_{1} + \hat s L^{\hat \gamma}_{1} }) - \phi_{a}( \gamma_{R^{\gamma}_{1} +  s L^{ \gamma}_{1} }) )
\cdot (\phi_{a}(\hat \gamma_{R^{\hat \gamma}_{0} + \hat s L^{\hat \gamma}_{0} }) - \phi_{a}( \gamma_{R^{\gamma}_{0} +  s L^{ \gamma}_{0} }) ) \geq 0.
\]
But as observed few lines above
\[
\phi_{a}( \gamma_{R^{\gamma}_{1} +  s L^{ \gamma}_{1} }) = \phi_{a}( \hat \gamma_{R^{\hat \gamma}_{1} +  s L^{\hat  \gamma}_{1} }),  \qquad
\phi_{a}( \gamma_{R^{\gamma}_{0} +  s L^{ \gamma}_{0} }) = \phi_{a}( \hat \gamma_{R^{\hat \gamma}_{0} +  s L^{\hat  \gamma}_{0} }).
\]
Hence the claim is equivalent to 
\[
(\phi_{a}(\hat \gamma_{R^{\hat \gamma}_{1} + \hat s L^{\hat \gamma}_{1} }) - \phi_{a}(\hat \gamma_{R^{\hat \gamma}_{1} +  s L^{ \hat \gamma}_{1} }) )
\cdot (\phi_{a}(\hat \gamma_{R^{\hat \gamma}_{0} + \hat s L^{\hat \gamma}_{0} }) - \phi_{a}(\hat \gamma_{R^{\hat \gamma}_{0} +  s L^{ \hat \gamma}_{0} }) ) \geq 0,
\]
that in turn is equivalent to
\[
(s - \hat s) (b_{1} - a_{1})  \cdot (s - \hat s) (b_{0} - a_{0}) = (s - \hat s)^{2} (b_{1} - a_{1})(b_{0}-a_{1}) \geq 0,  
\]
hence the claim follows.
\end{proof}

Hence, an optimal transport is achieved by not changing the ``angular'' parts and coupling radial parts according to 
optimal coupling on $\erre$. 
Since in the radial (or linear) part of the coupling is linear, one is allowed to rescale the radial speed 
and gain one degree of freedom.

In the next Sections we will use regularity properties of $\mathsf{CD}_{loc}$-spaces to properly apply the constructions of this Section to improve the curvature estimates and to study the globalization problem.

\section{Dimension reduction for a class of optimal transportations}\label{S:particular}

In this Section we start our general analysis in the particular case of optimal transport plan with lengths of geodesics 
depending only on the level set of $\f$ from where they start, that is 
\[
L(\gamma) = f(\f(\gamma_{0})), \qquad \forall \gamma \in G,
\]
and the level sets of $\f$ maintain their order during the evolution, that is 
\[
\f(\mu_{0}) \ni a \mapsto a - \frac{1}{2}f^{2}(a) \in \erre,
\]
is non decreasing. In what follows we will denote with $F(a)$ the function $a - f^{2}(a)/2$.
Thanks to Luzin's Theorem, we can also assume the map $e_{0}(G) \ni \gamma_{0} \mapsto f(\f(\gamma_{0}))$ to be continuous.

%
%

Under this particular assumption, the transportation enjoys nice properties.
In the following Lemma we prove that level sets are moved by $\gammA$ in a monotone way.
Recall that 
\[
e_{t}(G_{a}) = \{ \gamma_{t} : \gamma \in G, \f(\gamma_{0}) = a \}.
\]

\begin{lemma}\label{L:levelset}
Assume that $L(\gamma) = f(\f(\gamma_{0}))$, then it holds
\[
e_{t}(G_{a})= \f_{t}^{-1}\left(  a - \frac{t}{2}f^{2}(a)  \right) \cap e_{t}(G).
\]
for $a \in \f(\mu_{0})$.
\end{lemma}
\begin{proof}
The first inclusion follows immediately from Corollary \ref{C:1d-evo}: if $\f(\gamma_{0})=a$ then
\[
\f_{t}(\gamma_{t}) = a - \frac{t}{2}L^{2}(\gamma) =  a - \frac{t}{2}f^{2}(a).
\]

To prove the other inclusion we observe that the evolution at time $t$ of two different level sets of $\f_{0}$ cannot be contained in the same level set of $\f_{t}$. Indeed if $a > b \in \f(\mu_{0})$ then for all $\gamma \in G_{a}$ and $\bar \gamma \in G_{b}$ 
it holds 
\[
\f_{1}(\gamma_{1}) = - \f^{c}(\gamma_{1}) = F(a) \geq  F(b) = - \f^{c}(\bar \gamma_{1})  = 
\f_{1}(\bar \gamma_{1}).
\]
Hence for all $t \in (0,1)$
\[
\f_{t}(\gamma_{t}) = (1-t) \f(\gamma_{0}) + t \f_{1}(\gamma_{1}) > (1-t) \f(\bar \gamma_{0}) + t \f_{1}(\bar \gamma_{1}) 
= \f_{t}(\bar \gamma_{t}).
\]
The claim follows.
\end{proof}

\subsection{Level sets of Kantorovich potentials }\label{Ss:Kantor}

On the set $e_{0}(G)$ we will consider the partition given by the saturated sets of $\f$, i.e. $\{ \f^{-1}(a) \}_{a \in \erre}$.
Disintegration Theorem implies that 
\[
m \llcorner_{e_{0}(G)} = \int_{\f(\mu_{0})} \tilde m_{a} q(da), \qquad \f_{\sharp}\left( m\llcorner_{e_{0}(G)} \right)= q,
\]  
with $\tilde m_{a}( \f^{-1}(a)^{c} )=0$ for $q$-a.e. $a \in \f(\mu_{0})$, where,
in order to have a shorter notation, we have denoted by $\f(\mu_{0})$ the set  $\f (\supp[\mu_{0}])$.

\begin{proposition}\label{P:quotient}
The measure $q = \f_{\sharp}\left( m\llcorner_{e_{0}(G)} \right)$ is absolute continuous w.r.t. $\mathcal{L}^{1}$. 
Moreover for $q$-a.e. $a \in \f(\mu_{0})$ it holds
\[
\tilde m_{a} \ll \mathcal{S}^{h},
\]
where $\S^{h}$ is the spherical Hausdorff measure of codimension one.
\end{proposition}

\begin{proof}

{\it Step 1.} Recall that on $G$ the point wise metric Brenier's Theorem holds
true: $|D\f|_{w}(\gamma_{0}) = d(\gamma_{0},\gamma_{1})$, for all $\gamma \in G$.
Define the map
\[
e_{0}(G) \ni x \mapsto \hat \f(x) : = \inf_{y \in e_{1}(G) } \left\{  \frac{d^{2}(x,y)}{2} -\f^{c}(y) \right\}.
\]
Since $G$ is compact, $e_{0}(G)$ and $e_{1}(G)$ are bounded and 
$\hat \f$ is obtained as the infimum of Lipschitz maps with uniformly bounded Lipschitz constant.
Therefore $\hat \f$ is Lipschitz and coincide with $\f(x)$. Extend $\hat \f$ to the whole space keeping the same Lipschitz constant.

We can use the Coarea formula (see Section \ref{Ss:coarea}) in the particular case of  Lipschitz maps: 
for any $B \subset  X$ Borel 
\begin{equation}\label{E:coarea}
\int_{-\infty}^{+\infty} P(\{ \hat \f > a \}, B) da \geq c_{0} \int_{B} \| \nabla \hat \f \|(x) m(dx),  
\end{equation}
where $c_{0}$ is a strictly positive constant.

{\it Step 2.} For $(x,y) \in (e_{0},e_{1})(G)$, $\|\nabla \hat \f \| (x) \geq d(x,y)$. 
Indeed fix $(x,y) \in (e_{0},e_{1})(G)$, then $\hat \f (x) + \f^{c}(y) = d^{2}(x,y)/2$ and 
\[
\hat \f(x) - \hat \f (z) \geq \frac{1}{2} (d^{2}(x,y) - d^{2}(z,y)) =\frac{1}{2} (d(x,y) - d(z,y))(d(x,y) + d(z,y))
\]
Select a minimizing sequence $\rho_{n} \to 0$ for $\|\nabla \hat \f \| (x)$ and $z_{n}$ on the geodesic connecting $x$ to $y$ at distance 
$\rho_{n}$ from $x$. Then
\[
\frac{1}{\rho_{n}} \sup_{z \in B_{\rho_{n}}(x)} |\hat \f (z) - \hat \f(x)| \geq \frac{1}{2} \frac{1}{\rho_{n}} (d(x,y) - d(z_{n},y))(d(x,y) + d(z_{n},y)) 
=\frac{1}{2} (d(x,y) + d(z_{n},y)).
\]
Passing to the limit we have $\|\nabla \hat \f \| (x) \geq d(x,y)$. 

Let $E \subset \erre$ with $\mathcal{L}^{1}(E) = 0$, then  from \eqref{E:coarea} it follows that
\begin{equation}\label{E:tutto}
\int_{\f^{-1}(E) \cap e_{0}(G)} \| \nabla \hat \f \|(x) m(dx) =\int_{\hat \f^{-1}(E) \cap e_{0}(G)} \| \nabla \hat \f \|(x) m(dx) \leq  \frac{1}{c_{0}}
\int_{E} P(\{ \hat \f > a \}, e_{0}(G)) da = 0.
\end{equation}
But on $e_{0}(G)$ the gradient of $\hat \f$ is strictly positive, it follows that $m(\f^{-1}(E) \cap e_{0}(G))=0$ and therefore 
the first part of the claim is proved.
Moreover from \eqref{E:tutto} it follows that 
\[
\tilde m_{a}\leq P(\{\hat \f  > a \},\cdot). 
\]
Being the latter  absolutely continuous with respect to $\S^{h}$, also the second part of the statement follows. 
\end{proof}

\begin{remark}\label{R:general}
Proposition \ref{P:quotient} proves a property of disintegration at time $t=0$ where the particular shape 
of the optimal transportation or of the Kantorovich potentials do not play any role and indeed the proof 
is done without using any particular assumption. Hence the result will be used also in the general case.
\end{remark}

So Proposition \ref{P:quotient} implies the following decomposition for $t=0$: 
\[
m \llcorner_{e_{0}(G)} = \int_{\f(\mu_{0})} \tilde m_{a} q(a)  \mathcal{L}^{1}(da) = \int_{\f(\mu_{0})} \hat m_{a} \mathcal{L}^{1}(da),
\]
with clearly again $\hat m_{a} \ll \S^{h}$.

For $t \in [0,1)$ an analogous partition can be considered also on the support of $\mu_{t}$, $e_{t}(G)$.
Indeed the $d^{2}$-cyclical monotonicity of $\Gamma$ implies that the family
\[
\{ \gamma_{t} : \gamma \in G, \f(\gamma_{0}) = a \}_{a\in \f(\mu_{0})} = \{e_{t}(G_{a}) \}_{a \in \f(\mu_{0})}
\] 
is a disjoint family and a partition of $e_{t}(G)$. 
Therefore we consider the disintegration of $m\llcorner_{e_{t}(G)}$ w.r.t. the aforementioned family.
Since for every $t \in [0,1)$
\[
\mu_{0}(\f^{-1}(A)) = \mu_{t}(\{\gamma_{t} : \f(\gamma_{0}) \in A \}),
\]
the quotient measures of $\mu_{0}$ and $\mu_{t}$ are the same measure. 
We can conclude that the quotient measures of $m\llcorner_{e_{t}(G)}$ and of $m\llcorner_{e_{0}(G)}$ are equivalent and
\begin{equation}\label{E:mass2}
m\llcorner_{e_{t}(G)} = \int_{\f(\mu_{0})} \tilde m_{a,t} f_{t}(a)  \mathcal{L}^{1}(da) = \int_{\f(\mu_{0})} \hat m_{a,t} \mathcal{L}^{1}(da), 
\qquad \hat m_{a,t}( \{ \gamma_{t} : \f(\gamma_{0})=a \}^{c}) = 0.
\end{equation}
To keep notation consistent, we will denote also the conditional probabilities for $t=0$ with $\hat m_{a,0}$.

For $t=1$ only if $\mu_{1}$ is absolute continuous with respect to $m$ we can do the same disintegration.  
Indeed if this is the case, from Theorem 2.7 of \cite{gigli:maps}, 
for $m$-a.e. $x \in e_{1}(G)$ there is only 
one geodesic $\gamma$ in $G$ so that $\gamma_{1} = x$ and the family
\[
\{\gamma_{1} : \gamma \in G, \f(\gamma_{0}) = a \}_{a\in \f(\mu_{0})}
\]
is again partition of $e_{1}(G)$. Since we are assuming both $\mu_{0}$ and $\mu_{1}$ absolute continuous with respect to $m$, we have 
\[
m\llcorner_{e_{t}(G)} = \int_{\f(\mu_{0})} \hat m_{a,t} \mathcal{L}^{1}(da), \qquad \hat m_{a,t}(e_{t}(G_{a})) = \| \hat m_{a,t} \|,
\]
for all $t \in [0,1]$.

\begin{lemma}\label{L:cod1}
For every $t\in [0,1]$ and $\mathcal{L}^{1}$-a.e. $a \in \f(\mu_{0})$, with the exceptional set depending on $t$, 
it holds
\[
\hat m_{a,t} \ll \S^{h},
\]
where $\S^{h}$ is the spherical Hausdorff measure of codimension one defined in \eqref{E:spherical}. 
\end{lemma}
\begin{proof} 

For $t =0$ the claim has been already obtained in Proposition \ref{P:quotient}.
For $t \in (0,1]$ we observe that from Lemma \ref{L:levelset} the partition  of $e_{t}(G)$
\[
\{\gamma_{t} : \gamma \in G_{a} \}_{a \in \f(\mu_{0})},
\]
can be equivalently written as 
\[
\{ \f_{t}^{-1}(a) \}_{a \in \f_{t}(\mu_{t})}.
\]
Then using coarea formula as in Proposition \ref{P:quotient} the claim follows.
\end{proof}

Since level sets of $\f_{0}$ are moved after time $t$ to level sets of $\f_{t}$,  the monotone map $F (a) = a - f^{2}(a)/2$ is the optimal map between the quotient measures. 

\begin{lemma}\label{L:quotientgeodesic}
For each $t \in [0,1]$, consider the map $F_{t}(a) : = a -t f^{2}(a)/2$ defined on $\f(\mu_{0})$.
Then for each $t \in [0,1]$, $F_{t}$ is the optimal map for between
\[
(\f_{0})_{\sharp}\mu_{0}, \qquad (\f_{t})_{\sharp}\mu_{t}
\]
and $f$ is locally Lipschitz.
\end{lemma}

\begin{proof}
Just note that 
\begin{equation}\label{E:geoquotient}
F_{t}(\f(\gamma_{0})) = \f(\gamma_{0}) - \frac{t}{2}L^{2}(\gamma) = \f_{t}(\gamma_{t}).
\end{equation}
Since $g_{1}$ is monotone by assumption and, thanks to Proposition \ref{P:quotient},
$(\f_{i})_{\sharp}\mu_{i}$ are absolute continuous w.r.t. to $\mathcal{L}^{1}$ for $i=0,1$, the claim follows.
\end{proof}

\subsection{Disintegration in the direction of motion}\label{Ss:directionmotionpartial}

As already motivated in the Introduction, $\hat m_{a,t}$ is not the right reference measure to improve the curvature estimate
to a ``codimension 1''-like estimate.
So we consider the evolution in time of a single level set as a whole subset of $X$, that is the set 
$e([0,1]\times G_{a})$, and we disintegrate the reference measure $m$ with respect to the family $\{e_{t}(G_{a}) \}_{t \in [0,1]}$.
In this way the quotient space of the disintegration will be the time variable and as $t$ moves the conditional probabilities will move in the same direction of the optimal transportation.

Recall $\bar \Gamma_{a}(1) : = e \left( [0,1] \times G_{a} \right)$.
Thanks to Lemma \ref{L:partition}, we can consider the disintegration of $m\llcorner_{\bar \Gamma_{a}(1)}$ w.r.t.  
the family of sets $\{e_{t}(G_{a}) \}_{t\in [0,1]}$:
\[
m\llcorner_{\bar \Gamma_{a}(1)} = \int_{[0,1]} \bar m_{a,t}q(dt), \qquad  q \in \mathcal{P}([0,1]), \quad 
q(I) = m (e(I\times G_{a})).
\] 
Observe that any $\gamma \in G_{a}$ can be taken as quotient set, therefore 
Corollary \ref{C:disintegration} implies the strong consistency of the disintegration, i.e. for $q$-a.e. $t \in [0,1]$
$\bar m_{a,t}$ is concentrated on $e_{t}(G_{a})$.

\begin{proposition}\label{P:absolutecont}
The quotient measure $q_{a}$ is absolute continuous with respect to $\mathcal{L}^{1}$.
\end{proposition}

\begin{proof} 
Since $\Gamma_{a}$ is $d$-cyclically monotone, we can consider another partition of $\bar \Gamma_{a}(1)$.  
Consider the family of sets $\{\gamma_{s} : s \in [0,1] \}_{\gamma \in G_{a}}$. 
By Lemma \ref{L:partition}, we have that
\begin{itemize}
\item the following disintegration holds true:
\[
m\llcorner_{\bar \Gamma_{a}(1)} = \int \eta_{y} q_{a}(dy), 
\]
where the quotient measure $q_{a}$ is concentrated on $\{\gamma_{1/2} : \gamma \in G_{a}\}$ and $q_{a}$-a.e. 
conditional probability $\eta_{y}$ is concentrated on $\{\gamma_{s} : s\in [0,1], \gamma \in e_{1/2}^{-1}(y) \cap G_{a} \}$;
\item Since $\mathsf{CD}_{loc}(K,N)$ implies $\mathsf{MCP}(K,N)$, from Theorem 9.5 of \cite{biacava:streconv} we have that $\eta_{y} = g(y,\cdot)\mathcal{L}^{1}\llcorner_{[0,1]}$ for $q_{a}$-a.e. $y$, and for $r \leq R$
\begin{equation}\label{E:surfacevopoint}
 \Bigg( \frac{\sin \big( \frac{r}{R}d(\gamma_{0},\gamma_{R}) \sqrt{K/(N-1)} \big)}{\sin \big( d(\gamma_{0},\gamma_{R}) \sqrt{K/(N-1)} \big)} \Bigg)^{N-1}  
  \leq \frac{g(y,r)}{g(y,R)} \leq  
  \Bigg( \frac{\sin \big( \frac{r}{R}d(\gamma_{r},\gamma_{1}) \sqrt{K/(N-1)} \big)}{\sin \big( d(\gamma_{r},\gamma_{1}) \sqrt{K/(N-1)} \big)} \Bigg)^{N-1},
\end{equation}
where $\gamma = e_{1/2}^{-1}(y) \cap G_{a}$, and 
and the measure $g(y,\cdot)\mathcal{L}^{1}\llcorner_{[0,1]}$ has to be intended as 
$(\gamma)_{\sharp}(g(y,\cdot)\mathcal{L}^{1}\llcorner_{[0,1]})$, with 
$\gamma$ the unique element of $G_{a}$ so that $\gamma_{1/2} = y$.
\end{itemize}

To prove the claim it is enough to observe that the two disintegration proposed for $m\llcorner_{\bar \Gamma_{a}(1)}$ are the same.
Use Fubini's Theorem to get 
\[
\int_{[0,1]} \bar m_{a,t} q(dt) = m\llcorner_{\bar \Gamma_{a}(1)} = \int  g(y,\cdot)\mathcal{L}^{1}(dt) q_{a}(dy) = 
\int_{[0,1]} g(\cdot,t) q_{a}(dy) dt,
\]
therefore from uniqueness of disintegration, 
\[
\bar m_{a,t} = g(\cdot,t) q_{a} \left( \int g(y,t)q_{a}(dy) \right)^{-1}, \qquad
q = \left( \int g(y,t)q_{a}(dy) \right) \mathcal{L}^{1},
\]
and the claim follows.
\end{proof}

Hence if $dq/d\mathcal{L}^{1}$ denotes the density of $q$ with respect to $\mathcal{L}^{1}$, posing 
$m_{a,t} : = \left( dq/d\mathcal{L}^{1} \right) \bar m_{a,t}$, we have
\begin{equation}\label{E:L1m}
m\llcorner_{\bar \Gamma_{a}(1) } = \int_{[0,1]} m_{a,t} dt.
\end{equation}

Note that Proposition \ref{P:absolutecont}, and therefore \eqref{E:L1m}, has been obtained without using the assumption of constant speed of geodesics along the level set of $\f$. So we will use it also in the general case without any need of prove it again.

\begin{proposition}\label{P:cod2}
For $\mathcal{L}^{2}$-a.e. $(a,t) \in \f(\mu_{0})\times [0,1]$ it holds 
\[
m_{a,t} \ll \S^{h},
\]
where $\S^{h}$ is the spherical Hausdorff measure of codimension one defined in \eqref{E:spherical}. 
\end{proposition}

\begin{proof} 
Following the proof of Lemma \ref{L:cod1}, the claim will be proved 
if we write the family of sets 
\[
\{ \gamma_{t} : \gamma \in G_{a} \}_{t\in [0,1]} = \{ e_{t}(G_{a})\}_{t\in [0,1]},
\] 
as a family of level sets $\{ \Lambda^{-1}(t) \}_{t\in [0,1]}$ for some locally Lipschitz $\Lambda : \bar \Gamma_{a}(1) \to \erre$ with non zero gradient.

{\it Step 1.} Consider the evaluation map $e: [0,1] \times G_{a} \to \bar \Gamma_{a}(1)$ as $e(s,\gamma) : = \gamma_{s}$ 
and define the following function 
\[
\bar \Gamma_{a}(1) \ni x \mapsto \Lambda(x) : = P_{1} (e^{-1}(x)).
\]
Hence $\Lambda(x)$ is the unique $t$ for which there exists $\gamma \in G_{a}$ so that $\gamma_{t} = x$. From its definition, $\Lambda$ 
is clearly measurable and 
\[
\Lambda^{-1}(t) = \{\gamma_{t} : \gamma \in G_{a} \}.
\] 
Its derivative in the direction of $s \mapsto \gamma_{t+s}$ is 1 for any $t\in (0,1)$ and $\gamma \in G_{a}$.
We now show that $\Lambda$ is locally Lipschitz. Note that for $s < t$ and any $\gamma, \hat \gamma \in G_{a}$
\[
\Lambda (\gamma_{t}) - \Lambda(\hat \gamma_{s}) = t-s =\frac{1}{L(\hat \gamma)} d(\hat \gamma_{s},\hat \gamma_{t}).
\]
On the other hand from Lemma \ref{L:levelset} $\f_{s}(\gamma_{s}) = \f_{s}(\hat \gamma_{s})$ and therefore
\[
\frac{1}{(t-s)^{2}}d^{2}(\hat \gamma_{s}, \hat \gamma_{t} )=  \f_{s}(\hat \gamma_{s}) + \f^{c}_{1-t}(\hat \gamma_{t})
= \f_{s}(\hat \gamma_{s}) + \f^{c}_{1-t}( \gamma_{t}) \leq \frac{1}{(t-s)^{2}}d^{2}(\hat \gamma_{s},  \gamma_{t} ).
\]
Hence
\[
|\Lambda(\gamma_{t}) - \Lambda(\hat \gamma_{s})| \leq \frac{1}{C} d(\gamma_{t}, \hat \gamma_{s}).
\]
and therefore the claim is proved.
\end{proof}

So the results obtained in this Section are: assuming that
$L(\gamma) = f(\f(\gamma_{0}))$,
and  $a - \frac{1}{2}f^{2}(a)$ is non decreasing, we have 
\[
m\llcorner_{e_{t}(G)} = \int_{\f(\mu_{0})} \hat m_{a,t} \mathcal{L}^{1}(da), \qquad 
m\llcorner_{\bar \Gamma_{a}(1)} = \int_{[0,1]} m_{a,t} \mathcal{L}^{1}(dt)
\]
for all $t \in [0,1]$ and $a \in \f(\mu_{0})$ and $\hat m_{a,t} , m_{a,t} \ll \mathcal{S}^{h}\llcorner_{e_{t}(G_{a})}$ for 
$\mathcal{L}^{1}$-a.e. $a,t \in [0,1]$.

\section{Dimension reduction for the General transportation}\label{S:general}

In this Section we obtain the result of Section \ref{S:particular} dropping the assumption of constant length on level sets of $\f$
but assuming few regularity properties for $\gammA$.
In particular we will assume a regularity property of the length map that has been already introduce 
in Section \ref{S:setting}: for $t\in (0,1)$
\[
e_{t}(G) \ni x \mapsto L_{t}(x): =L(e_{t}^{-1}(x)) \in (0,\infty).
\]

\begin{assumption}\label{A:lengthreg}
For all $t\in (0,1)$ the map $L_{t}$ is locally Lipschitz: for $\mu_{t}$-a.e. $x \in e_{t}(G)$ there 
exists an open neighborhood $U(x)$ of $x$ and a positive constant $C$ so that 
\[
|L_{t}(z) - L_{t}(w)| \leq C d(z,w), \qquad \forall z,w \in U(x).
\]
\end{assumption}

We can now introduce the function $\Phi_{t}$ defined on $e_{t}(G)$:
\[
\Phi_{t}(x) : = \f_{t}(x) + \frac{t}{2}L^{2}_{t}(x). 
\]
As already pointed out in Section \ref{S:setting}, the relevance of $\Phi_{t}$ is explained by the following equivalent identities:
\[
\Phi_{t}(\gamma_{t})  = \f(\gamma_{0}),\qquad  \Phi_{t}^{-1}(a) = \{\gamma_{t} : \gamma \in G, \f(\gamma_{0}) = a \}.
\]
It follows from Assumption \ref{A:lengthreg} that also $\Phi_{t}$ is locally Lipschitz. Moreover almost by definition 
\[
(\Phi_{t})_{\sharp} m \llcorner_{e_{t}(G)}  \ll \mathcal{L}^{1},
\]
indeed since $\Phi_{t} \circ e_{t}  = \f \circ e_{0}$ it follows that $(\Phi_{t})_{\sharp}\mu_{t} = (\f)_{\sharp}\mu_{0}$ and therefore
\[
(\Phi_{t})_{\sharp} \r_{t} m \llcorner_{e_{t}(G)} = (\f)_{\sharp} \r_{0} m \llcorner_{e_{0}(G)}.
\]
Since $\r_{t} > 0$ on $e_{t}(G)$, also the converse is true, that is 
\[
\mathcal{L}^{1}\llcorner_{\f(\mu_{0})} \ll  (\Phi_{t})_{\sharp} m \llcorner_{e_{t}(G)}.
\]
Anyway this property is not sufficient to guarantee that its metric gradient do not vanish. 
See \cite{abc:sard} for a counter example to this property (constructed on $\erre^{2}$). 
One of the first steps we have to do is prove that the reference measures of codimension one are all absolute continuous with respect to the spherical Hausdorff measure 
$\S^{h}$, and, as in the proof of Proposition \ref{P:quotient}, we will use Coarea formula and we will apply it to the function $\Phi_{t}$. 
Since Coarea formula brings information only where the gradient is non zero we have to ask for the following property to hold.
\begin{assumption}\label{A:Phi} For all $t\in [0,1]$ for $\gammA$-a.e. $\gamma \in G$ the following holds
\[
\lim_{s \to 0} \frac{\Phi_{t}(\gamma_{t})  - \Phi_{t}(\gamma_{t+s})  }{d(\gamma_{t}, \gamma_{t+s})} \in (0,\infty).
\]
\end{assumption}

\begin{remark}\label{R:assuverified} Assumption \ref{A:lengthreg} is verified in the hypothesis of Section \ref{S:particular} that is: 
\[
L(\gamma) = f(\f(\gamma_{0})), \qquad \forall \gamma \in G, 
\]
with $f$ so that $a \mapsto a - f^{2}(a)/2$ is non-decreasing. 
Indeed as proved in Lemma \ref{L:quotientgeodesic} $f$ is locally Lipschitz. Moreover if $F_{t} : \f_{0}(\mu_{0}) \to \f_{t}(\mu_{t})$
is the locally bi-Lipschitz function of \eqref{E:geoquotient}, then
\[
e_{t}(G) \ni \gamma_{t} \mapsto  f \circ F^{-1}_{t} \circ \f_{t} (\gamma_{t})
\]
is locally Lipschitz and coincides with $L_{t}(\gamma_{t})$. 
Noticing that $\Phi_{t} = F_{t}^{-1} \circ \f_{t}$ Assumption \ref{A:Phi} is straightforward. 
\end{remark}

Before showing how Assumption \ref{A:lengthreg} and \ref{A:Phi} are used in the metric framework, we prove that if $X$ is a Riemannian manifold
with $d$ geodesic distance induced by a Riemannian tensor $g$ and $m$ is the volume measure, then 
Assumption \ref{A:lengthreg} and \ref{A:Phi} are verified.

\begin{proposition}\label{P:assumptionriemannian}
Assume $(X,d,m)$ has a Riemannian structure, that is $(X,d)$ is a Riemannian manifold with metric $g$ and $m$ is the volume measure. 
Then Assumption \ref{A:lengthreg} and \ref{A:Phi} are both verified.
\end{proposition}

\begin{proof}
Assumption \ref{A:lengthreg} follows from the Monge-Mather shortening principle, see \cite{villa:Oldnew} Theorem 8.5.

Actually Theorem 8.5 of \cite{villa:Oldnew} proves Lipschitz regularity on compact sets of the transport map from intermediate times: 
if 
\[
T_{t} : e_{0}(G) \to e_{t}(G), \qquad (T_{t})_{\sharp} \mu_{0} = \mu_{t}
\]
then for any $t \in (0,1)$ the map $T_{t}^{-1}$ is Lipschitz and in particular $m$-almost every where differentiable:
Since $\Phi_{t} = \f \circ T_{t}^{-1}$, it follows that Assumption \ref{A:Phi} is equivalent to prove that 
\[
g_{\gamma_{0}} (\nabla\f(\gamma_{0}), DT_{t}^{-1}\nabla  \f_{t}(\gamma_{t}) ) > 0.
\]
In order to compute the previous quantity is convenient to consider the expression of $DT_{t}^{-1}$ proved in \cite{corderomccann:brescamp}, 
see Theorem 4.2: 
\[
DT_{t}^{-1} = Y (H - t Hess_{x}\f_{t}^{c}),
\]
where $Y$ is the differential of the exponential map in $(\gamma_{t}, -t \nabla \f_{t}^{c}) \in X  \times T_{\gamma_{t}}X$, $H$ is the Hessian of the squared of distance function and $\f_{t}^{c}$ has been introduced in Section \ref{Ss:slopKant} and minus its gradient composed with the exponential maps gives the optimal transport from $\mu_{t}$ to $\mu_{0}$. 
Since 
\[
Y \nabla \f_{t}(\gamma_{t}) = \nabla \f(\gamma_{0})
\]
it follows from Gauss Lemma, see \cite{gallot:riemannian} Theorem 3.70, that 
\begin{equation}\label{E:equivalent}
g_{\gamma_{0}} (\nabla\f(\gamma_{0}), DT_{t}^{-1}\nabla  \f_{t}(\gamma_{t}) ) = 
g_{\gamma_{t}} (\nabla\f(\gamma_{t}),(H - t Hess_{x}\f_{t}^{c}) \nabla  \f_{t}(\gamma_{t})  ).
\end{equation}
Since $(H - t Hess_{x}\f_{t}^{c})$ is symmetric with strictly positive determinant $\mu_{t}$-almost everywhere, the claim follows.
\end{proof}

\subsection{Level sets of Kantorovich potentials}\label{Ss:levelsetgeneral}

\begin{proposition}\label{P:cod1}
For every $t\in [0,1)$ and $\mathcal{L}^{1}$-a.e. $a \in \f(\mu_{0})$, with the exceptional set depending on $t$, 
it holds
\[
\hat m_{a,t} \ll \S^{h},
\]
where $\S^{h}$ is the spherical Hausdorff measure of codimension one defined in \eqref{E:spherical}. 
\end{proposition}

\begin{proof} 
{\it Step 1.} For $t =0$ the claim has been already proved in Proposition \ref{P:quotient}, see Remark \ref{R:general}.

As a consequence of Assumption \ref{A:lengthreg}, $\Phi_{t}$ is locally Lipschitz on $e_{t}(G)$. 
Since we are proving a local property, possibly taking a compact subset of $G$,
we can assume without loss of generality that $\Phi_{t}$ is Lipschitz on the whole $e_{t}(G)$. 
Denote with $\hat \Phi_{t}$ its Lipschitz extension to $X$. 
Coarea formula for Lipschitz maps applies (see Section \ref{Ss:coarea} and references therein): 
for any measurable $A \subset X$
\[
\int_{-\infty}^{+\infty} P(\{ \hat \Phi_{t} > a \}, A) \mathcal{L}^{1}(da) \geq c_{0} \int_{A} \| \nabla \hat \Phi_{t} \|(x)m(dx),
\]
where $c_{0}$ is a strictly positive constant.

{\it Step 2.} Since $\hat \Phi_{t}$ is Lipschitz, 
\[
P(\{\hat  \Phi_{t} > a \}, \cdot) \leq c \S^{h}. 
\]
where $c$ is a positive constant depending on $K$ and $N$.
So we have
\[
\int \| \nabla \hat \Phi_{t} \| \hat m_{a,t} \mathcal{L}^{1}(da) \leq \int_{-\infty}^{+\infty} P(\{ \hat \Phi_{t} > a \}, \cdot) \mathcal{L}^{1}(da),
\]
which in turn gives 
\[
\| \nabla\hat \Phi_{t} \| \hat m_{a,t} \leq c \S^{h}.
\]
From Assumption \ref{A:Phi} it follows that $\|\nabla \hat \Phi_{t} \| >0$ on $e_{t}(G)$ and the claim follows.
\end{proof}

\subsection{Disintegration in the direction of motion}\label{Ss:motiongeneral}

As in Section \ref{Ss:directionmotionpartial}, we disintegrate $m\llcorner_{\bar \Gamma_{a}(1)}$ in with respect to the 
partition $\{e_{t}(G_{a}) \}_{t \in [0,1]}$. From Proposition \ref{P:absolutecont} we have
\[
m\llcorner_{\bar \Gamma_{a}(1)} = \int_{[0,1]} m_{a,t} \mathcal{L}^{1}(dt).
\]

We now prove a regularity property for the conditional measures $m_{a,t}$.  
Recall that we are considering optimal transportation with uniformly positive and bounded lengths: 
there exists $C > 0$ so that
\[
\frac{1}{C} < L(\gamma) < C, \qquad \forall \gamma \in G.
\]

\begin{lemma}\label{L:cod2}
For $\mathcal{L}^{2}$-a.e. $(a,t) \in \f(\mu_{0})\times [0,1]$ it holds 
\[
m_{a,t} \ll \S^{h},
\]
where $\S^{h}$ is the spherical Hausdorff measure of codimension one defined in \eqref{E:spherical}. 
\end{lemma}

\begin{proof} The idea of the proof is exactly the same as Proposition \ref{P:cod2}.

{\it Step 1.}
Define the map
\[
\bar \Gamma_{a}(1) \ni x \mapsto \Lambda(x) : = P_{1} (e^{-1}(x)),
\]
hence $\Lambda(x)$ is the unique $t$ for which there exists $\gamma \in G_{a}$ so that $\gamma_{t} = x$. From its definition, $\Lambda$ 
is clearly measurable, $\Lambda^{-1}(t) = \{\gamma_{t} : \gamma \in G_{a} \}$ and its derivative in the direction of $s \mapsto \gamma_{t+s}$ is 1 
for any $t\in (0,1)$ and $\gamma \in G_{a}$.

{\it Step 2.} We now show that $\Lambda$ is locally Lipschitz. 
Consider the map 
\[
G_{a} \times G_{a} \times [0,1]^{2} \ni (\bar \gamma,\hat \gamma,s,t) \mapsto Y(\bar \gamma, \hat \gamma,s,t) 
= d(\bar \gamma_{s}, \hat \gamma_{t}) - \frac{1}{C} |t-s|.
\]
Fix $\gamma \in G_{a}$ and note that for any $s,t \in [0,1]$:
\[
Y(\gamma,\gamma,s,t) = |t-s| \left( L(\gamma) - \frac{1}{C} \right) > 0.
\]
By continuity, there exists an open set $U$ in $G_{a} \times G_{a} \times [0,1]^{2}$ so that $Y(U) \subset (0, \infty)$ and 
\[
K : = \{(\gamma,\gamma,s,t):  s,t \in [0,1]  \} \subset U.
\]
Since $K$ is compact, there exists $\ve >0$ so that $K^{\ve} \subset U$ where $K^{\ve}$ is the $\ve$-neighborhood of $K$ 
in the metric space $G_{a}\times G_{a} \times [0,1]^{2}$. 
Consider $\bar \gamma, \hat \gamma \in G_{a}$ so that  
\[
d_{\infty} (\bar \gamma, \gamma) < \frac{\ve}{2}, \qquad  d_{\infty} (\hat \gamma, \gamma) < \frac{\ve}{2},
\]
where $d_{\infty}$ the metric on $\G(X)$. Then $(\hat \gamma, \bar \gamma, s,t) \in K^{\ve}$ for any $s,t \in [0,1]$. 
Therefore
\[
d(\hat \gamma_{s} , \bar \gamma_{t}) > \frac{1}{C}|s-t|, \qquad \forall \, s,t \in [0,1].
\]
Hence we have shown that for any $\gamma \in G_{a}$ there exists $\ve >0$ so that the map 
\[
e\left( [0,1] \times B_{\ve}(\gamma) \right) \ni x \mapsto \Lambda (x)
\]
is Lipschitz indeed for $x, y \in e\left( [0,1] \times B_{\ve}(\gamma) \right)$ with say $x = \bar \gamma_{s}$ and 
$y = \hat \gamma_{t}$ it holds
\[
|\Lambda(x) - \Lambda(y)| = |s-t| \leq C d(\bar \gamma_{s} , \hat \gamma_{t}).
\]

{\it Step 3.} Repeating the proof of Proposition \ref{P:cod1} with coarea formula we obtain that 
\[
m_{a,t}\llcorner_{e\left( [0,1] \times B_{\ve}(\gamma) \right)} \ll \S^{h}, 
\]
for $\mathcal{L}^{1}$-a.e. $t \in [0,1]$. Since $G_{a}$ is compact the claim on the whole $\bar \Gamma_{a}(1)$
follows by a covering argument.
\end{proof}

\section{Uniqueness of conditional measures}\label{S:unique}
This Section is devoted to find a relation and possibly a comparison between $m_{a,t}$ and $\hat m_{a,t}$.
Find a comparison between this two different reference measure of codimension one is fundamental.
Indeed disintegrate $\gammA$ w.r.t. $\{e_{0}^{-1} ( \f^{-1}(a) ) \}_{a \in \erre}$ that is the set of geodesic starting from a given level set of $\f$:
\[
\gammA = \int_{\f(\mu_{0})} \gammA_{a} q(a)\mathcal{L}^{1}(da), \qquad \gammA_{a}\big( (\f \circ e_{0}) ^{-1}(a)  \big) =1.
\]
Clearly this disintegration is just the lift for each $t$ of the disintegration of $\mu_{t}$ w.r.t. $\{ e_{t}(G_{a}) \}_{a \in \f(\mu_{0})}$.
Therefore the quotient measure $q(a)\mathcal{L}^{1}(da)$ is the same quotient measure of $\mu_{t}$ for every $t \in [0,1]$.
Then necessarily, 
\[
\int_{\f(\mu_{0})} \r_{t} \hat m_{a,t} \mathcal{L}^{1}(da) = \mu_{t} = (e_{t})_{\sharp} \gammA = \int_{\f(\mu_{0})} (e_{t})_{\sharp} \gammA_{a} q(a)\mathcal{L}^{1}(da),
\]
and from uniqueness of disintegration,
\[
(e_{t})_{\sharp} \gammA_{a} = \left( \int \r_{t}(z) \hat m_{a,t}(dz)\right)^{-1} \r_{t} \hat m_{a,t}.
\]
Hence if we want to express the geodesic of codimension one $(e_{t})_{\sharp} \gammA_{a}$ in terms of the reference measure 
$m_{a,t}$ moving in the same direction of the optimal transportation, we have to prove that 
$(e_{t})_{\sharp} \gammA_{a}  \ll m_{a,t}$.
To do that we will prove that $\hat m_{a,t} \ll m_{a,t}$.

\begin{figure}
\psfrag{e0}{$e_{0}(G_{a})$}
\psfrag{e1}{$e_{1/2}(G_{a})$}
\psfrag{e2}{$e_{1}(G_{a})$}
\psfrag{ma}{$m\llcorner_{\bar \Gamma_{a}}$}
\psfrag{t0}{$t=0$}
\psfrag{t1}{$t=1/2$}
\psfrag{t2}{$t=1$}
\psfrag{ma0}{$m_{a,0}$}
\psfrag{ma2}{$m_{a,1/2}$}
\psfrag{ma1}{$m_{a,1}$}
\psfrag{hm0}{$\hat m_{a,0}$}
\psfrag{hm1}{$\hat m_{a,1/2}$}
\psfrag{hm2}{$\hat m_{a,1}$}
\psfrag{m0}{$m\llcorner_{e_{0}(G)}$}
\psfrag{m1}{$m\llcorner_{e_{1/2}(G)}$}
\psfrag{m2}{$m\llcorner_{e_{1}(G)}$}
\psfrag{f0}{$\f^{-1}(a)$}
\psfrag{f1}{$\Phi_{1/2}^{-1}(a)$}
\psfrag{f2}{$\Phi_{1}^{-1}(a)$}
\centering{\resizebox{16cm}{9cm}{\includegraphics{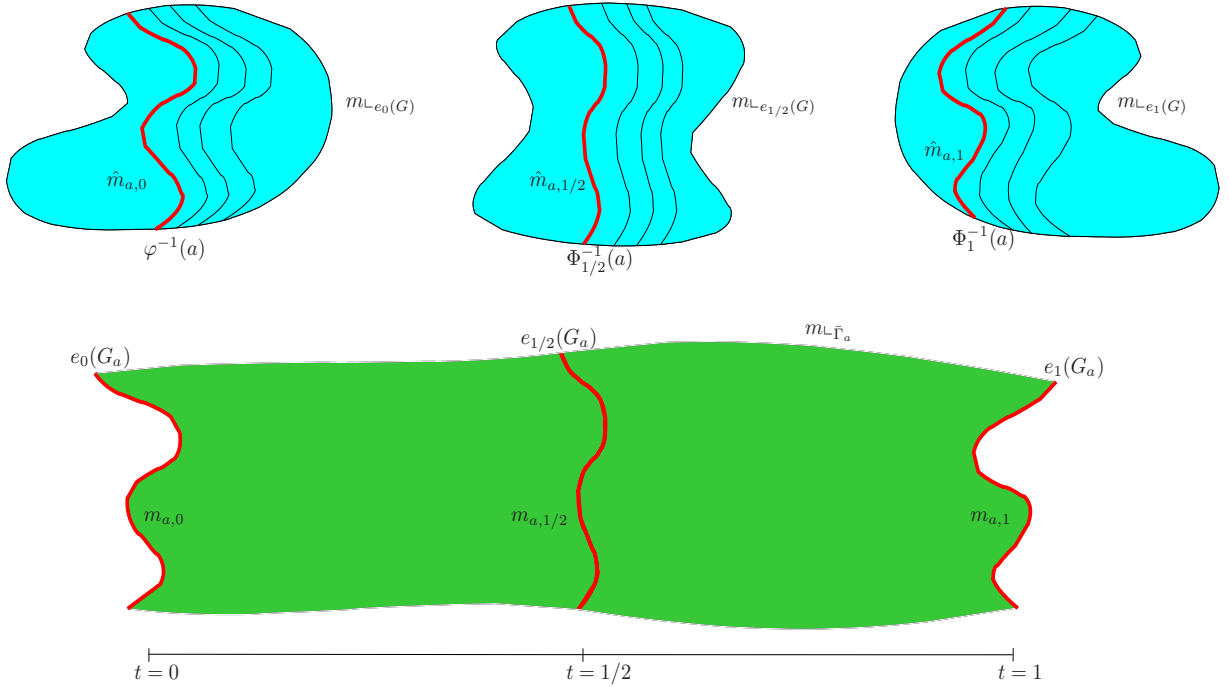}}}
\caption{Above and below the disintegration with conditional $\hat m_{a,t}$ and $m_{a,t}$, respectively.}
\label{fig:different} 
\end{figure}

\begin{remark} Here we want to stress the differences between $\hat m_{a,t}$ and $m_{a,t}$.
It is worth underlining again that both measures are concentrated on $e_{t}(G_{a})$. Also they are both obtained as conditional measures of $m$ 
or, otherwise stated, they belong to the range of two different disintegration maps of $m$:
\[
m\llcorner_{e_{t}(G)} = \int_{\f(\mu_{0})} \hat m_{a,t} \mathcal{L}^{1}(da),     \qquad m\llcorner_{\bar \Gamma_{a}(1)} = \int_{[0,1]} m_{a,t} \mathcal{L}^{1}(dt).
\]
Since in both disintegrations the quotient measure is $\mathcal{L}^{1}$, conditional 
measures  can be interpret as the ``derivative'' with respect to the parameter in the quotient space, $a$ in the first case and $t$ in the second one, of $m$.
Even if $m$ and $e_{t}(G_{a})$ are fixed, what do matters, and implies $\hat m_{a,t} \neq m_{a,t}$, is the difference between $e_{t+\ve}(G_{a})$ and $e_{t}(G_{a+\ve})$. The difference can be observed in Figure \ref{fig:different}.
\end{remark}

\begin{lemma}\label{L:Lebesgue}
For every $a \in \f(\mu_{0})$, 
\[
\lim_{s \to 0}\frac{1}{s} \int_{(t,t+s)} m_{a,\tau} \mathcal{L}^{1}(d\tau) = m_{a,t}, 
\]
for $\mathcal{L}^{1}$-a.e. $t\in [0,1]$, where the convergence is in the weak sense.
\end{lemma}
\begin{proof}
Since $(X,d,m)$ is locally compact, the space of real valued continuous and bounded functions $C_{b}(X)$ is separable. 
Let $\{f_{k} \}_{k \in \enne} \subset C_{b}(X)$ be a dense family. 

Fix $a\in \f(\mu_{0})$. The Lebesgue differentiation theorem implies that for every $k \in \enne$
\[
\frac{1}{s} \int_{(t,t+s)} \bigg(  \int f_{k}(z)m_{a,\tau}(dz) \bigg)  \mathcal{L}^{1}(d\tau) \to \int f_{k}(z) m_{a,t}(dz), \quad
\textrm{as } s \searrow 0,
\]
as real numbers, for all $t \in [0,1] \setminus E_{a,k}$ with $\mathcal{L}^{1}(E_{a,k})=0$. Hence $E_{a} : =\cup_{m \in \enne} E_{a,k}$ is 
$\mathcal{L}^{1}$-negligible. 
Take $f \in C_{b}(X)$ and chose $\{f_{k_{h}}\}_{h\in \enne}$ approximating $f$ in the uniform norm.
Using $f_{k_{h}}$, it is then fairly easy to show that 
\[
\lim_{s\to 0 }\frac{1}{s} \int_{(t,t+s)} \bigg(  \int f(z)m_{a,\tau}(dz) \bigg)  \mathcal{L}^{1}(d\tau) = \int f(z) m_{a,t}(dz)
\]
for all $t \in [0,1] \setminus E_{a}$. 
\end{proof}

The analogous statement of Lemma \ref{L:Lebesgue} is true for the conditional measures $\hat m_{a,t}$ of \eqref{E:mass2}: fix $t \in [0,1]$, then  
\[
\lim_{b \to 0}\frac{1}{b} \int_{(a,a+b)} \hat m_{\alpha,t} \mathcal{L}^{1}(d\alpha) = \hat m_{a,t}, 
\]
for $\mathcal{L}^{1}$-a.e. $a\in \f(\mu_{0})$, where the convergence is in the weak sense.

\subsection{Comparison between conditional measures}

The next one is the main technical statement of the Section.

\begin{proposition}\label{P:stesso}
For $\mathcal{L}^{1}$-a.e. $a \in \f(\mu_{0})$ and every sequence $\ve_{n} \to 0^{+}$ there exists a subsequence $\ve_{n_{k}}$ so that:
 \[
\lim_{\ve \to 0^{+}} \frac{1}{\ve} \cdot m \llcorner_{\Phi_{t}^{-1} ([a-\ve,a])}  =
\lim_{k \to \infty} \frac{1}{\ve_{n_{k}}} \cdot m \llcorner_{\Phi_{t}^{-1} ([a-\ve_{n_{k}},a]) \cap \bar  \Gamma_{a}(1)} 
\]
for $\mathcal{L}^{1}$-a.e. $t \in  [0,1]$, where the exceptional set depends on the subsequence $\ve_{n_{k}}$ and
the limit is in the weak topology.
\end{proposition}

\begin{proof}

{\it Step 1.} 
We show that for every $a\in \f(\mu_{0})$, 
\[
\lim_{\ve \to 0} \frac{1}{\ve} \int_{(0,1)} m\big(  \Phi_{t}^{-1} ([a-\ve,a])  \setminus \bar \Gamma_{a}(1) \big)  dt = 0.
\]
Suppose by contradiction the existence of $a\in \f(\mu_{0})$ and of a sequence $\ve_{n} \to 0$ such that 
\[
\lim_{n \to \infty} \frac{1}{\ve_{n}} \int_{(0,1)} m\big(  \Phi_{t}^{-1} ([a-\ve_{n},a])  \setminus \bar \Gamma_{a}(1) \big)  dt  \geq \alpha.
\]
Then, since in Lemma \ref{L:cod1} we have proved that  $\|\nabla \hat \Phi_{t} \| \hat m_{b,t} \leq c\,\S^{h}\llcorner_{e_{t}(G_{b})}$ for $\mathcal{L}^{1}$-a.e. 
$b \in \f(\mu_{0})$, with $\|\nabla \hat \Phi_{t} \|$ positive $m$-a.e., it follows that 
\[
\lim_{n \to \infty} \frac{1}{\ve_{n}}  \int_{(0,1)} \int_{[a-\ve_{n},a]}  \S^{h} (e_{t}(G_{b}) \setminus \bar \Gamma_{a}(1))\mathcal{L}^{1}(db) \mathcal{L}^{1}(dt) 
\geq \alpha.
\]
Then by Fubini's Theorem
\[
\lim_{n \to \infty} \frac{1}{\ve_{n}} \int_{[a-\ve_{n},a]}   \int_{(0,1)}  \S^{h} (e_{t}(G_{b}) \setminus \bar \Gamma_{a}(1))\mathcal{L}^{1}(dt) \mathcal{L}^{1}(db) 
\geq \alpha.
\]
Hence there exists a sequence $a_{k}$ converging to $a$ from below such that 
\[
\int_{(0,1)} \S^{h}(e_{t}(G_{a_{k}}) \setminus \bar \Gamma_{a}(1))  \mathcal{L}^{1}(dt) \geq \alpha.
\]
for all $k \in \enne$.

{\it Step 2.} It follows from Lemma \ref{L:cod2} and Proposition \ref{P:noboundary} that, since $m\llcorner_{\bar \Gamma_{a_{k}}(1)} = \int m_{a_{k},t}dt$:
\[
\liminf_{k \to 0} m\left(\bar \Gamma_{a_{k}}(1) \setminus \bar \Gamma_{a}(1) \right)  \geq \alpha.
\]
Since as $k\to \infty$ the sequence $a_{k}$ is converging to $a$, the sequence of compact sets of geodesics $\{G_{a_{k}}\}_{k\in \enne}$ is converging in Hausdorff topology to a subset of $G_{a}$, hence the same happens for the sequence of compact sets $\{ \bar \Gamma_{a_{k}}(1) \}_{k\in \enne}$.
Then just observe that 
\[
m(\bar \Gamma_{a}(1) ) = \lim_{\delta \to 0} m (\bar \Gamma_{a}(1) ^{\delta} ) \geq m(\bar \Gamma_{a}(1) )+  
\liminf_{k \to \infty}m(\bar \Gamma_{a_{k}}(1) \setminus \bar \Gamma_{a}(1) ) \geq m(\bar \Gamma_{a}(1) ) + \alpha,
\]
where $\bar \Gamma_{a}(1) ^{\delta} = \{z \in X : d(z, \bar \Gamma_{a}(1) ) \leq \delta \}$ is a neighborhood of $\bar \Gamma_{a}(1)$ and the first inequality follows from the definition of Hausdorff convergence.
Since $\alpha>0$ we have a contradiction and therefore for each $a \in \f(\mu_{0})$
\[
\lim_{\ve \to 0} \frac{1}{\ve} \int_{(0,1)} m\big(  \Phi_{t}^{-1} ([a-\ve,a])  \setminus \bar \Gamma_{a}(1) \big)  dt = 0.
\]
So for each sequence $\ve_{n} \to 0$ there exists a subsequence $\ve_{n_{k}}$ such that 
\[
\lim_{k \to \infty} \frac{1}{\ve_{n_{k}}} m\big(  \Phi_{t}^{-1} ([a-\ve_{n_{k}},a])  \setminus \bar \Gamma_{a}(1) \big)  = 0,
\]
for $\mathcal{L}^{1}$-a.e. $t \in [0,1]$.

{\it Step 3.} Let $\{f_{h} \}_{h \in \enne} \subset C_{b}(X)$ be a dense family. Then for each $f_{k}$
\[
\lim_{k \to \infty} \left( \frac{1}{\ve_{n_{k}}} \int_{[a-\ve_{n_{k}}, a]} f_{h} \hat m_{b,t} \mathcal{L}^{1}(db) - 
\frac{1}{\ve_{n_{k}}} \int_{\Phi_{t}^{-1}([a-\ve_{n_{k}}, a]) \cap \bar \Gamma_{a}(1)} f_{h}m 
\right) = 0
\]
for all $t \in [0,1]$ minus a set of measure zero. Reasoning as Lemma \ref{L:Lebesgue}, we have the existence 
of a set $E\subset [0,1]$ with $\mathcal{L}^{1}(E) = 0$ such that for all $f \in C_{b}(X)$ it holds
\[
\lim_{k \to \infty} \left( \frac{1}{\ve_{n_{k}}} \int_{[a-\ve_{n_{k}}, a]} f \hat m_{b,t} \mathcal{L}^{1}(db) - 
\frac{1}{\ve_{n_{k}}} \int_{\Phi_{t}^{-1}([a-\ve_{n_{k}}, a]) \cap \bar \Gamma_{a}(1)} fm 
\right) = 0
\]
for all $t \in [0,1] \setminus E$. Then again from Lemma \ref{L:Lebesgue} applied to $\hat m_{a,t}$ we have the claim.
\end{proof}

The proof of the next Corollary follows from Lemma \ref{L:Lebesgue} and Proposition \ref{P:stesso}.
\begin{corollary}\label{C:step}
For $\mathcal{L}^{1}$-a.e. $a \in \f(\mu_{0})$ the following holds: for every sequence $\ve_{n} \to 0$ there exists a subsequence 
$\ve_{n_{k}} \to 0$ so that 
\[
\lim_{k \to \infty} \frac{1}{\ve_{n_{k}}} \cdot m \llcorner_{\Phi_{t}^{-1} ([a-\ve_{n_{k}},a]) \cap \bar  \Gamma_{a}(1)} = \hat m_{a,t}
\]
for $\mathcal{L}^{1}$-a.e. $t \in  [0,1]$, where the exceptional set depends on the subsequence $\ve_{n_{k}}$ and
the limit is in the weak topology.
\end{corollary}

We now prove that  $\hat m_{a,t} \ll m_{a,t}$. Let us recall the disintegration formula for $m$ as constructed in 
Proposition \ref{P:absolutecont}: for each $a \in \f(\mu_{0})$ since the geodesics in $G_{a}$ are disjoint even for different times 
it holds
\begin{equation}\label{E:motionagain}
m_{\bar \Gamma_{a}(1)} = \int_{e_{1/2}(G_{a})} g(y,\cdot)\mathcal{L}^{1}\llcorner_{[0,1]} q_{a}(dy).
\end{equation}
where $g$ satisfies \eqref{E:surfacevopoint}, $q_{a}$ is the quotient measure satisfying for $I \subset e_{1/2}(G_{a})$
\[
q_{a}(I) = m( \{ \gamma_{t} : \gamma \in G_{a}, \gamma_{1/2} \in I \})
\]
and the measure $g(y,\cdot)\mathcal{L}^{1}\llcorner_{[0,1]}$ has to be intended as 
$(\gamma)_{\sharp}(g(y,\cdot)\mathcal{L}^{1}\llcorner_{[0,1]})$, with 
$\gamma$ the unique element of $G_{a}$ so that $\gamma_{1/2} = y$.
Being the evaluation map $e_{1/2}$ a Borel isomorphism between $G_{a}$ and $e_{1/2}(G)$ 
the measure $q_{a}$ can be also interpret as a measure on $G_{a}$.

\begin{proposition}\label{P:step2}
For $\mathcal{L}^{1}$-a.e. $a \in \f(\mu_{0})$ 
\[
\hat m_{a,t} \ll m_{a,t}
\]
for $\mathcal{L}^{1}$-a.e. $t \in [0,1]$. Equivalently $(e_{1/2})_{\sharp} \gammA_{a} \ll q_{a}$.
\end{proposition}

\begin{proof}
Consider $a \in \f(\mu_{0})$ and a subsequence $\ve_{n_{k}}$ so that Corollary \ref{C:step} holds. 

{\it Step 1.} Consider the evaluation map $e : [0,1] \times e_{1/2}(G_{a}) \to \bar \Gamma_{a}(1)$ defined as usual by
\[
e(s,y) =  e_{s} \circ e_{1/2}^{-1}(y).
\]
Note that it is continuous, surjective and its inverse is continuous as well. Hence $\bar \Gamma_{a}(1)$ and  
$[0,1] \times e_{1/2}(G_{a})$ are homeomorphic.

Take $I$ compact subset of $e_{1/2}(G_{a})$ with $q_{a}(I) = 0$.
Since $q_{a}$ is a regular finite measure on $e_{1/2}(G_{a})$,
by outer regularity there exists a sequence $\{A_{i}\}_{i\in \enne}$ with $A_{i } \subset e_{1/2}(G_{a})$ 
and open in the subspace topology of $e_{1/2}(G_{a})$
so that  
\[
I \subset A_{i}, \qquad  q_{a}(A_{i}) \leq \frac{1}{i}.
\]
Take now any open set $U \subset [0,1]$ neighborhood of $1/2$. Then $e(U \times A_{i})$ will be an open set 
in $\bar \Gamma_{a}(1)$ for each $i\in \enne$.

{\it Step 2.}
Then
\begin{align*}
\frac{1}{\ve_{n_{k}}} m & \llcorner_{\Phi_{t}^{-1} ([a-\ve_{n_{k}},a]) \cap \bar  \Gamma_{a}(1)}(e(U\times A_{i})) \crcr
= &~ 
\frac{1}{\ve_{n_{k}}} \int_{A_{i}} (g(y,\cdot) \mathcal{L}^{1})
\left(U \cap \{ \tau \in [t,1] : \Phi_{t}(\gamma_{\tau}) \in [a-\ve_{n_{k}},a ] \}\right)
q_{a}(dy).
\end{align*}
Let $s_{k} \in (0,1)$ be such that 
\[
s_{k} = \max \{ s : \Phi_{t}(\gamma_{t+s}) \geq a -\ve_{n_{k}} \}. 
\]
Then $s_{k} \geq \mathcal{L}^{1}\left( \{ \tau \in [t,1] : \Phi_{t}(\gamma_{\tau}) \in [a-\ve_{n_{k}},a ] \}\right)$
and 
\[
\ve_{n_{k}} = a - (a - \ve_{n_{k}}) \geq \Phi_{t}(\gamma_{t}) - \Phi_{t}(\gamma_{t+s_{k}})
\]
Therefore 
\begin{equation}\label{E:limite}
\lim_{k \to \infty} \frac{1}{\ve_{n_{k}}} \mathcal{L}^{1}\left( \{ \tau \in [0,1] : \Phi_{t}(\gamma_{\tau}) \in [a-\ve_{n_{k}},a ] \}\right) \leq 
\lim_{k \to \infty} \frac{s_{k}}{\Phi_{t}(\gamma_{t}) - \Phi_{t}(\gamma_{t+s_{k}})}
\end{equation}
and the last term by Assumption \ref{A:Phi} is bounded. Since $g$ is uniformly bounded as well, it follows that 
\[
\frac{1}{\ve_{n_{k}}} m  \llcorner_{\Phi_{t}^{-1} ([a-\ve_{n_{k}},a]) \cap \bar  \Gamma_{a}(1)}(e(U\times A_{i})) \leq C q_{a}(A_{i}), 
\]
for some positive constant $C$ not depending on $k$.

{\it Step 3.} We now observe that $\bar \Gamma_{a}(1)$ is a compact set. Hence any function $f \in C_{b}(X)$
can be extended, by Tiezte's Theorem, to a bounded and continuous function on the whole space, say $\tilde f$.
It follows that 
\[
\lim_{k \to \infty} \frac{1}{\ve_{n_{k}}} \cdot m \llcorner_{\Phi_{t}^{-1} ([a-\ve_{n_{k}},a]) \cap \bar  \Gamma_{a}(1)} = \hat m_{a,t}
\]
holds also in the weak topology of $\mathcal{P}(\bar \Gamma_{a}(1))$.
So we can use lower semicontinuity on open sets of weakly converging measures, also for open sets in the trace topology of 
$\bar \Gamma_{a}(1)$. Therefore
\[
\hat m_{a,t}(e(U\times A_{i})) \leq \liminf_{k \to \infty} \frac{1}{\ve_{n_{k}}} m  \llcorner_{\Phi_{t}^{-1} ([a-\ve_{n_{k}},a]) \cap \bar  \Gamma_{a}(1)}(e(U\times A_{i})) \leq C \frac{1}{i}.
\]
By outer regularity, $\hat m_{a,t} (I ) = 0$ and the claim follows.
\end{proof}

Direct consequence of Proposition \ref{P:step2} is that for $\mathcal{L}^{1}$-a.e. $a \in \f(\mu_{0})$ we have $\hat m_{a,t} = \theta_{a,t} q_{a}$ 
for $\mathcal{L}^{1}$-a.e. $t \in [0,1]$, that is 
\[
\hat m_{a,t}(K) = \int_{e_{1/2}(e_{t}^{-1}(K))}  \theta_{a,t}(y) q_{a}(dy), 
\]
for all $K \subset e_{t}(G_{a})$.

\subsection{A formula for the density}

We now derive an explicit expression for the density of $\hat m_{a,t}$ with respect to $m_{a,t}$.

\begin{lemma}\label{L:welldefined}
For $\mathcal{L}^{1}$-a.e. $a\in \f(\mu_{0})$ and every sequence $\ve_{n} \to 0^{+}$, there exists a subsequence $\ve_{n_{k}}$ such that  the limit 
\begin{equation}\label{E:1dim}
\lim_{k \to \infty} \frac{1}{\ve_{n_{k}}}  \mathcal{L}^{1}\Big( \big\{ \tau \in (0,1) :  \Phi_{t}(\gamma_{\tau}) \in [a- \ve_{n_{k}},a] \big\} \Big)
\end{equation}
exists for $\gammA_{a}$-a.e. $\gamma \in G_{a}$ and $\mathcal{L}^{1}$-a.e. $t\in [0,1]$. If we denote by $\lambda_{t}(\gamma_{t})$ its value, then 
\[
\hat m_{a,t} = \lambda_{t} m_{a,t}.
\]
\end{lemma}

\begin{proof}
Consider $a \in \f(\mu_{0})$ and $\ve_{n_{k}}$ so that Corollary \ref{C:step} and Proposition \ref{P:step2} holds. 
Then we have
\[
\lim_{k\to \infty} \int \frac{1}{\ve_{n_{k}}} \Big(g(y,\cdot)\mathcal{L}^{1}\Big)_{\Big( \big\{ \tau \in (0,1) :  \Phi_{t}(\gamma_{\tau}) \in [a-\ve_{n_{k}},a] \big\} \Big)} 
q_{a}(dy)
=  \theta_{a,t}q_{a},
\]
again for $\mathcal{L}^{1}$-a.e. $t \in  [0,1]$ with the exceptional set depending on the subsequence and where the convergence is in the weak topology.
Using a localization argument on the support of $q_{a}$, it follows that there exists another subsequence that we will call again $\ve_{n_{k}}$ so that 
\[
\lim_{k\to \infty} \frac{1}{\ve_{n_{k}}} \Big(g(y,\cdot)\mathcal{L}^{1}\Big)_{\Big( \big\{ \tau \in (0,1) :  \Phi_{t}(\gamma_{\tau}) \in [a-\ve_{n_{k}},a] \big\} \Big)} = \theta_{a,t} \delta_{y}
\]
for $\mathcal{L}^{1}$-a.e. $t \in  [0,1]$ and $q_{a}$-a.e. $y \in e_{1/2}(G_{a})$. Then by continuity of $t\mapsto g(y,t)$ for $q_{a}$-a.e. $y$, it follows that 
\begin{align*}
\lim_{k\to \infty} &~\frac{1}{ \mathcal{L}^{1}\Big( \big\{ \tau \in (0,1) :  \Phi_{t}(\gamma_{\tau}) \in [a-\ve_{n_{k}},a] \big\} \Big)}  \crcr
 &~ \quad \cdot \Big(g(y,\cdot)\mathcal{L}^{1}\Big)_{\Big( \big\{ \tau \in (0,1) :  \Phi_{t}(\gamma_{\tau}) \in [a-\ve_{n_{k}},a] \big\} \Big)} = g(y,t) \delta_{y}
\end{align*}
for $\mathcal{L}^{1}$-a.e. $t \in  [0,1]$ and $q_{a}$-a.e. $y \in e_{1/2}(G_{a})$.
Then necessarily
\[
\lim_{k \to \infty} \frac{1}{\ve_{n_{k}}}  \mathcal{L}^{1}\Big( \big\{ \tau \in (0,1) :  \Phi_{t}(\gamma_{\tau}) \in [a- \ve_{n_{k}},a] \big\} \Big)
\]
exists $\mathcal{L}^{1}$-a.e. $t \in  [0,1]$ and $q_{a}$-a.e. $y \in e_{1/2}(G_{a})$. By uniqueness of the limit 
\[
\theta_{a,t} = g(y,t) \cdot \lim_{k \to \infty} \frac{1}{\ve_{n_{k}}}  \mathcal{L}^{1}\Big( \big\{ \tau \in (0,1) :  \Phi_{t}(\gamma_{\tau}) \in [a- \ve_{n_{k}},a] \big\} \Big).
\]
Hence if we define $\lambda_{t} (y) = \theta_{a,t}(y)/g(y,t)$ then
\[
\lambda_{t} (y) = \lim_{k \to \infty} \frac{1}{\ve_{n_{k}}}  \mathcal{L}^{1}\Big( \big\{ \tau \in (0,1) :  \Phi_{t}(\gamma_{\tau}) \in [a- \ve_{n_{k}},a] \big\} \Big)
\]
and since $m_{a,t} = g(\cdot,t) q_{a}$ and $\hat m_{a,t} = \theta_{a,t} q_{a}$ it follows that
\[
\hat m_{a,t} = \lambda_{t} m_{a,t},
\]
and therefore the claim.
\end{proof}

At the beginning of this section we observed that 
\[
e_{t\,\sharp} \gammA_{a} = \left( \int \r_{t}(z) \hat m_{a,t}(dz)\right)^{-1} \r_{t} \hat m_{a,t},
\]
so now Proposition \ref{P:step2} and Lemma \ref{L:welldefined} implies the next corollary.
\begin{corollary}\label{C:surfacevo}
The measure $(e_{t})_{\sharp}\gammA_{a}$ is absolute continuous with respect to the surface measure $m_{a,t}$.
\end{corollary}

Let $\hat h_{a,t}$ be such that $(e_{t})_{\sharp}\gammA_{a} = \hat h_{a,t} m_{a,t}$.
We prefer to think of $\hat h_{a,t}$ as a function defined on $G_{a}$ rather than on $e_{t}(G_{a})$, hence define 
$h_{a,r}: G_{a} \to [0,\infty]$ by $h_{a,r}(\gamma):= \hat h_{a,r}(\gamma_{r})$.
So we have found a decomposition of $\r_{t}$:
\[
\r_{t}(\gamma_{t}) =\left( \int \r_{t}(z) \hat m_{a,t}(dz)\right) \frac{1}{\lambda_{t}(\gamma_{t})}h_{a,t}(\gamma),
\]
where $a = \f(\gamma_{0})$.
We now deduce a more convenient expression for $\lambda_{t}$. 
Recall the definition 
\[
I_{t}(\gamma) = \{ \tau \in [0,1] : \gamma_{\tau} \in e_{t}(G)\} = \{ \tau \in [0,1] : d(\gamma_{\tau}, e_{t}(G) ) = 0\},
\]

\begin{theorem}\label{T:expression1}
For $\mathcal{L}^{1}$-a.e $t \in [0,1]$
\begin{equation}\label{E:incrementalration}
\frac{1}{\lambda_{t}(\gamma_{t})} = \lim_{s\to 0} \frac{\Phi_{t}(\gamma_{t})  - \Phi_{t}(\gamma_{t+s})}{s},
\end{equation}
point wise for $\gammA$-a.e. $\gamma \in G$.
\end{theorem}

\begin{proof}
{\it Step 1.} From Lemma \ref{L:welldefined} for $\mathcal{L}^{1}$-a.e. $a \in \f(\mu_{0})$, for every $\ve_{n} \to 0^{+}$ there exists a subsequence 
$\ve_{n_{k}}$ such that 
\[
\lambda_{t}(\gamma_{t}) = \lim_{k \to \infty} \frac{1}{\ve_{n_{k}}}   \mathcal{L}^{1}\Big( \big\{ \tau \in (0,1) :  \Phi_{t}(\gamma_{\tau}) \in [a- \ve_{n_{k}},a] \big\} \Big),
\]
point wise $\gammA_{a} \otimes \mathcal{L}^{1}$-a.e. in $G_{a} \times [0,1]$. 
An equivalent expression of $\lambda_{t}(\gamma_{t})$ is: 
\[
\lambda_{t}(\gamma_{t}) = \lim_{k\to \infty} \frac{ \big((\Phi_{t} \circ \gamma)_{\sharp} \mathcal{L}^{1}\big)  ([a-\ve_{n_{k}},a]) }{\mathcal{L}^{1}  ([a-\ve_{n_{k}},a])}.
\]

Using Assumption \ref{A:Phi} $\lambda_{t}$ can be written in terms of the same limit above substituting $\Phi_{t}\circ \gamma$, that is defined only on $I_{t}(\gamma)$,
with an extension of $\Phi_{t}$ to a neighborhood of $t$. 

Since for each $\gamma \in G$ the set $I_{t}(\gamma)$ is compact,
 we can extend $\Phi_{t}$  by linearity on each geodesic of $G_{a}$. By $d$-monotonicity this will create no problem in the definition.  
More specifically: for $\delta> 0$ fixed,  for each $\tau \in [t - \delta, t+ \delta]$ and $\gamma \in G_{a}$  
there exists 
\[
\tau_{m} = \max \{s \in I_{t}(\gamma) : s \leq \tau \}, \qquad  \tau_{M} = \min\{ s \in I_{t}(\gamma) : \tau \leq s \}.
\]
Clearly $\tau_{m}$ and $\tau_{M}$ depends on $\gamma$ and if $\tau \in I_{t}(\gamma)$ they all coincide $\tau = \tau_{m} = \tau_{M}$. 
Then we define the extension map $\hat \Phi_{t}$ by linearity
\[
\hat \Phi_{t}(\gamma_{\tau}) = \Phi_{t}(\gamma_{\tau_{m}}) + (\tau - \tau_{m}) \frac{\Phi_{t}(\gamma_{\tau_{M}}) - \Phi_{t}(\gamma_{\tau_{m}})}{\tau_{M} - \tau_{m}}.
\]
Since by $d$-cyclical monotonicity $\gamma_{t} \neq \bar \gamma_{s}$ for all $t,s \in [0,1]$ if $\gamma, \bar \gamma \in G_{a}$ with $\gamma \neq \bar \gamma$,
the map $\hat \Phi_{t}$ is well defined on $e([t-\delta, t+\delta] \times G_{a})$ and is measurable. 
Moreover by Assumption \ref{A:Phi}, on each line the map
\[
[t-\delta, t + \delta] \ni \tau \mapsto \hat \Phi_{t}(\gamma_{\tau})
\]
is differentiable in $t$ with strictly negative derivative and is Lipschitz in the whole interval $[t -\delta, t+ \delta]$. 

Consider now $\tau_{k} = \max \{ \tau \in [t,t+\delta] : \hat \Phi_{t}(\gamma_{\tau}) \geq a-\ve_{n_{k}} \}$, then 
\[
\mathcal{L}^{1}\Big(  \big( \hat \Phi_{t} \circ \gamma\big)^{-1} [a-\ve_{n_{k}}, a]   \Big) \leq \tau_{k} -t.
\]
Since by construction
\[
\hat  \Phi_{t}(\gamma_{t}) - \hat \Phi_{t}(\gamma_{\tau_{k}}) \geq \frac{1}{c}(\tau_{k} - t), 
\]
for some positive constant $C$,  we have  $(\tau_{k} - t) \leq c \ve_{n_{k}}$ implying 
that 
\begin{align*}
\int_{G_{a}} &~\left| \frac{\big((\Phi_{t} \circ \gamma)_{\sharp} \mathcal{L}^{1}\big)  ([a-\ve_{n_{k}},a]) }{\ve_{n_{k}}} - 
\frac{\big((\hat \Phi_{t} \circ \gamma)_{\sharp} \mathcal{L}^{1}\big)  ([a-\ve_{n_{k}},a]) }{\ve_{n_{k}}} \right| \gammA_{a}(d\gamma) \crcr
\leq &~ \int_{G_{a}} \frac{\big((\hat \Phi_{t} \circ \gamma)_{\sharp} \mathcal{L}^{1}\big)  ([a-\ve_{n_{k}},a] \cap I_{t}(\gamma)^{c}) }{\ve_{n_{k}}} \gammA_{a}(d\gamma)\crcr
\leq &~ \int_{G_{a}} \frac{   \mathcal{L}^{1}\left( (t - c\ve_{n_{k}}, t+ c\ve_{n_{k}} )  \cap I_{t}(\gamma)^{c} \right)}{\ve_{n_{k}}} \gammA_{a}(d\gamma).
\end{align*}
By Lemma \ref{L:density} the last integral converges to $0$ as $k \to \infty$. 
We have therefore proved that for $\mathcal{L}^{1}$-a.e. $a \in \f(\mu_{0})$, for every $\ve_{n} \to 0$ there exists a subsequence 
$\ve_{n_{k}}$ such that 
\[
\lambda_{t}(\gamma_{t}) = \lim_{k \to \infty} \frac{1}{\ve_{n_{k}}}   \mathcal{L}^{1}\Big( \big\{ \tau \in (0,1) : \hat  \Phi_{t}(\gamma_{\tau}) \in [a- \ve_{n_{k}},a] \big\} \Big),
\]
for $\gammA_{a} \otimes \mathcal{L}^{1}$-a.e. in $G_{a} \times [0,1]$.

{\it Step 2.} Now we take advantage from the fact that $\hat \Phi_{t} \circ \gamma$ is defined on a connected set and invertible.
For any $\ve$ sufficiently small the following identity holds:  
\[
\frac{ \big((\hat \Phi_{t} \circ \gamma)_{\sharp} \mathcal{L}^{1}\big)  ([a-\ve,a]) }{\ve} =
\frac{s_{\ve}}{ \hat \Phi_{t}(\gamma_{t})- \hat \Phi_{t}(\gamma_{t+s_{\ve}})},
\]
where $s_{\ve}$ is the unique $s \in [t,t+\delta]$ such that
\[
\hat \Phi_{t}(\gamma_{t+s_{\ve}}) = a -\ve.
\]
It follows that 
\begin{equation}\label{E:positive}
\frac{1}{\lambda_{t}(\gamma_{t})} = \lim_{s \to 0 } \frac{\hat \Phi_{t}(\gamma_{t}) - \hat \Phi_{t}(\gamma_{t+s})   }{s},
\end{equation}
for $\gammA_{a} \otimes \mathcal{L}^{1}$-a.e. $(\gamma,t) \in G_{a}\times [0,1]$.
Restricting $s$ to $I_{t}(\gamma)$ the claim follows.
\end{proof}

\section{Global estimates and main theorems}\label{S:reduction}
So far we have proved that in a metric measure space $(X,d,m)$ verifying $\mathsf{CD}_{loc}(K,N)$ or $\mathsf{CD}^{*}(K,N)$
(actually $\mathsf{MCP}(K,N)$ would be enough),
given a geodesic $\mu_{t} = \r_{t} m$ in the $L^{2}$-Wasserstein space with some regularity, the following decomposition holds:
\[
\r_{t}(\gamma_{t}) =\left( \int \r_{t}(z) \hat m_{a,t}(dz)\right) \frac{1}{\lambda_{t}(\gamma_{t})}h_{a,t}(\gamma),
\]
where $a = \f(\gamma_{0})$ and the functions involved in the decomposition are determined by the following identities:
\begin{align*}
h_{a,t}m_{a,t} = &~ (e_{t})_{\sharp} \gammA_{a} = \left( \int \r_{t}(z) \hat m_{a,t}(dz)\right)^{-1} \r_{t} \hat m_{a,t}, \crcr
\frac{1}{\lambda_{t}(\gamma_{t})} = &~ \lim_{s \to 0 } \frac{\Phi_{t}(\gamma_{t}) - \Phi_{t}(\gamma_{t+s})   }{s}.
\end{align*}
From Assumption \ref{A:Phi}, $\lambda_{t}(\gamma_{t}) >0$ $\gammA$-a.e. and the above expression make sense.

To give a complete meaning to this decomposition we have to prove additional properties for both $h_{a,t}$ and $\lambda_{a,t}$.
In this Section we will consider this function $h_{a,t}$ and $\lambda_{t}$ in the perspective of lower curvature bounds.
In particular, thanks to the metric results proved in Section \ref{S:dmonotone},
we prove that $h_{a,t}$ verifies $\mathsf{CD}^{*}(K,N-1)$.

As already observed (see \eqref{E:motionagain}) a disintegration of $m\llcorner_{\bar \Gamma_{a}(1)}$ is given by the next expression:
\begin{equation}\label{E:horizontal}
m_{\bar \Gamma_{a}(1)} = \int_{e_{1/2}(G_{a})} \left( g(y,\cdot) \mathcal{L}^{1}\llcorner_{[0,1]} \right) q_{a}(dy),
\end{equation}
where $g(y,\cdot) \mathcal{L}^{1}\llcorner_{[0,1]}$ has to be intended as a measure on $\gamma_{[0,1]} \subset X$, the image of $\gamma$ where $\gamma = e_{1/2}^{-1}(y)$.

Since $\lambda_{t}(\gamma_{t})> 0$ also for $\gammA_{a}$-a.e $\gamma \in G_{a}$, it follows from Proposition \ref{P:step2} that 
$(e_{1/2})_{\sharp}\gammA_{a}$ can be taken to be the quotient measure in \eqref{E:horizontal},
at the price of changing the value of $g$:
\begin{equation}\label{E:quotient}
m_{\bar \Gamma_{a}(1)} = \int_{G_{a}} \left( g(\gamma_{1/2},\cdot) \mathcal{L}^{1}\llcorner_{[0,1]} \right) \gammA_{a}(d\gamma),
\end{equation} 
with the change of the value constant in $t$ and therefore the new $g$ still verifies \eqref{E:surfacevopoint}.
For ease of notation in what follows we will just denote with $g(\gamma,t)$ instead of $g(\gamma_{1/2},t)$.
The new densities $g$ enjoy the following property.

\begin{lemma}\label{L:likeh} For $\gammA_{a}$-a.e. $\gamma \in G_{a}$ 
\[
h_{a,t}(\gamma) g(\gamma,t) = 1, \qquad \mathcal{L}^{1}-a.e. \ t \in [0,1].
\]
\end{lemma}
\begin{proof} 
The function $\hat h_{a,t}$ has been introduced after Corollary \ref{C:surfacevo}.
For any measurable sets $H \subset G_{a}$, $I \subset [0,1]$ the following identities hold:
\begin{align*}
\gammA_{a}(H) \mathcal{L}^{1}(I) = &~ \int_{I} (h_{a,t} m_{a,t})(e_{t}(H)) dt 
=  \int_{\{\gamma_{t} : \gamma \in H, t \in I\}} \hat h_{a,t}(z)  m_{a,t}(dz) dt \crcr 
= &~ \int_{\{\gamma_{t} : \gamma \in H, t \in I\}} \hat h_{a,t}(z) m(dz) \crcr
= &~ \int_{H} \left( \int_{I}  h_{a,t}(\gamma) g(\gamma,t) dt  \right) \gammA_{a}(d\gamma),
\end{align*}
where passing from the second to the third line we used \eqref{E:quotient} and $\hat h _{a,t}$ was introduced after 
Corollary \ref{C:surfacevo}.
The claim follows from the arbitrariness of $H$  and $I$.
\end{proof}

\subsection{Gain of one degree of freedom}

As proved in Section \ref{S:dmonotone}, for any  $a_{1} <  b_{1} <  a_{0} < b_{0}$
and for any $\gamma \in G_{a}$ so that $(a_{0},b_{1}) \subset \phi_{a}(\gamma_{[0,1]})$
we can define $R^{\gamma}_{0}, L^{\gamma}_{0} \subset [0,1]$ and  
$R^{\gamma}_{1}, L^{\gamma}_{1} \subset [0,1]$ so that 
\[
\phi_{a}\circ \gamma \left((R^{\gamma}_{0}, R^{\gamma}_{0} + L^{\gamma}_{0} ) \right)= (a_{0},b_{0}),  \qquad
\phi_{a}\circ \gamma \left((R^{\gamma}_{1}, R^{\gamma}_{1} + L^{\gamma}_{1} ) \right)= (a_{1},b_{1}), 
\]
where $\phi_{a}$ is a Kantorovich potential associated to the $d$-monotone set 
$\{ (\gamma_{s}, \gamma_{t})  : \gamma \in G_{a},  s \leq t \}$.
The previous equations are equivalent to 
\[
\phi_{a}\circ \gamma ( R^{\gamma}_{0} ) = b_{0}, \qquad \phi_{a}\circ \gamma ( R^{\gamma}_{0} +L^{\gamma}_{0}) = a_{0}.
\]
and 
\[
\phi_{a}\circ \gamma ( R^{\gamma}_{1} ) = b_{1}, \qquad \phi_{a}\circ \gamma ( R^{\gamma}_{1} +L^{\gamma}_{0}) = a_{1}.
\]
Accordingly for all $t \in [0,1]$ we define 
\[
R^{\gamma}_{t}:= (1-t) R^{\gamma}_{0} + tR^{\gamma}_{1}, \qquad L^{\gamma}_{t}:= (1-t)L^{\gamma}_{0} + tL^{\gamma}_{1}.
\]

Let $H \subset G_{a}$ be so that for all $\gamma \in H$ both $(a_{0},b_{0}) , (a_{1},b_{1}) \subset \phi_{a}(\gamma_{(0,1)})$
with $a_{1} <  b_{1} <  a_{0} < b_{0}$. The Proposition \ref{P:w2geo} implies if we define 
\begin{equation}\label{E:geoattuale}
[0,1] \ni t \mapsto \nu_{t}  : = \frac{1}{\gammA_{a}(H)} \int_{H}   \frac{1}{L^{\gamma}_{t}} 
\mathcal{L}^{1}\llcorner_{[R^{\gamma}_{t},R^{\gamma}_{t} + L^{\gamma}_{t}]}  \gammA_{a}(d\gamma) \in \mathcal{P}([0,1] \times G_{a}),
\end{equation}
then $[0,1] \ni t \mapsto (e_{\sharp}) \nu_{t}$ is a  $W_{2}$-geodesic.

Moreover from Lemma \ref{L:likeh} we can deduce that for each $t \in [0,1]$ 
the density $p_{t}(x)$ of $(e)_{\sharp}\nu_{t}$ w.r.t. $m$ is given by
\begin{equation}\label{E:density}
p_{t}(\gamma_{\tau}) =
\begin{cases}
\displaystyle \frac{1}{\gammA_{a}(H)L^{\gamma}_{t}}  h_{a,\tau}(\gamma), & \tau \in  [R^{\gamma}_{t},R^{\gamma}_{t} + L^{\gamma}_{t}], \crcr 
0, & \textrm{otherwise}.
\end{cases}
\end{equation}
The dynamical optimal plan associated to $\nu_{t}$ can be obtained as follows: consider the following map 
\begin{eqnarray*}
\Theta :  \G(X) \times [0,1] & \to & \G(X) \\
 (\gamma,s) & \mapsto & t \mapsto \eta_{t} = \gamma_{(1-t)(R^{\gamma}_{0} +s L^{\gamma}_{0}) + t (R^{\gamma}_{1}+s L^{\gamma}_{1})} 
\end{eqnarray*}
Then if we pose  
\begin{equation}\label{E:dynamic}
\tilde \gammA_{a} : = \Theta_{\sharp} \bigg( \frac{1}{\gammA_{a}(H)} \gammA_{a}\llcorner_{H} \otimes \mathcal{L}^{1}\llcorner_{[0,1]} \bigg),
\end{equation}
it follows that $(e)_{\sharp}\nu_{t} = (e_{t})_{\sharp} \tilde \gammA_{a}$.

\begin{theorem}\label{T:surface}
For $\gammA_{a}$-a.e. $\gamma \in G_{a}$ and for any $0 \leq \tau_{0}<\tau_{1}\leq 1$ the following inequality holds true:
\begin{equation}\label{E:surface}
h_{a,\tau_{1/2}}^{-\frac{1}{N-1}}(\gamma)  
 \geq  \sigma_{K,N-1}^{(1/2)}\big( (\tau_{1}-\tau_{0})L(\gamma) \big) \left\{   h_{a,\tau_{0}}^{-\frac{1}{N-1}}(\gamma) +   h_{a,\tau_{1}}^{-\frac{1}{N-1}}(\gamma) \right\},
\end{equation}
where $\tau_{1/2}= (\tau_{0}+ \tau_{1})/2$.
\end{theorem}

\begin{proof}
As a preliminary step, we note that in order to prove the claim is sufficient to prove \eqref{E:surface} locally, i.e. for $R_{0}$ and $R_{1}$ sufficiently close.
As proved in \cite{sturm:loc}, reduced curvature dimension condition enjoys the globalization property.

{\it Step 1.} Since $\phi_{a}$ is 1-Lipschitz and $G_{a}$ is compact, there exist real numbers $\alpha_{i}, \beta_{i}$ for $i=0,1$ so that
\[
\phi_{a}\circ e_{0} (G_{a}) \subset [\alpha_{0}, \alpha_{1}], \quad \phi_{a}\circ e_{1} (G_{a}) \subset [\beta_{0}, \beta_{1}].
\]
For any $n\in \enne$ and $\enne \ni k \leq n-1 $ we can consider the following family of curves
\[
E_{k,n} : = \left(\phi_{a}\circ e_{0}\right)^{-1}\left(\left[\alpha_{0} + \frac{k}{n}\alpha_{1}, \alpha_{0} + \frac{k+1}{n}\alpha_{1}  \right]\right), \qquad 
D_{k,n}: = \left(\phi_{a}\circ e_{1}\right)^{-1}\left(\left[\beta_{0} + \frac{k}{n}\beta_{1}, \beta_{0} + \frac{k+1}{n}\beta_{1}  \right]\right),
\]
where the maps $\phi_{a}\circ e_{i}$, for $i=0,1$, has to be considered as defined only on $G_{a}$. Then we define the family of compact sets
\[
M_{h,k,n} : = E_{h,n} \cap D_{k,n}.
\]
For any $n \in \enne$, as $h$ and $k$ vary from 0 to $n-1$ the sets $M_{h,k,n} $ cover $G_{a}$. In particular we will consider this covering for $n$
so that
\[
\frac{1}{n} \ll \min\{ L(\gamma), \gamma \in G_{a} \}, |\alpha_{1} - \alpha_{0}|, |\beta_{1} - \beta_{0}|.
\]
Under the previous condition
\[
\min \{ \phi_{a}(\gamma_{0}) : \gamma \in M_{h,k,n} \}    \gg \max \{ \phi_{a}(\gamma_{1}) : \gamma \in M_{h,k,n} \}.
\]
Then for any $a > b$ real numbers so that
\[
\min \{ \phi_{a}(\gamma_{0}) : \gamma \in M_{h,k,n} \}  > a > b >  \max \{ \phi_{a}(\gamma_{1}) : \gamma \in M_{h,k,n} \},
\]
for any $\gamma \in M_{h,k,n}$ the image $\phi_{a}(\gamma_{[0,1]})$ contains $[b,a]$. Therefore we are under the hypothesis of Proposition \ref{P:w2geo}.

{\it Step 2.}
Fix a compact set $H \subset M_{h,k,n}$ and $a,  b$ such that the curvature dimension condition $\mathsf{CD}(K,N)$ holds
true for all measures supported in 
\[
\phi_{a}^{-1}([b,a]) \cap \{\gamma_{[0,1]} : \gamma \in H \}.
\]
Chose now $a_{0},b_{0}$ and $a_{1},b_{1}$ so that $(b_{0},a_{0}), (b_{1},a_{1}) \subset [b,a]$.
In the same manner as Proposition \ref{P:w2geo} consider $R^{\gamma}_{0},R^{\gamma}_{1},L^{\gamma}_{0}$ and $L^{\gamma}_{1}$.
Finally define $\{(e)_{\sharp}\nu_{t}\}_{t \in [0,1]}$ as before in \eqref{E:geoattuale} and the associated dynamical optimal plan $\tilde \gammA_{a}$ as in \eqref{E:dynamic}.
Note that since $M_{h,k,n}$ is a covering of $G_{a}$ we can always assume $\gammA_{a}(M_{h,k,n})>0$ and therefore $\gammA_{a}(H)>0$.

Condition $\mathsf{CD}_{loc}(K,N)$ for $t=1/2$ imply that for $\tilde \gammA_{a}$-a.e. 
$\eta \in \G(X)$
\[
p^{-1/N}_{1/2}(\eta_{1/2}) \geq \tau_{K,N}^{(1/2)}(d(\eta_{0},\eta_{1})) \left\{ p_{0}^{-1/N}(\eta_{0}) + p_{1}^{-1/N}(\eta_{1}) \right\},
\]
that can be formulated also in the following way: for $\mathcal{L}^{1}$-a.e. $s\in [0,1]$ and $\gammA_{a}$-a.e. $\gamma \in H$ 
\[
p^{-1/N}_{1/2}(\gamma_{R^{\gamma}_{1/2} + s L^{\gamma}_{1/2}}) \geq 
\tau_{K,N}^{(1/2)}\big( (R^{\gamma}_{1}-R^{\gamma}_{0} + s | L^{\gamma}_{1} - L^{\gamma}_{0}| ) L(\gamma) \big) 
\left\{ p_{0}^{-1/N}(\gamma_{R^{\gamma}_{0}+sL^{\gamma}_{0}}) + p_{1}^{-1/N}(\gamma_{R^{\gamma}_{1}+s L^{\gamma}_{1}}) \right\}.
\]
Then using \eqref{E:density} and the continuity of $r \mapsto h_{r}(\gamma)$ (Lemma \ref{L:likeh}), letting $s \searrow 0$, it follows that
\begin{align}\label{E:splitting}
(L^{\gamma}_{0}+&L^{\gamma}_{1})^{1/N} h^{-1/N}_{a,R^{\gamma}_{1/2}}(\gamma) \crcr
\geq &~ \sigma_{K,N-1}^{(1/2)}\big( (R^{\gamma}_{1}-R^{\gamma}_{0})L(\gamma)\big)^{\frac{N-1}{N}} \left\{ (L^{\gamma}_{0})^{1/N} 
h_{a, R^{\gamma}_{0}}^{-1/N}(\gamma) +(L^{\gamma}_{1})^{1/N} h_{a,R^{\gamma}_{1}}^{-1/N}(\gamma) \right\},
\end{align}
for $\gammA_{a}$-a.e. $\gamma \in H$, with exceptional set depending on $a_{0},b_{0},a_{1},b_{1}$.

{\it Step 3.} Note that all the involved quantities in \eqref{E:splitting} are continuous w.r.t. 
$R^{\gamma}_{0},L^{\gamma}_{0},R^{\gamma}_{1},L^{\gamma}_{1}$, that in turn are continuous functions of $a_{0},b_{0},a_{1},b_{1}$ 
respectively.
Therefore there exists a common exceptional set $H' \subset H$ of zero $\gammA_{a}$-measure such that \eqref{E:splitting} holds true for all 
for all $a_{0} > a_{1} \in (a,b)$, and all $b_{0},b_{1}$ so that $a_{0}-b_{0},a_{1}-b_{1}$ are sufficiently small and all $\gamma  \in H \setminus H'$.
Then for fixed fixed $\gamma \in H\setminus H'$, varying $L^{\gamma}_{0},L^{\gamma}_{1}$ in \eqref{E:splitting} yields 
\[
h_{a, R^{\gamma}_{1/2}}^{-\frac{1}{N-1}}(\gamma)  
 \geq  \sigma_{K,N-1}^{(1/2)}\big( (R^{\gamma}_{1}-R^{\gamma}_{0})L(\gamma) \big) \left\{   h_{a,R^{\gamma}_{0}}^{-\frac{1}{N-1}}(\gamma) +   
 h_{a,R^{\gamma}_{1}}^{-\frac{1}{N-1}}(\gamma) \right\}.
\]
Indeed the optimal choice is 
\[
L^{\gamma}_{0}  = L \frac{h_{a,R^{\gamma}_{0}}^{-1/(N-1)}(\gamma)  }{h_{a,R^{\gamma}_{0}}^{-1/(N-1)}(\gamma) 
+ h_{a,R^{\gamma}_{1}}^{-1/(N-1)}(\gamma)}, \qquad
L^{\gamma}_{1}  = L \frac{h_{a,R^{\gamma}_{1}}^{-1/(N-1)}(\gamma)  }{h_{a,R^{\gamma}_{0}}^{-1/(N-1)}(\gamma) 
+ h_{a,R^{\gamma}_{1}}^{-1/(N-1)}(\gamma)}
\]
for sufficiently small $L>0$.

Using the same argument of \cite{cavasturm:MCP}, we prove the global \eqref{E:surface} for $\tau^{\gamma}_{0}, \tau^{\gamma}_{1}$ so that 
\[
\phi_{a}(\gamma_{\tau^{\gamma}_{0}}) \leq a, \qquad  \phi_{a}(\gamma_{\tau^{\gamma}_{1}}) \geq b,
\]
for $\gammA$-a.e. $\gamma \in M_{h,k,n}$. Since $n$ can be as big as we want, $\tau^{\gamma}_{0}$ and $\tau^{\gamma}_{0}$ can be taken $0$ and 
$1$ respectively. Therefore we obtain the claim.
\end{proof}

We have therefore proved one of the main results of this note.
\begin{theorem}\label{T:main1}
Let $(X,d,m)$ be a non-branching metric measure space verifying $\mathsf{CD}_{loc}(K,N)$ or $\mathsf{CD}^{*}(K,N)$
and let $\{ \mu_{t}\}_{t\in [0,1]} \subset \mathcal{P}_{2}(X,d,m)$ be a geodesic with $\mu_{t} = \r_{t} m$.
Assume moreover Assumption \ref{A:lengthreg} and Assumption \ref{A:Phi}. Then
\[
\r_{t}(\gamma_{t}) =C(a) \frac{1}{\lambda_{t}(\gamma_{t})}h_{a,t}(\gamma), \qquad \gammA-a.e. \ \gamma \in G,
\]
where $a = \f(\gamma_{0})$ and $C(a) = \| \gammA_{a} \|$ is a constant depending only on $a$. The map 
$[0,1] \ni t \mapsto h_{a,t}(\gamma)$ verifies $\mathsf{CD}^{*}(K,N-1)$ for $\gammA$-a.e. $\gamma \in G$ and 
\[
\frac{1}{\lambda_{t}(\gamma_{t})} = \lim_{s \to 0 } \frac{\Phi_{t}(\gamma_{t}) - \Phi_{t}(\gamma_{t+s})   }{s}.
\]
\end{theorem}

\subsection{Globalization for a class of optimal transportation}

In order to prove globalization theorem of $\mathsf{CD}_{loc}$ it now necessary to show 
concavity in time of $\lambda_{t}(\gamma_{t})$. 
We will do that in the framework of Section \ref{S:particular}:
$L(\gamma)$ depends only on $\f(\gamma_{0})$, i.e.
\[
L(\gamma) = f(\f(\gamma_{0})), \qquad  \gammA-a.e. \ \gamma \in G, 
\]
for some $f : \f(\mu_{0}) \to (0,\infty)$ such that $\f(\mu_{0}) \ni a \mapsto a - f^{2}/a$ is non increasing.

\begin{proposition}\label{P:1dim}
Assume the following: 
Then for $\gammA$-a.e. $\gamma \in G$ the following holds true
\[
\lambda_{t}(\gamma_{t}) = (1-t) \lambda_{0}(\gamma_{0})+ t \lambda_{1}(\gamma_{1}),
\]
for every $t \in [0,1]$.
\end{proposition}

\begin{proof}
Since $\Phi_{t} = F_{t}^{-1} \circ \f_{t}$, where $F_{t}(a) = a - tf^{2}/2$ and 
\[
\lim_{s\to 0 } \frac{\f_{t}(\gamma_{t})  - \f_{t}(\gamma_{t+ s})  }{s} = L^{2}(\gamma), 
\]
for all $t \in (0,1)$, it follows that from Theorem \ref{T:expression1} that 
\begin{align*}
\lambda_{t}(\gamma_{t}) = &~  (\partial_{a}F_{t}) (F_{t}^{-1} (\f_{t}(\gamma_{t}))) \frac{1}{L^{2}(\gamma)} \crcr
= &~  (\partial_{a}F_{t}) ( \f(\gamma_{0}) ) \frac{1}{L^{2}(\gamma)}. 
\end{align*}
Since $(\partial_{a}F_{t}) g_{t}$ is linear in $t$ the claim follows.
\end{proof}

Using the results proved so far, we can now state the following.

\begin{theorem}\label{T:cd}
Let $(X,d,m)$ be a non-branching metric measure space verifying $\mathsf{CD}_{loc}(K,N)$ or $\mathsf{CD}^{*}(K,N)$
and let $\{ \mu_{t}\}_{t\in [0,1]} \subset \mathcal{P}_{2}(X,d,m)$ be a geodesic with $\mu_{t} = \r_{t} m$.
Assume moreover that 
\[
L(\gamma) = f(\f(\gamma_{0})), 
\]
for some $f : \f(\mu_{0}) \to (0,\infty)$ such that $\f(\mu_{0}) \ni a \mapsto a - f^{2}/a$ is non increasing.
Then 
\[
\r_{t}(\gamma_{t})^{-1/N} \geq \r_{0}(\gamma_{0})^{-1/N} \tau_{K,N}^{(1-t)}(d(\gamma_{0},\gamma_{1})) 
+ \r_{1}(\gamma_{1})^{-1/N} \tau_{K,N}^{(s)}(d(\gamma_{0},\gamma_{1})),
\]
for every $t \in [0,1]$ and for $\gammA$-a.e. $\gamma \in G$.
\end{theorem}

\begin{proof}
From Remark \ref{R:assuverified},
\[
\r_{t}(\gamma_{t}) = \left( \int\r_{t}(z) \hat m_{a,t}(dz)\right) \frac{h_{a,t}(\gamma)}{\lambda_{t}(\gamma_{t})},
\]
where the integral is constant in $t$ and therefore in order to prove the claim we can assume 
\[
\r_{t}(\gamma_{t}) =  \frac{1}{\lambda_{t}(\gamma_{t})}h_{a,t}(\gamma).
\]
Then from Theorem \ref{T:surface} and Proposition \ref{P:1dim} 
\[ 
\begin{aligned}
\r^{-1/N}_{t}(\gamma_{t}) = &~ \Big( \frac{1}{\lambda_{t}(\gamma_{t})}h_{a,t}(\gamma) \Big)^{-1/N} \crcr
= &~ \Big((1-t)\lambda_{0}(\gamma_{0}) +t\lambda_{1}(\gamma_{1}) \Big)^{\frac{1}{N}} \Big( h_{a,t}^{-1/(N-1)}(\gamma) \Big)^{\frac{N-1}{N}}  \crcr
\geq &~ \Big( (1-t)\lambda_{0}(\gamma_{0})\Big)^{1/N} \Big(\sigma_{K,N-1}^{(1-t)}(d(\gamma_{0},\gamma_{1})) h^{-\frac{1}{N-1}}_{a, 0  }(\gamma)   \Big)^{\frac{N-1}{N}}  \crcr
+ &~ \Big( t \lambda_{1}(\gamma_{1})\Big)^{1/N} \Big(\sigma_{K,N-1}^{(t)}(d(\gamma_{0},\gamma_{1})) h^{-\frac{1}{N-1}}_{a, 1  }(\gamma)   \Big)^{\frac{N-1}{N}}  
\crcr
=&~ \r_{0}^{-1/N}(\gamma_{0})\tau_{K,N}^{(1-t)}(d(\gamma_{0},\gamma_{1}))  
+ \r_{1}^{-1/N}(\gamma_{1})\tau_{K,N}^{(t)}(d(\gamma_{0},\gamma_{1})). 
\end{aligned}
\]
The claim follows.
\end{proof}

\section{More on the one-dimensional component}\label{S:formal}

Assuming the metric measure space $(X,d,m)$ to be \emph{infinitesimally strictly convex}, see Subsection \ref{Ss:gradiff},
we can give an more explicit expression for $\lambda_{t}$.

Define the restriction map as follows. For any $t \in (0,1)$ let
$restr_{[t,1]} : \G(X) \to \G(X)$ be defined as follows $restr_{[t,1]}(\gamma)_{s} = \gamma_{(1-s) t + s}$.
Denote by $\gammA_{[t,1]}$ the measure $restr_{[t,1]\sharp}\gammA$.

\begin{lemma}
For all $t \in [0,1)$ the measure $\gammA_{[t,1]}$ represents $\nabla (1-t)(- \f_{t})$.
\end{lemma}

The notion of test plans representing gradients has been introduced in Definition \ref{D:representing}.

\begin{proof}
First observe that $\f_{t} \in S^{2}(e_{t}(G),d,m)$. Indeed from Proposition \ref{P:speed}, since $\f_{t}$ is a Kantorovich potential for $(\mu_{t},\mu_{1})$,
it follows that 
\[
|D\f_{t}|_{w} (\gamma_{t}) = \frac{d(\gamma_{t},\gamma_{1})}{1-t} = d(\gamma_{0},\gamma_{1}), \qquad \textrm{for } \gammA-a.e.\gamma, 
\]
and therefore $|D\f_{t}|_{w} \in L^{2}(e_{t}(G),m)$. 
We know that $\gammA_{[t,1]}$ is the optimal dynamical transference 
plan between $\mu_{t}$ and $\mu_{1}$ and $(1-t)\f_{t}$ is the Kantorovich potential for the $d^{2}$ cost, hence Proposition \ref{P:speed} implies that 
\[
\lim_{t \downarrow 0} \int \frac{\f_{t}(\gamma_{0}) -\f_{t}(\gamma_{\tau})}{\tau} \gammA_{[t,1]}(d\gamma) 
= \int \frac{d^{2}(\gamma_{0},\gamma_{1})}{1-t} \gammA_{[t,1]}(d\gamma). 
\]
Since $\| \gammA_{[t,1]}\|_{2}^{2}= \int d^{2}(\gamma_{0},\gamma_{1}) \gammA_{[t,1]}(d\gamma)$, 
\[
(1-t) \lim_{t \downarrow 0} \int \frac{\f_{t}(\gamma_{0}) -\f_{t}(\gamma_{\tau})}{\tau} \gammA_{[t,1]}(d\gamma) = \| \gammA_{[t,1]}\|_{2}^{2}
\]
and the claim follows.
\end{proof}

Using Theorem \ref{T:changingorder} and Theorem \ref{T:expression1} we can now write $\lambda_{t}$ in a differential expression.

\begin{proposition}\label{P:lambda}
Let $(X,d,m)$ be infinitesimally strictly convex.
Then $\lambda_{t}$ verifies the following identity: for every $t \in [0,1)$
\[ 
\frac{1}{\lambda_{t}(\gamma_{t})} = D\Phi_{t} (\nabla \f_{t})(\gamma_{t}),\qquad \gammA-a.e. \gamma,
\]
where the exceptional set depends on $t$.
\end{proposition}

\begin{proof}

Since $(X,d,m)$ is infinitesimally strictly convex, and $\gammA_{[t,1]}$ represents $\nabla(1-t)(- \f_{t})$, 
from Theorem \ref{T:changingorder} it follows that 
\begin{align*}
\lim_{\tau\downarrow 0} \int_{restr_{[t,1]}(G)} \frac{\Phi_{t}(\gamma_{0}) - \Phi_{t}(\gamma_{\tau})}{\tau} \gammA_{[t,1]}(d\gamma) 
= &~ (1-t) \int D\Phi_{t}(\nabla\f_{t})(x) \mu_{t}(dx)\crcr
= &~ (1-t) \int_{restr_{[t,1]}(G)} D\Phi_{t}(\nabla\f_{t})(\gamma_{0}) \gammA_{[t,1]}(d\gamma).
\end{align*}
Since the previous identity holds true even if we restrict to a subset of $restr_{[t,1]}(G)$, it follows that it holds point-wise:
for $\gammA_{[t,1]}$-a.e. $\gamma$
\[
\lim_{\tau\downarrow 0} \frac{\Phi_{t}(\gamma_{0}) - \Phi_{t}(\gamma_{\tau})}{\tau} = (1-t) D\Phi_{t}(\nabla\f_{t})(\gamma_{0}).
\]
So fix $\hat \gamma$ in the support of $\gammA_{[t,1]}$ such that the limit exists and  
consider $\gamma$ in the support of $\gammA$ such that $\hat \gamma_{\tau} = \gamma_{(1-\tau)t + \tau}$, 
then we have 
\[
\frac{\Phi_{t}(\hat \gamma_{0}) - \Phi_{t}(\hat \gamma_{\tau})}{\tau} = \frac{\Phi_{t}(\gamma_{t}) - \Phi_{t}(\gamma_{(1-\tau)t + \tau})}{\tau} = 
\frac{\Phi_{t}(\gamma_{t}) - \Phi_{t}(\gamma_{(1-\tau)t + \tau}) }{\tau(1-t)} (1-t),
\]
and therefore the claim follows from Theorem \ref{T:expression1}.
\end{proof}

Under the infinitesimally strictly convexity assumption, we have therefore the following decomposition: 
\[
\frac{1}{c(\f(\gamma_{0}))}  \r_{t} (\gamma_{t}) =    D\Phi_{t} (\nabla \f_{t})(\gamma_{t})    h_{a,t}(\gamma_{t}), 
\]
where $c(a) = \int \r_{t}(z) \hat m_{a,t}(dz)$ is independent of $t$, and $h$ verifies $\mathsf{CD}^{*}(K,N-1)$.

\subsection{A formal computation}
We conclude this note with a formal calculation in order to show a formal expression of 
$D\Phi_{t} (\nabla \f_{t})(\gamma_{t}) $ in a smooth framework.

So let us assume $X$ be the Euclidean space with distance 
given by the euclidean distance and $m$ any measure absolute continuous with 
respect to the Lebesgue measure of the right dimension.
Let $\mu_{t} = \r_{t} m$ be the usual geodesic in the $L^{2}$-Wasserstein space over $X$
and let $T_{t}, T_{t,1} : X \to X$ be optimal maps such that 
\[
(T_{t})_{\sharp}\mu_{0} = \mu_{t} \qquad (T_{t,1})_{\sharp}\mu_{t} = \mu_{1}.
\]
Hence 
\[
T_{t} = Id - t \nabla \f_{0}, \qquad  T_{t,1} = Id - (1-t)\nabla \f_{t},
\]
with $\f_{0}$ a Kantorovich potential associated to $\mu_{0},\mu_{1}$ and $\f_{t}$ the usual evolution at time $t$ of $\f_{0}$. 
Then the standard identity holds:
\[
\f_{t}(\gamma_{t}) = (1-t)\f_{0}(\gamma_{0}) + t \f_{1}(\gamma_{1}).
\]
Clearly $\gamma_{0} = T_{t}^{-1}(\gamma_{t})$ and $\gamma_{1} = T_{t,1}(\gamma_{t})$.
Then one can differentiate the standard identity in the direction $s \mapsto \gamma_{t+s}$.
Then we get
\[
\| \nabla \f_{t}\|^{2}(\gamma_{t}) = (1-t) \langle \nabla \f_{0}(\gamma_{0}), DT_{t}^{-1}(\gamma_{t}) \nabla\f_{t}(\gamma_{t}) \rangle + t \langle \nabla \f_{1}(\gamma_{1}), DT_{t,1}(\gamma_{t}) \nabla \f_{t}(\gamma_{t}) \rangle.
\]
Moreover one can write $\Phi_{t}$ in a more convenient way:
\[
\Phi_{t} = \f_{0} \circ T_{t}^{-1}
\]
and then compute $\lambda_{t}$ using Proposition \ref{P:lambda}
\begin{align*}
\frac{1}{\lambda_{t}(\gamma_{t})} = 
&~  \langle  (DT_{t}^{-1})^{t}(\gamma_{0}) \nabla \f_{0} (\gamma_{0}), \nabla \f_{t}(\gamma_{t}) \rangle \crcr
= &~  \langle   \nabla \f_{0} (\gamma_{0}), DT_{t}^{-1}(\gamma_{t}) \nabla \f_{t}(\gamma_{t}) \rangle 
\end{align*}
Then using what calculated before
\begin{align*}
\frac{1}{\lambda_{t}(\gamma_{t})} = &~  \frac{1}{1-t}\| \nabla \f_{t} (\gamma_{t})\|^{2}   
- \frac{t}{1-t}\langle \nabla \f_{1}(\gamma_{1}), DT_{t,1}(\gamma_{t}) \nabla \f_{t}(\gamma_{t}) \rangle \crcr
= &~ \| \nabla \f_{t} (\gamma_{t})\|^{2}  + t \langle H\f_{t}(\gamma_{t}) \nabla \f_{t}(\gamma_{t}),\nabla \f_{t}(\gamma_{t}) \rangle,  
\end{align*}
where $H\f_{t}$ is the Hessian of $\f_{t}$. 
Clearly the effect of curvature would change the expression of $DT_{t,1}$. 
Hence on a linear space 
\[
\frac{1}{\lambda_{t}(\gamma_{t})} =  \langle ( Id + t H  \f_{t}(\gamma_{t})) \nabla \f_{t}(\gamma_{t}),\nabla \f_{t}(\gamma_{t}) \rangle.
\]
As a final comment, 
by Corollary \ref{C:1d-evo}, it holds that $\f^{t} = - \f_{t}^{c}$ and since 
\[
\f_{t}^{c} (x) = H^{t}_{0}(-\f) = \inf_{y \in X} \frac{1}{2t}d^{2}(x,y) - \f(y),
\]
it follows by semi-concavity that $Id - tH \f^{c}_{t} \geq 0$, in the sense of symmetric matrices. 
Note that we have derived in a different way the same expression for $\lambda_{t}$ obtained in \eqref{E:equivalent} 
from the decomposition of the differential of optimal transport map on manifold of \cite{corderomccann:brescamp}.
Again from \cite{corderomccann:brescamp} it follows that 
\[
Id - tH \f^{c}_{t} > 0,
\]
showing again consistency with Assumption \ref{A:Phi}.

\end{document}